\documentclass{daj}

\usepackage{epsf,epsfig,amsfonts,amsgen,amsmath,amstext,amsbsy,amsopn,amsthm,cases,listings,color
}
\usepackage{ebezier,eepic}
\usepackage{color}
\usepackage{multirow}
\usepackage{epstopdf}
\usepackage{graphicx}
\usepackage{pgf,tikz}
\usepackage{mathrsfs}
\usepackage[marginal]{footmisc}
\usepackage{enumitem}
\usepackage{booktabs}
\usepackage{url}
\usepackage{mathtools}
\usepackage{pgfplots}
\usepackage{amssymb}
\usepackage{subfigure}
\usepackage{mathrsfs}
\usetikzlibrary{arrows}
\definecolor{uuuuuu}{rgb}{0.27,0.27,0.27}
\definecolor{sqsqsq}{rgb}{0.1255,0.1255,0.1255}
\newtheorem{definition}{Definition} [section]
\newtheorem{theorem}[definition]{Theorem}
\newtheorem{lemma}[definition]{Lemma}
\newtheorem{proposition}[definition]{Proposition}

\newtheorem{corollary}[definition]{Corollary}
\newtheorem{conjecture}[definition]{Conjecture}
\newtheorem{claim}[definition]{Claim}

\newtheorem{observation}[definition]{Observation}

\pgfplotsset{compat=1.18}
\dajAUTHORdetails{%
  title = {Hypergraphs with Infinitely Many Extremal Constructions}, %% please capitalize all significant words
  author = {Jianfeng Hou, Heng Li, Xizhi Liu, Dhruv Mubayi, and Yixiao Zhang},
  plaintextauthor = {Jianfeng Hou, Heng Li, Xizhi Liu, Dhruv Mubayi, and Yixiao Zhang},
  plaintexttitle = {Hypergraphs with Infinitely Many Extremal Constructions}, %%  title without math or LaTeX
  runningtitle = {Infinitely Many Extremal Constructions},
  runningauthor = {J. Hou, H. Li, X. Liu, D. Mubayi, and Y. Zhang},
  copyrightauthor = {J. Hou, H. Li, X. Liu, D. Mubayi, and Y. Zhang},
  keywords = {Hypergraph Tur{\' a}n problem, stability, nonminimal hypergraphs, Lagrangian, multilinear polynomials, shadow, feasible region.},
}   %%% END \dajAUTHORdetails

\dajEDITORdetails{%
   year={2023},
   number={18},
   received={12 July 2022},   % received date: example: 7 January 2017
   published={4 December 2023},  % published date
   doi={10.19086/da.88508},       % XXX = number of paper, e.g. da006 for paper#6
}   %%% END \dajEDITORdetails

\begin{document}

\begin{frontmatter}[classification=text]
%% EDITOR: this will force the keywords to appear right after the Abstract.
%%   If the abstract is too long and would force the keywords off the
%%   front page, please comment out % [classification=text] above
%%   This way the keywords will be floated on the bottom of the first page
%%   even though the Abstract spills over to the next page.

%%% AUTHOR: Title goes here.  This line is optional.  You must use it
%%   if title has footnote attached or requires nontrivial typesetting,
%%   e.g., inclusion of linebreaks to force nice layout.
\title{Hypergraphs with Infinitely Many Extremal Constructions} %% please capitalize all significant words

%%% AUTHOR:
%%% List all authors. If you wish, place grant acknowledgements in \thanks.
%%% In brackets include a short tag for each author.
\author[JH]{Jianfeng Hou\thanks{Supported by National Natural Science Foundation of China (Grant No. 12071077).}}
\author[HL]{Heng Li}
\author[XL]{Xizhi Liu\thanks{Research was supported by ERC Advanced Grant 101020255.}}
\author[DM]{Dhruv Mubayi\thanks{Research was supported by NSF awards DMS-1763317 and DMS-1952767.}}
\author[YZ]{Yixiao Zhang}

%%% AUTHOR: Abstract goes here
\begin{abstract}
We give the first exact and stability
results for a hypergraph Tur\'{a}n problem with infinitely many  extremal constructions
that are far from each other in edit-distance.
This includes an example of triple systems with Tur\'{a}n density $2/9$,
thus answering some questions posed by the third and fourth authors and Reiher about the feasible region of hypergraphs. Our results also provide extremal constructions whose shadow density is a transcendental number.

Our novel approach is to construct certain multilinear polynomials that attain
their maximum (in the standard simplex) on a line segment and then to use these polynomials to define an operation on hypergraphs that  gives extremal constructions.
\end{abstract}
\end{frontmatter}

%%% AUTHOR: body of paper starts here
\section{Introduction}\label{SEC:Introduction}
%%%%%%%%%%%%%%%%%%%%%%%%%%%%%%%%%%%%%%%%%%%%%%%%%
\subsection{Tur\'{a}n number and stability}\label{SUBSEC:turan-number-and-stability}
%Let $r\ge 2$ and $\mathcal{H}$ be a $r$-uniform graph (or $r$-graph for short).
%We use $\nu(\mathcal{H})$ and $|\mathcal{H}|$ denote the number of vertices and edges of $\mathcal{H}$.
For $r\ge 2$ an $r$-uniform hypergraph (henceforth $r$-graph) $\mathcal{H}$ is a collection of $r$-subsets of some finite set $V$.
Given a family $\mathcal{F}$ of $r$-graphs we say $\mathcal{H}$ is $\mathcal{F}$-free
if it does not contain any member of $\mathcal{F}$ as a subgraph.
The {\em Tur\'{a}n number} $\mathrm{ex}(n,\mathcal{F})$ of $\mathcal{F}$ is the maximum
number of edges in an $\mathcal{F}$-free $r$-graph on $n$ vertices.
The {\em Tur\'{a}n density} $\pi(\mathcal{F} )$ of $\mathcal{F}$ is defined as
$\pi(\mathcal{F})=\lim_{n\to \infty}\mathrm{ex}(n,\mathcal{F})/{\binom{n}{r}}$.
A family $\mathcal{F}$ is called \emph{nondegenerate} if $\pi(\mathcal{F}) >0$.
The study of $\mathrm{ex}(n,\mathcal{F})$ is perhaps the central topic in extremal graph and hypergraph theory.

Much is known about $\mathrm{ex}(n,\mathcal{F})$ when $r=2$,
and one of the most famous results in this regard is Tur\'{a}n's theorem,
which states that for every integer $\ell \ge 2$ the Tur\'{a}n number $\mathrm{ex}(n,K_{\ell+1})$
is uniquely achieved by the balanced $\ell$-partite graph on $n$ vertices,
which is called the Tur\'{a}n graph $T(n,\ell)$.

For $r\ge 3$ determining $\pi(\mathcal{F})$ for a family $\mathcal{F}$ of $r$-graphs
is known to be notoriously hard in general.
Indeed, the problem of determining $\pi(K_{\ell}^{r})$ raised by Tur\'{a}n~\cite{TU41},
where $K_{\ell}^{r}$ is the complete $r$-graph on $\ell$ vertices, is still wide open for all $\ell>r\ge 3$.
Erd\H{o}s offered $\$ 500$ for the determination of any $\pi(K_{\ell}^{r})$
with $\ell > r \ge 3$ and $\$ 1000$ for the determination of all $\pi(K_{\ell}^{r})$ with $\ell > r \ge 3$.

\begin{conjecture}[Tur\'{a}n~\cite{TU41}]\label{CONJ:Turan-complete-r-graph}
For every integer $\ell\ge 3$ we have $\pi(K_{\ell+1}^{3}) = 1- {4}/{\ell^2}$.
\end{conjecture}

The case $\ell = 3$ above, which states that $\pi(K_{4}^{3}) = 5/9$ has generated a lot of interest and activity over the years.
Many constructions (e.g. see \cite{BR83,KO82,FO88}) are known to achieve the value in Conjecture~\ref{CONJ:Turan-complete-r-graph}
for $\ell = 3$.
In particular, Kostochka~\cite{KO82} showed that there are at least $2^{n-2}$ nonisomorphic
extremal $K_{4}^{3}$-free constructions on $3n$ vertices (assuming Tur\'{a}n's Tetrahedron conjecture is true).
This is perhaps one of the reasons why it is so challenging.
On the other hand, successively better upper bounds for $\pi(K_{4}^{3})$ were obtained by
de Caen $\cite{DC88}$, Giraud (see $\cite{CL99}$), Chung and Lu  $\cite{CL99}$, and Razborov $\cite{RA10}$.
The current record is $\pi(K_{4}^{3}) \le 0.561666$, which was obtained by Razborov  $\cite{RA10}$ using the Flag Algebra machinery.

In this work, we continue the research on hypergraph Tur\'{a}n problems with the flavour
of trying to show that certain phenomena impossible for graphs are possible for $r$-graphs when $r\ge 3$.
There are several results that belong to this category.
For instance, the well known work of Frankl and R\"odl on jumps~\cite{FR84},
and Pikhurko's theorem that for $r\ge 3$ there exist uncountably many Tur\'an densities of infinite families of $r$-graphs~\cite{PI14}.
More specifically, we will show that there are infinitely many finite families $\mathcal{F}$ with similar properties
as $K_{4}^{3}$ (assuming Tur\'an's conjecture) in terms of the number of extremal constructions.

Many families $\mathcal{F}$ have the property that there is a unique $\mathcal{F}$-free hypergraph
$\mathcal{G}$ on $n$ vertices achieving $\mathrm{ex}(n,\mathcal{F})$,
and moreover, any $\mathcal{F}$-free hypergraph $\mathcal{H}$ of size close to $\mathrm{ex}(n,\mathcal{F})$
can be transformed to $\mathcal{G}$ by deleting and adding very few edges.
Such a property is called the stability of $\mathcal{F}$.
The first stability theorem, which says that $K_{\ell}$ is stable for all integers $\ell\ge 3$,
was proved independently by Erd\H{o}s and Simonovits $\cite{SI68}$.
Their result motivated the third author~\cite{MU07} to make the following definition.

\begin{definition}[$t$-stable]\label{DFN:t-stable}
Let $r \ge 2$ and $t \ge 1$ be integers.
A family $\mathcal{F}$ of $r$-graphs is $t$-stable if for every $m\in\mathbb{N}$
there exist $r$-graphs $\mathcal{G}_{1}(m),\ldots,\mathcal{G}_{t}(m)$ on $m$ vertices such that the following holds.
For every $\delta >0$ there exist $\epsilon > 0$ and $N_0$ such that for all $n \ge N_0$
if $\mathcal{H}$ is an $\mathcal{F}$-free $r$-graph on $n$ vertices with $|\mathcal{H}| > (1-\epsilon) \mathrm{ex}(n,\mathcal{F})$
%\begin{align}
%|\mathcal{H}| > (1-\epsilon) \mathrm{ex}(n,\mathcal{F}), \notag
%\end{align}
then $\mathcal{H}$ can be transformed to some $\mathcal{G}_{i}(n)$ by adding and removing at most $\delta n^r$ edges.
Say $\mathcal{F}$ is stable if it is $1$-stable.
\end{definition}

Denote by $\xi(\mathcal{F})$ the minimum integer $t$ such that $\mathcal{F}$ is $t$-stable,
and set $\xi(\mathcal{F}) = \infty$ if there is no such $t$.
Call $\xi(\mathcal{F})$ \emph{the stability number} of $\mathcal{F}$.

The classical Erd\H{o}s--Stone--Simonovits theorem~\cite{ES46,ES66} and Erd\H{o}s--Simonovits stability theorem~\cite{SI68}
imply that every nondegenerate family of graphs is stable.
For hypergraphs there are many families (whose Tur\'{a}n densities are unknown)
which are conjecturally not stable.
Two famous examples are Tur\'{a}n's conjecture on tetrahedra (i.e. $K_{4}^3$) mentioned above (e.g. see~\cite{AS95,KO82,LM22}),
and the Erd\H{o}s--S\'{o}s conjecture on triple systems with bipartite links (e.g. see~\cite{FF84,LM22}).
In fact, families that are not stable and whose Tur\'{a}n densities can be determined were constructed only very recently.
In~\cite{LM22}, the third and fourth authors constructed the first finite $2$-stable family of $3$-graphs.
Later in~\cite{LMR1}, together with Reiher, they further constructed the first finite $t$-stable family of $3$-graphs
for every integer $t\ge 3$.
In~\cite{LMR3}, using a similar method that was used for the hypergraph jump problem (see e.g.~\cite{FPRT07})
their constructions were extended to every integer $r \ge 4$.

Using Kostochka's constructions the authors proved in~\cite{LM22} that $\xi(K_{4}^3) = \infty$ (assuming Tur\'{a}n's conjecture).
However, the methods used in~\cite{LM22} and~\cite{LMR1}
for constructing the $2$-stable family and $t$-stable families cannot be extended further to
construct a finite family whose stability number is infinite, and
this was left as an open question in both papers.

Our main result in this work is to give a novel method to construct (finite) families $\mathcal{F}$ with $\xi(\mathcal{F}) = \infty$.
Our method is based on a simple property of multilinear polynomials,
which will be stated in Section~\ref{PROP:multilinear-polynomial}.
We will use this method to give infinitely many (finite) families of $3$-graphs
whose number of extremal constructions grows with the number of vertices.
Further, using the machinery provided in~\cite{LMR2} we also prove an Andr\'{a}sfai--Erd\H{o}s--S\'{o}s type~\cite{AES74}
stability theorem for these families.

We identify an $r$-graph $\mathcal{H}$ with its edge set,
use $V(\mathcal{H})$ to denote its vertex set, and denote by $v(\mathcal{H})$ the size of $V(\mathcal{H})$.
An $r$-graph $\mathcal{H}$ is a \emph{blowup} of an $r$-graph $\mathcal{G}$ if there exists a
map $\psi\colon V(\mathcal{H}) \to V(\mathcal{G})$ so that $\psi(E) \in \mathcal{G}$ iff $E\in \mathcal{H}$.
We say $\mathcal{H}$ is \emph{$\mathcal{G}$-colorable}
if there exists a map $\phi\colon V(\mathcal{H}) \to V(\mathcal{G})$ so that $\phi(E)\in \mathcal{G}$ for all $E\in \mathcal{H}$,
and we call such a map $\phi$ a \emph{homomorphism} from $\mathcal{H}$ to $\mathcal{G}$.
In other words, $\mathcal{H}$ is $\mathcal{G}$-colorable if and only if $\mathcal{H}$ occurs as a subgraph in some blowup
of $\mathcal{G}$.

\begin{theorem}\label{THM:main-sec1.1}
For every integer $t\ge 3$ there exists a finite family $\mathcal{F}_t$ of $3$-graphs such that the following statements hold.
\begin{enumerate}[label=(\alph*)]
\item We have $\mathrm{ex}(n,\mathcal{F}_t) \le \frac{(t-2)(t-1)}{6t^2}n^3$ for all $n\in \mathbb{N}$,
        and equality holds whenever $t \mid n$.
\item\label{it:2} If $t \mid n$, then the number of nonisomorphic maximum $\mathcal{F}_t$-free $3$-graphs on $n$ vertices is
        at least ${n}/{2t}$.
\item We have $\xi(\mathcal{F}_t) = \infty$.
\item For every integer $t\ge 4$ there exist constants $\epsilon = \epsilon(t)>0$
        and $N_0 = N_0(t)$ such that the following holds for every integer $n\ge N_0$.
        Every $n$-vertex $\mathcal{F}_t$-free $3$-graph with minimum degree at least $\left(\frac{(t-2)(t-1)}{2t^2}-\epsilon\right)n^2$
        is $\Gamma_{t}$-colorable, where $\Gamma_{t}$ is some fixed $3$-graph on $t+2$ vertices.
\end{enumerate}
\end{theorem}
\noindent\textbf{Remarks.}
\begin{itemize}
\item Part~\ref{it:2} says that there are many nonisomorphic extremal $\mathcal{F}_t$-free $3$-graphs on $n$ vertices for $t \mid n$.
This phenomenon already showed up in Graph theory.
For example, it follows from an old result of Simonovits~\cite[Theorem~5]{SI68} that a maximum $K_{1,3,3}$-free graph on $n$ vertices
can be obtained from an $n$-vertex balanced bipartite graph by adding graphs of maximum degree $2$ into both parts.
It is clear that there are many nonisomorphic maximum graphs with maximum degree $2$ on a fixed number of vertices,
so there are many nonisomorphic maximum $K_{1,3,3}$-free graph on $n$ vertices.

For hypergraphs, Balogh and Clemen\footnote{Personal communication.} claim that there exists a single $3$-graph $F$
such that for special integers $n$
the maximum $F$-free $3$-graphs on $n$ vertices are $3$-graphs that can be obtained from the balanced $n$-vertex $3$-partite $3$-graph
by adding Steiner triple systems into all three parts.
Thus it follows from a result of Keevash~\cite{Ke18} that there are exponentially many such maximum $F$-free $3$-graphs on $n$ vertices.

These examples are different from our result in the sense that they all have stability number $1$, i.e. all extremal constructions have a small edit-distance to each other.

\item Motivated by the hypergraph jump problem\footnote{Our results in this paper are not very related to the Jump problem, so we omit the related definitions here.
We refer the interested reader to e.g. \cite{FR84,FPRT07,BT11} for related results.}, the authors asked in~\cite{LMR1}
whether there exists a finite family whose Tur\'{a}n density is $2/9$ that is not stable?
By letting $t = 3$ in Theorem~\ref{THM:main-sec1.1} we see that the finite family $\mathcal{F}_3$ of $3$-graphs
has Tur\'{a}n density $2/9$ but $\xi(\mathcal{F}_3) = \infty$, thus answering this question.
\end{itemize}

%%%%%%%%%%%%%%%%%%%%%%%%%%%%%%%%%%%%%%%%%%%%%%%%%
\subsection{Multilinear polynomials}\label{SEC:Multilinear-polynomials}
In this short section, we state a simple result about polynomials. Later we will see that our main theorem reduces to this result.

Denote by $\Delta_{m-1}$ the standard $(m-1)$-dimensional simplex,
i.e.
\begin{align}
\Delta_{m-1} = \left\{(x_1,\ldots,x_m)\in [0,1]^m\colon x_1+\cdots+x_m= 1\right\}. \notag
\end{align}
Given an $m$-variable continuous function $f$ we define
\begin{align}
\lambda(f) = \max\left\{f(x_1,\ldots,x_m)\colon (x_1,\ldots,x_m)\in \Delta_{m-1}\right\},  \notag
\end{align}
and
\begin{align}
Z(f) = \left\{(x_1,\ldots,x_m)\in \Delta_{m-1}\colon f(x_1,\ldots,x_m)-\lambda(f) = 0\right\}.  \notag
\end{align}
Since $\Delta_{m-1}$ is compact, the restriction of $f$ to $\Delta_{m-1}$ attains a maximum value. Thus $\lambda(f)$ and $Z(f)$ are well-defined.

Let $p(X_1,\ldots,X_m)$ be a polynomial, where $X_1,\ldots,X_m$ are indeterminates.
We say $p$ is \emph{multilinear} if for every $\alpha \in \mathbb{R}$ and for every $i\in [m]$
we have $p(X_1,\ldots,\alpha \cdot X_i ,\ldots,X_m) = \alpha \cdot p(X_1,\ldots, X_i ,\ldots,X_m)$. Note that this is equivalent to the assertion that each term of $p$ is of the form $\prod_{i \in S} x_i$ for some $S$.
We say $p$ is \emph{nonnegative} (or \emph{nonpositive}) if $p(x_1,\ldots,x_m)\ge 0$ (or $p(x_1,\ldots,x_m)\le 0$)
for all $(x_1,\ldots,x_m)\in \Delta_{m-1}$.
For a pair $\{i,j\}\subset [m]$ we say $p$ is symmetric with respect to $X_i$ and $X_j$ if
$$p(X_1,\ldots,X_i,\ldots, X_j, \ldots, X_m)=p(X_1,\ldots,X_j,\ldots, X_i, \ldots, X_m).$$

An easy observation is that a multilinear polynomial $p(X_1,\ldots,X_m)$ is symmetric with respect to $X_i$ and $X_j$
iff there exist polynomials $p_1, p_2, p_3$ without indeterminates $X_i$ and $X_j$ such that
$p = p_1 + p_2 (X_i + X_j) + p_3 X_iX_j$.

Given two vectors $\vec{x}, \vec{y}\in \mathbb{R}^m$ define the line segment $L(\vec{x}, \vec{y})$
with endpoints $\vec{x}$ and $\vec{y}$ as
\begin{align}
L(\vec{x}, \vec{y}) = \left\{\alpha\cdot \vec{x}+(1-\alpha)\cdot \vec{y}\colon \alpha \in [0,1]\right\}. \notag
\end{align}

We will later prove the following result about multilinear polynomials.

\begin{proposition}\label{PROP:multilinear-polynomial}
Let $p(X_1,\ldots,X_m) = p_1 + p_2 (X_i + X_j) + p_3 X_iX_j$
be an $m$-variable multilinear polynomial that is symmetric with respect to $X_i$ and $X_j$.
Suppose that $p_3$ is nonnegative,
and $p_4, p_5$ are nonnegative polynomials satisfying $p_4+p_5=p_3$.
Then the $(m+2)$-variable polynomial
\begin{align}
& \hat{p}(X_1,\ldots, X_i, X_i', \ldots, X_j, X_j', \ldots,X_m)  \notag\\
&  = p_1 + p_2(X_i+X'_i+X_j+X'_j)
            + p_{4}(X_i+X'_i)(X_j+X'_j)
            + p_{5}(X_i+X_j)(X'_i+X'_j) \notag
\end{align}
satisfies $\lambda(\hat{p}) = \lambda(p)$, and moreover,
for every $(x_1,\ldots, x_m)\in Z(p)$ we have $L(\vec{y}, \vec{z}) \subset Z(\hat{p})$,
where $\vec{y}, \vec{z}\in \Delta_{m+1}$ are defined by
\begin{align}
\vec{y}
&  = \left(x_1,\ldots,x_{i-1},(x_i+x_j)/2, 0,x_{i+1}, \ldots,x_{j-1}, 0,(x_i+x_j)/2, x_{j+1}, \ldots,  x_m\right), \notag\\
\vec{z}
&  = \left(x_1,\ldots,x_{i-1},0, (x_i+x_j)/2,x_{i+1}, \ldots,x_{j-1},(x_i+x_j)/2, 0,x_{j+1}, \ldots,  x_m\right). \notag
\end{align}
\end{proposition}

%%%%%%%%%%%%%%%%%%%%%%%%%%%%%%%%%%%%%%%%%%%%%%%%%
\subsection{Nonminimal hypergraphs}\label{SUBSEC:nonminimal-hypergraphs}
Now we return to hypergraphs.
Our constructions are closely related to the so-called nonminimal hypergraphs.
To define them properly we need some definitions related to the Lagrangian of a hypergraph,
which was introduced by Frankl and R\"odl in~\cite{FR84}.

For an $r$-graph $\mathcal{G}$ on $m$ vertices
(let us assume for notational transparency that $V(\mathcal{G}) = [m]$) the multilinear polynomial
$p_{\mathcal{G}}$ is defined by
\begin{align}
p_{\mathcal{G}}(X_1,\ldots,X_m) = \sum_{E\in \mathcal{G}}\prod_{i\in E}X_i. \notag
\end{align}
The \emph{Lagrangian} of $\mathcal{G}$ is defined by $\lambda(\mathcal{G}) = \lambda(p_{\mathcal{G}})$.
Define
\begin{align}
Z(\mathcal{G}) = Z(p_{\mathcal{G}})
= \left\{(x_1,\ldots,x_n)\in \Delta_{m-1}\colon p_{\mathcal{G}}(x_1,\ldots,x_m) = \lambda(\mathcal{G})\right\}. \notag
\end{align}

For a vertex set $S\subset V(\mathcal{G})$ we use
$\mathcal{G}-S$ to denote the induced subgraph of $\mathcal{G}$ on $V(\mathcal{G})\setminus S$.

\begin{definition}
An $r$-graph $\mathcal{G}$ is minimal if $\lambda(\mathcal{G}-v)< \lambda(\mathcal{G})$
for every vertex $v\in \mathcal{G}$. Otherwise, it is nonminimal.
\end{definition}

A substantial amount of research on hypergraph Tur\'{a}n problems is about minimal hypergraphs.
Indeed,  for every nondegenerate family $\mathcal{F}$ of hypergraphs that was studied before this work,
the extremal $\mathcal{F}$-free constructions are either minimal or blowups of minimal hypergraphs or close to them in edit-distance.
Conversely, in~\cite{PI14}, Pikhurko proved that for every minimal hypergraph $\mathcal{G}$
there exists a finite family $\mathcal{F}$ whose unique extremal construction is a blowup of $\mathcal{G}$
\footnote{Pikhurko's results are much more general than what we stated here,
and we refer the reader to~\cite{PI14} for more details.}.
In contrast, there was very little literature about nonminimal hypergraphs.

Parallel to Pikhurko's result about minimal hypergraphs we have the following result about nonminimal hypergraphs.

\begin{proposition}\label{PROP:nonminal-finite-family}
Let $r\ge 3$ be an integer.
For every nonminimal $r$-graph $\mathcal{G}$ there exists a finite family $\mathcal{F}$ of $r$-graphs
such that for every integer $n\in\mathbb{N}$
there exists a maximum $\mathcal{F}$-free $r$-graph on $n$-vertices which is a blowup of $\mathcal{G}$.
\end{proposition}

Next, we introduce an operation motivated by Proposition~\ref{PROP:multilinear-polynomial}.
It shows  that for a certain class of $3$-graphs $\mathcal{G}$
we can produce a nonminimal $3$-graph containing $\mathcal{G}$.

For a $3$-graph $\mathcal{G}$ and a pair of vertices $\{u,v\}\subset V(\mathcal{G})$
the \emph{neighborhood} $N_{\mathcal{G}}(u,v)$ and \emph{codegree} $d_{\mathcal{G}}(u,v)$ of $\{u,v\}$ are defined as
$$N_{\mathcal{G}}(u,v) = \{w\in\mathcal{G}\colon \{u,v,w\}\in\mathcal{G}\}
\quad\text{and}\quad d_{\mathcal{G}}(u,v) = |N_{\mathcal{G}}(u,v)|.$$
We say $\mathcal{G}$ is \emph{$2$-covered} if $d_{\mathcal{G}}(u,v) \ge 1$ for every pair $\{u,v\}\subset V(\mathcal{G})$. For a vertex $u \in V(\mathcal{G})$
the \emph{link} of $v$ is defined as
$$L_{\mathcal{G}}(u) = \{\{v,w\}\subset\mathcal{G}\colon \{u,v,w\}\in\mathcal{G}\}.$$ Vertices $u$ and $u'$ in $\mathcal{G}$ are clones if
$L_{\mathcal{G}}(u)=L_{\mathcal{G}}(v)$.

\begin{definition}[Crossed blowup]
Let $\mathcal{G}$ be a $3$-graph and
$\{v_1,v_2\}\subset \mathcal{G}$ be a pair of vertices with $d(v_1,v_2) = k \ge 2$.
Fix an ordering of the vertices in $N_{\mathcal{G}}(v_1,v_2)$, say $N_{\mathcal{G}}(v_1,v_2) = \{u_1,\ldots, u_k\}$.
The crossed blowup $\mathcal{G}\boxplus\{v_1,v_2\}$ of $\mathcal{G}$ on $\{v_1,v_2\}$ is defined as follows.
\begin{enumerate}[label=(\alph*)]
\item Remove all edges containing the pair $\{v_1,v_2\}$ from $\mathcal{G}$,
\item add two new vertices $v_1'$ and $v_2'$, make $v_1'$ a clone of $v_1$ and $v_2'$ a clone of $v_2$,
\item for every $i\in [k-1]$ add the edge set $\{u_iv_1v_1',u_iv_1v_2',u_iv_2v_1',u_iv_2v_2'\}$,
                and for $i = k$ add the edge set $\{u_kv_1v_2,u_kv_1v_2',u_kv_1'v_2,u_kv_1'v_2'\}$.
\end{enumerate}
\end{definition}
\noindent\textbf{Remark.}
If we replace $\{u_kv_1v_2,u_kv_1v_2',u_kv_1'v_2,u_kv_1'v_2'\}$
by $\{u_kv_1v_1',u_kv_1v_2',u_kv_2v_1',u_kv_2v_2'\}$ in Step (c),
then we obtain an ordinary blowup of $\mathcal{G}$.

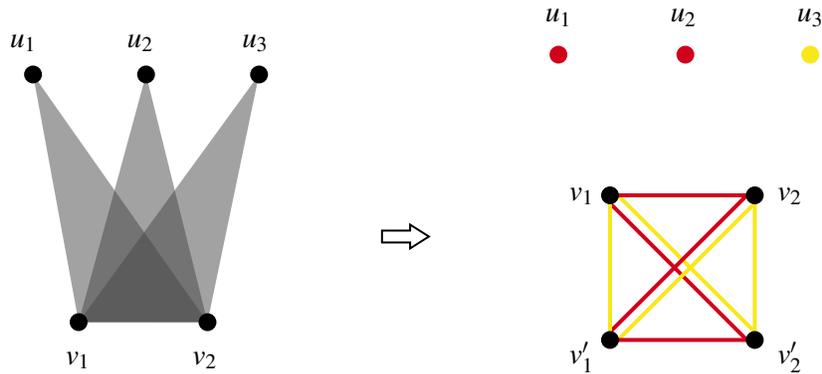
\begin{figure}[htbp]
\centering
\tikzset{every picture/.style={line width=0.75pt}} %set default line width to 0.75pt
\begin{tikzpicture}[x=0.75pt,y=0.75pt,yscale=-1,xscale=1]
%uncomment if require: \path (0,472); %set diagram left start at 0, and has height of 472
%Straight Lines [id:da2707940059619163]
\draw [draw opacity=0][fill=uuuuuu,fill opacity=0.5 ]   (164,341.57) -- (229,341.6) -- (255,216.57) -- cycle ;
%Straight Lines [id:da0769522816292596]
\draw [draw opacity=0][fill=uuuuuu,fill opacity=0.5 ]   (164,341.57) -- (229,341.6) -- (198,216.57) -- cycle ;
%Straight Lines [id:da020730320080502684]
\draw [draw opacity=0][fill=uuuuuu,fill opacity=0.5 ]   (164,341.57) -- (229,341.6) -- (141,216.57) -- cycle ;
%Shape: Circle [id:dp7839333893737988]
\draw  [color={rgb, 255:red, 0; green, 0; blue, 0 }  ,draw opacity=1 ][fill={rgb, 255:red, 0; green, 0; blue, 0 }  ,fill opacity=1 ] (137,216.57) .. controls (137,214.36) and (138.79,212.57) .. (141,212.57) .. controls (143.21,212.57) and (145,214.36) .. (145,216.57) .. controls (145,218.78) and (143.21,220.57) .. (141,220.57) .. controls (138.79,220.57) and (137,218.78) .. (137,216.57) -- cycle ;
%Shape: Circle [id:dp5077108863096422]
\draw  [color={rgb, 255:red, 0; green, 0; blue, 0 }  ,draw opacity=1 ][fill={rgb, 255:red, 0; green, 0; blue, 0 }  ,fill opacity=1 ] (194,216.57) .. controls (194,214.36) and (195.79,212.57) .. (198,212.57) .. controls (200.21,212.57) and (202,214.36) .. (202,216.57) .. controls (202,218.78) and (200.21,220.57) .. (198,220.57) .. controls (195.79,220.57) and (194,218.78) .. (194,216.57) -- cycle ;
%Shape: Circle [id:dp9193981178520048]
\draw  [color={rgb, 255:red, 0; green, 0; blue, 0 }  ,draw opacity=1 ][fill={rgb, 255:red, 0; green, 0; blue, 0 }  ,fill opacity=1 ] (251,216.57) .. controls (251,214.36) and (252.79,212.57) .. (255,212.57) .. controls (257.21,212.57) and (259,214.36) .. (259,216.57) .. controls (259,218.78) and (257.21,220.57) .. (255,220.57) .. controls (252.79,220.57) and (251,218.78) .. (251,216.57) -- cycle ;
%Shape: Circle [id:dp22815968333839765]
\draw  [color={rgb, 255:red, 208; green, 2; blue, 27 }  ,draw opacity=1 ][fill={rgb, 255:red, 208; green, 2; blue, 27 }  ,fill opacity=1 ] (402,206.57) .. controls (402,204.36) and (403.79,202.57) .. (406,202.57) .. controls (408.21,202.57) and (410,204.36) .. (410,206.57) .. controls (410,208.78) and (408.21,210.57) .. (406,210.57) .. controls (403.79,210.57) and (402,208.78) .. (402,206.57) -- cycle ;
%Shape: Circle [id:dp39760586019524835]
\draw  [color={rgb, 255:red, 208; green, 2; blue, 27 }  ,draw opacity=1 ][fill={rgb, 255:red, 208; green, 2; blue, 27 }  ,fill opacity=1 ]   (466,206.57) .. controls (466,204.36) and (467.79,202.57) .. (470,202.57) .. controls (472.21,202.57) and (474,204.36) .. (474,206.57) .. controls (474,208.78) and (472.21,210.57) .. (470,210.57) .. controls (467.79,210.57) and (466,208.78) .. (466,206.57) -- cycle ;
%Shape: Circle [id:dp9238639134326649]
\draw [color={rgb, 255:red, 248; green, 231; blue, 28 }  ,draw opacity=1 ][fill={rgb, 255:red, 248; green, 231; blue, 28 }  ,fill opacity=1 ] (529,206.57) .. controls (529,204.36) and (530.79,202.57) .. (533,202.57) .. controls (535.21,202.57) and (537,204.36) .. (537,206.57) .. controls (537,208.78) and (535.21,210.57) .. (533,210.57) .. controls (530.79,210.57) and (529,208.78) .. (529,206.57) -- cycle ;
%Shape: Circle [id:dp6045425835919433]
\draw  [color={rgb, 255:red, 0; green, 0; blue, 0 }  ,draw opacity=1 ][fill={rgb, 255:red, 0; green, 0; blue, 0 }  ,fill opacity=1 ] (160,341.57) .. controls (160,339.36) and (161.79,337.57) .. (164,337.57) .. controls (166.21,337.57) and (168,339.36) .. (168,341.57) .. controls (168,343.78) and (166.21,345.57) .. (164,345.57) .. controls (161.79,345.57) and (160,343.78) .. (160,341.57) -- cycle ;
%Shape: Circle [id:dp011595876435121832]
\draw  [color={rgb, 255:red, 0; green, 0; blue, 0 }  ,draw opacity=1 ][fill={rgb, 255:red, 0; green, 0; blue, 0 }  ,fill opacity=1 ] (225,341.6) .. controls (225,339.39) and (226.79,337.6) .. (229,337.6) .. controls (231.21,337.6) and (233,339.39) .. (233,341.6) .. controls (233,343.81) and (231.21,345.6) .. (229,345.6) .. controls (226.79,345.6) and (225,343.81) .. (225,341.6) -- cycle ;
%Straight Lines [id:da8737091986006313]
\draw [color={rgb, 255:red, 208; green, 2; blue, 27 }  ,draw opacity=1 ][line width=1.5]    (432,281.57) -- (501,350.57) ;
%Straight Lines [id:da04529129448829061]
\draw [color={rgb, 255:red, 248; green, 231; blue, 28 }  ,draw opacity=1 ][line width=1.5]    (436,277.57) -- (505,346.57) ;
%Straight Lines [id:da4907355976190615]
\draw [color={rgb, 255:red, 208; green, 2; blue, 27 }  ,draw opacity=1 ][line width=1.5]    (432,346.57) -- (501,277.57) ;
%Straight Lines [id:da8473852837090858]
\draw [color={rgb, 255:red, 248; green, 231; blue, 28 }  ,draw opacity=1 ][line width=1.5]    (436,350.57) -- (505,281.57) ;
%Straight Lines [id:da7190055283068713]
\draw [color={rgb, 255:red, 208; green, 2; blue, 27 }  ,draw opacity=1 ][line width=1.5]    (432,277.57) -- (505,277.57) ;
%Straight Lines [id:da3797166095911113]
\draw [color={rgb, 255:red, 208; green, 2; blue, 27 }  ,draw opacity=1 ][line width=1.5]    (432,350.57) -- (505,350.57) ;
%Straight Lines [id:da5965713051451851]
\draw [color={rgb, 255:red, 248; green, 231; blue, 28 }  ,draw opacity=1 ][line width=1.5]    (505,277.57) -- (505,350.57) ;
%Shape: Circle [id:dp2743067568148707]
\draw  [color={rgb, 255:red, 0; green, 0; blue, 0 }  ,draw opacity=1 ][fill={rgb, 255:red, 0; green, 0; blue, 0 }  ,fill opacity=1 ] (501,277.57) .. controls (501,275.36) and (502.79,273.57) .. (505,273.57) .. controls (507.21,273.57) and (509,275.36) .. (509,277.57) .. controls (509,279.78) and (507.21,281.57) .. (505,281.57) .. controls (502.79,281.57) and (501,279.78) .. (501,277.57) -- cycle ;
%Shape: Circle [id:dp4139060700431938]
\draw  [color={rgb, 255:red, 0; green, 0; blue, 0 }  ,draw opacity=1 ][fill={rgb, 255:red, 0; green, 0; blue, 0 }  ,fill opacity=1 ] (501,350.57) .. controls (501,348.36) and (502.79,346.57) .. (505,346.57) .. controls (507.21,346.57) and (509,348.36) .. (509,350.57) .. controls (509,352.78) and (507.21,354.57) .. (505,354.57) .. controls (502.79,354.57) and (501,352.78) .. (501,350.57) -- cycle ;
%Straight Lines [id:da5219288300825065]
\draw [color={rgb, 255:red, 248; green, 231; blue, 28 }  ,draw opacity=1 ][line width=1.5]    (432,277.57) -- (432,350.57) ;
%Shape: Circle [id:dp25086485925693536]
\draw  [color={rgb, 255:red, 0; green, 0; blue, 0 }  ,draw opacity=1 ][fill={rgb, 255:red, 0; green, 0; blue, 0 }  ,fill opacity=1 ] (428,350.57) .. controls (428,348.36) and (429.79,346.57) .. (432,346.57) .. controls (434.21,346.57) and (436,348.36) .. (436,350.57) .. controls (436,352.78) and (434.21,354.57) .. (432,354.57) .. controls (429.79,354.57) and (428,352.78) .. (428,350.57) -- cycle ;
%Shape: Circle [id:dp6463114215313843]
\draw  [color={rgb, 255:red, 0; green, 0; blue, 0 }  ,draw opacity=1 ][fill={rgb, 255:red, 0; green, 0; blue, 0 }  ,fill opacity=1 ] (428,277.57) .. controls (428,275.36) and (429.79,273.57) .. (432,273.57) .. controls (434.21,273.57) and (436,275.36) .. (436,277.57) .. controls (436,279.78) and (434.21,281.57) .. (432,281.57) .. controls (429.79,281.57) and (428,279.78) .. (428,277.57) -- cycle ;
%Right Arrow [id:dp062152089813581]
\draw   (317,295.37) -- (331.19,295.37) -- (331.19,292.37) -- (340.66,298.37) -- (331.19,304.37) -- (331.19,301.37) -- (317,301.37) -- cycle ;
% Text Node
\draw (128,195) node [anchor=north west][inner sep=0.75pt]   [align=left] {$u_1$};
% Text Node
\draw (187,195) node [anchor=north west][inner sep=0.75pt]   [align=left] {$u_2$};
% Text Node
\draw (245,195) node [anchor=north west][inner sep=0.75pt]   [align=left] {$u_3$};
% Text Node
\draw (156,355) node [anchor=north west][inner sep=0.75pt]   [align=left] {$v_1$};
% Text Node
\draw (220,355) node [anchor=north west][inner sep=0.75pt]   [align=left] {$v_2$};
% Text Node
\draw (398,182) node [anchor=north west][inner sep=0.75pt]   [align=left] {$u_1$};
% Text Node
\draw (461,182) node [anchor=north west][inner sep=0.75pt]   [align=left] {$u_2$};
% Text Node
\draw (525,182) node [anchor=north west][inner sep=0.75pt]   [align=left] {$u_3$};
% Text Node
\draw (410,272) node [anchor=north west][inner sep=0.75pt]   [align=left] {$v_1$};
% Text Node
\draw (515,272) node [anchor=north west][inner sep=0.75pt]   [align=left] {$v_2$};
% Text Node
\draw (410,350) node [anchor=north west][inner sep=0.75pt]   [align=left] {$v'_1$};
% Text Node
\draw (515,350) node [anchor=north west][inner sep=0.75pt]   [align=left] {$v'_2$};
% Text Node
%\draw (120,385) node [anchor=north west][inner sep=0.75pt]   [align=left] {$\{u_1v_1v_2,u_2v_1v_2,u_3v_1v_2\}$};
%% Text Node
%\draw (365,385) node [anchor=north west][inner sep=0.75pt]   [align=left] {$\{u_1v_1v_2,u_2v_1v_2,u_3v_1v_2\}\boxplus\{v_1,v_2\}$};
\end{tikzpicture}
\caption{$\{u_1v_1v_2,u_2v_1v_2,u_3v_1v_2\}$ and $\{u_1v_1v_2,u_2v_1v_2,u_3v_1v_2\}\boxplus\{v_1,v_2\}$.
        The link of red vertices is the red $K_{2,2}$, the link of the yellow vertex is the yellow $K_{2,2}$.}
\end{figure}

\begin{figure}[htbp]
\centering
\tikzset{every picture/.style={line width=0.75pt}} %set default line width to 0.75pt
\begin{tikzpicture}[x=0.75pt,y=0.75pt,yscale=-1,xscale=1]
%uncomment if require: \path (0,335); %set diagram left start at 0, and has height of 335
%Straight Lines [id:da27148598584772743]
\draw [draw opacity=0][fill=uuuuuu,fill opacity=0.5 ][line width=1.5]    (197.5,166.5) -- (133,209) -- (197.67,93) -- cycle ;
%Straight Lines [id:da10946111707021844]
\draw [draw opacity=0][fill=uuuuuu,fill opacity=0.9 ][line width=1.5]    (197.67,93) -- (197.5,166.5) -- (262.13,209) -- cycle ;
%Straight Lines [id:da28855248607142037]
\draw [draw opacity=0][fill=uuuuuu,fill opacity=0.7 ][line width=1.5]    (197.5,166.5) -- (262.13,209) -- (133,209) -- cycle ;
%Shape: Circle [id:dp2690033406545205]
\draw [draw opacity=0][fill=uuuuuu,fill opacity=0.5 ]   (456,50) -- (456,94) -- (416,145) -- cycle ;
\draw [draw opacity=0][fill=uuuuuu,fill opacity=0.5 ]   (456,50) -- (456,94) -- (489,145) -- cycle ;
\draw [draw opacity=0][fill=uuuuuu,fill opacity=0.5 ]   (456,50) -- (456,94) -- (416,221) -- cycle ;
\draw [draw opacity=0][fill=uuuuuu,fill opacity=0.5 ]   (456,50) -- (456,94) -- (489,221) -- cycle ;
\draw  [fill={rgb, 255:red, 0; green, 0; blue, 0 }  ,fill opacity=1 ] (194,166.5) .. controls (194,164.57) and (195.57,163) .. (197.5,163) .. controls (199.43,163) and (201,164.57) .. (201,166.5) .. controls (201,168.43) and (199.43,170) .. (197.5,170) .. controls (195.57,170) and (194,168.43) .. (194,166.5) -- cycle ;
%Shape: Circle [id:dp722943896656709]
\draw  [fill={rgb, 255:red, 0; green, 0; blue, 0 }  ,fill opacity=1 ] (194.17,93) .. controls (194.17,91.07) and (195.74,89.5) .. (197.67,89.5) .. controls (199.61,89.5) and (201.17,91.07) .. (201.17,93) .. controls (201.17,94.93) and (199.61,96.5) .. (197.67,96.5) .. controls (195.74,96.5) and (194.17,94.93) .. (194.17,93) -- cycle ;
%Shape: Circle [id:dp8492148008144036]
\draw  [fill={rgb, 255:red, 0; green, 0; blue, 0 }  ,fill opacity=1 ] (258.63,209) .. controls (258.63,207.07) and (260.2,205.5) .. (262.13,205.5) .. controls (264.07,205.5) and (265.63,207.07) .. (265.63,209) .. controls (265.63,210.93) and (264.07,212.5) .. (262.13,212.5) .. controls (260.2,212.5) and (258.63,210.93) .. (258.63,209) -- cycle ;
%Shape: Circle [id:dp7200773255993671]
\draw  [fill={rgb, 255:red, 0; green, 0; blue, 0 }  ,fill opacity=1 ] (129.5,209) .. controls (129.5,207.07) and (131.07,205.5) .. (133,205.5) .. controls (134.93,205.5) and (136.5,207.07) .. (136.5,209) .. controls (136.5,210.93) and (134.93,212.5) .. (133,212.5) .. controls (131.07,212.5) and (129.5,210.93) .. (129.5,209) -- cycle ;
%Shape: Circle [id:dp7967811434385534]
\draw  [color={rgb, 255:red, 208; green, 2; blue, 27 }  ,draw opacity=1 ][fill={rgb, 255:red, 208; green, 2; blue, 27 }  ,fill opacity=1 ] (452.17,52) .. controls (452.17,50.07) and (453.74,48.5) .. (455.67,48.5) .. controls (457.61,48.5) and (459.17,50.07) .. (459.17,52) .. controls (459.17,53.93) and (457.61,55.5) .. (455.67,55.5) .. controls (453.74,55.5) and (452.17,53.93) .. (452.17,52) -- cycle ;
%Shape: Circle [id:dp4774466902395187]
\draw  [color={rgb, 255:red, 248; green, 231; blue, 28 }  ,draw opacity=1 ][fill={rgb, 255:red, 248; green, 231; blue, 28 }  ,fill opacity=1 ] (452.17,93) .. controls (452.17,91.07) and (453.74,89.5) .. (455.67,89.5) .. controls (457.61,89.5) and (459.17,91.07) .. (459.17,93) .. controls (459.17,94.93) and (457.61,96.5) .. (455.67,96.5) .. controls (453.74,96.5) and (452.17,94.93) .. (452.17,93) -- cycle ;
%Straight Lines [id:da21452137346935674]
\draw [color={rgb, 255:red, 248; green, 231; blue, 28 }  ,draw opacity=1 ][line width=1.5]    (417,149.2) -- (486,218.2) ;
%Straight Lines [id:da8090837592360811]
\draw [color={rgb, 255:red, 208; green, 2; blue, 27 }  ,draw opacity=1 ][line width=1.5]    (421,145.2) -- (490,214.2) ;
%Straight Lines [id:da6796195899100035]
\draw [color={rgb, 255:red, 248; green, 231; blue, 28 }  ,draw opacity=1 ][line width=1.5]    (417,214.2) -- (486,145.2) ;
%Straight Lines [id:da7684442579691388]
\draw [color={rgb, 255:red, 208; green, 2; blue, 27 }  ,draw opacity=1 ][line width=1.5]    (421,218.2) -- (490,149.2) ;
%Straight Lines [id:da7290372643061915]
\draw [color={rgb, 255:red, 208; green, 2; blue, 27 }  ,draw opacity=1 ][line width=1.5]    (417,145.2) -- (490,145.2) ;
%Straight Lines [id:da24429050114668827]
\draw [color={rgb, 255:red, 208; green, 2; blue, 27 }  ,draw opacity=1 ][line width=1.5]    (417,218.2) -- (490,218.2) ;
%Straight Lines [id:da4231891522606539]
\draw [color={rgb, 255:red, 248; green, 231; blue, 28 }  ,draw opacity=1 ][line width=1.5]    (490,145.2) -- (490,218.2) ;
%Shape: Circle [id:dp7845353605383356]
\draw  [color={rgb, 255:red, 0; green, 0; blue, 0 }  ,draw opacity=1 ][fill={rgb, 255:red, 0; green, 0; blue, 0 }  ,fill opacity=1 ] (486,145.2) .. controls (486,142.99) and (487.79,141.2) .. (490,141.2) .. controls (492.21,141.2) and (494,142.99) .. (494,145.2) .. controls (494,147.41) and (492.21,149.2) .. (490,149.2) .. controls (487.79,149.2) and (486,147.41) .. (486,145.2) -- cycle ;
%Shape: Circle [id:dp1625102223393322]
\draw  [color={rgb, 255:red, 0; green, 0; blue, 0 }  ,draw opacity=1 ][fill={rgb, 255:red, 0; green, 0; blue, 0 }  ,fill opacity=1 ] (486,218.2) .. controls (486,215.99) and (487.79,214.2) .. (490,214.2) .. controls (492.21,214.2) and (494,215.99) .. (494,218.2) .. controls (494,220.41) and (492.21,222.2) .. (490,222.2) .. controls (487.79,222.2) and (486,220.41) .. (486,218.2) -- cycle ;
%Straight Lines [id:da87770068738158]
\draw [color={rgb, 255:red, 248; green, 231; blue, 28 }  ,draw opacity=1 ][fill={rgb, 255:red, 248; green, 231; blue, 28 }  ,fill opacity=1 ][line width=1.5]    (417,145.2) -- (417,218.2) ;
%Shape: Circle [id:dp8912610025207777]
\draw  [color={rgb, 255:red, 0; green, 0; blue, 0 }  ,draw opacity=1 ][fill={rgb, 255:red, 0; green, 0; blue, 0 }  ,fill opacity=1 ] (413,218.2) .. controls (413,215.99) and (414.79,214.2) .. (417,214.2) .. controls (419.21,214.2) and (421,215.99) .. (421,218.2) .. controls (421,220.41) and (419.21,222.2) .. (417,222.2) .. controls (414.79,222.2) and (413,220.41) .. (413,218.2) -- cycle ;
%Shape: Circle [id:dp7930198682480216]
\draw  [color={rgb, 255:red, 0; green, 0; blue, 0 }  ,draw opacity=1 ][fill={rgb, 255:red, 0; green, 0; blue, 0 }  ,fill opacity=1 ] (413,145.2) .. controls (413,142.99) and (414.79,141.2) .. (417,141.2) .. controls (419.21,141.2) and (421,142.99) .. (421,145.2) .. controls (421,147.41) and (419.21,149.2) .. (417,149.2) .. controls (414.79,149.2) and (413,147.41) .. (413,145.2) -- cycle ;
%Right Arrow [id:dp878619978049479]
\draw   (321,142) -- (335.19,142) -- (335.19,139) -- (344.66,145) -- (335.19,151) -- (335.19,148) -- (321,148) -- cycle ;
% Text Node
\draw (395,139) node [anchor=north west][inner sep=0.75pt]   [align=left] {3};
% Text Node
\draw (500,139) node [anchor=north west][inner sep=0.75pt]   [align=left] {4};
% Text Node
\draw (395,205) node [anchor=north west][inner sep=0.75pt]   [align=left] {$3'$};
% Text Node
\draw (500,205) node [anchor=north west][inner sep=0.75pt]   [align=left] {$4'$};
% Text Node
%\draw (420,238) node [anchor=north west][inner sep=0.75pt]   [align=left] {$K_4^3\boxplus \{3, 4\}$};
%% Text Node
%\draw (185,238) node [anchor=north west][inner sep=0.75pt]   [align=left] {$K_4^3$};
% Text Node
\draw (205,85) node [anchor=north west][inner sep=0.75pt]   [align=left] {1};
% Text Node
\draw (115,200) node [anchor=north west][inner sep=0.75pt]   [align=left] {3};
% Text Node
\draw (205,155) node [anchor=north west][inner sep=0.75pt]   [align=left] {2};
% Text Node
\draw (266,200) node [anchor=north west][inner sep=0.75pt]   [align=left] {4};
% Text Node
\draw (465,45) node [anchor=north west][inner sep=0.75pt]   [align=left] {1};
% Text Node
\draw (465,85) node [anchor=north west][inner sep=0.75pt]   [align=left] {2};
\end{tikzpicture}
\caption{$K_4^3$ and $K_4^3\boxplus \{3, 4\}$.}
\end{figure}
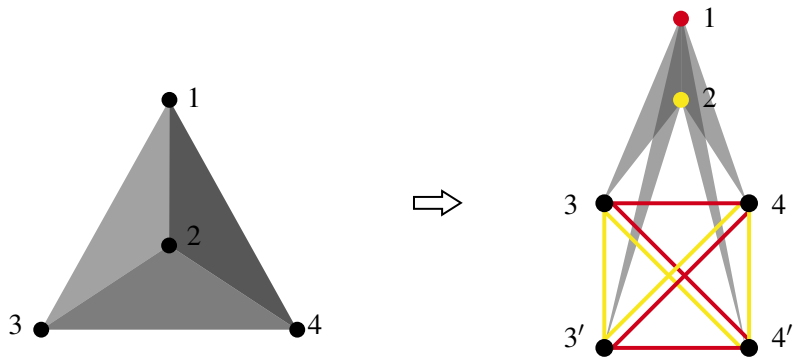

\begin{definition}
Let $\mathcal{G}$ be a $3$-graph.
A pair $\{v_1,v_2\}\subset V(\mathcal{G})$ is symmetric in $\mathcal{G}$ if
$$L_{\mathcal{G}}(v_1) - v_2 = L_{\mathcal{G}}(v_2) - v_1.$$
\end{definition}

The crossed blowup of a $3$-graph has the following properties.

\begin{proposition}\label{PROP:nonminimal-corss-blowup}
Suppose that $\mathcal{G}$ is an $m$-vertex $3$-graph and $\{v_1,v_2\}\subset V(\mathcal{G})$ is a pair of vertices with $d(v_1,v_2)\ge 2$.
Then the following statements hold.
\begin{enumerate}[label=(\alph*)]
\item The $3$-graph $\mathcal{G}$ is contained in $\mathcal{G}\boxplus\{v_1,v_2\}$ as an induced subgraph.
        In particular, $\lambda(\mathcal{G}) \le \lambda(\mathcal{G}\boxplus\{v_1,v_2\})$.
\item The $3$-graph $\mathcal{G}\boxplus\{v_1,v_2\}$ is $2$-covered iff $\mathcal{G}$ is $2$-covered.
\item If $\{v_1,v_2\}$ is symmetric in $\mathcal{G}$,
        then $\lambda(\mathcal{G}\boxplus\{v_1,v_2\}) = \lambda(\mathcal{G})$.
        If, in addition, there exists $(x_1,\ldots,x_{m})\in Z(\mathcal{G})$ with $x_1 + x_2 > 0$,
        then the set $Z(\mathcal{G}\boxplus\{v_1,v_2\})$ contains a one-dimensional simplex (i.e. a nontrivial line segment).
\end{enumerate}
\end{proposition}
\noindent\textbf{Remarks.}
\begin{itemize}
\item It is easy to see that Proposition~\ref{PROP:nonminimal-corss-blowup}~(b) does not necessarily
hold for the ordinary blowup of $\mathcal{G}$.

\item The requirement that $\{v_1,v_2\}$ is symmetric in
Proposition~\ref{PROP:nonminimal-corss-blowup}~(c)
cannot be removed since some calculations show that
$\lambda(C_{5}^{3}) = 1/25$ while $\lambda(C_{5}^{3}\boxplus\{4,5\}) = 4/81$,
where $C_{5}^{3} = \{123,234,345,451,512\}$.
\end{itemize}

\begin{figure}[htbp]
\centering
\tikzset{every picture/.style={line width=0.75pt}} %set default line width to 0.75pt
\begin{tikzpicture}[x=0.75pt,y=0.75pt,yscale=-1,xscale=1]
%uncomment if require: \path (0,335); %set diagram left start at 0, and has height of 335
%Straight Lines [id:da7168940547873892]
\draw [color={rgb, 255:red, 248; green, 231; blue, 28 }  ,draw opacity=1 ][line width=1.5]    (393.5,210) -- (449,168.5) ;
%Straight Lines [id:da5634675898472736]
\draw [color={rgb, 255:red, 208; green, 2; blue, 27 }  ,draw opacity=1 ][line width=1.5]    (397,213.5) -- (452.5,172) ;
%Straight Lines [id:da3717948075076003]
\draw [color={rgb, 255:red, 248; green, 231; blue, 28 }  ,draw opacity=1 ][line width=1.5]    (397,168.5) -- (452.5,210) ;
%Straight Lines [id:da3686840932484756]
\draw [color={rgb, 255:red, 208; green, 2; blue, 27 }  ,draw opacity=1 ][line width=1.5]    (393.5,172) -- (449,213.5) ;
%Straight Lines [id:da4218484256075099]
\draw [color={rgb, 255:red, 208; green, 2; blue, 27 }  ,draw opacity=1 ][line width=1.5]    (393.5,168.5) -- (452.5,168.5) ;
%Straight Lines [id:da04279868434491951]
\draw [color={rgb, 255:red, 208; green, 2; blue, 27 }  ,draw opacity=1 ][line width=1.5]    (452.52,213.26) -- (393.5,213.5) ;
%Straight Lines [id:da39612265488648424]
\draw [draw opacity=0][fill=uuuuuu,fill opacity=0.2 ]   (478.5,122.5) -- (423.82,88.97) -- (452.5,213.5) -- cycle ;
%Straight Lines [id:da9801196055618364]
\draw [draw opacity=0][fill=uuuuuu,fill opacity=0.2 ]   (368.5,122.5) -- (423.58,89.38) -- (393.5,213.5) -- cycle ;
%Straight Lines [id:da8987465286674774]
\draw [color={rgb, 255:red, 248; green, 231; blue, 28 }  ,draw opacity=1 ][line width=1.5]    (393.5,213.5) -- (393.5,168.5) ;
%Straight Lines [id:da9984126145669769]
\draw [color={rgb, 255:red, 248; green, 231; blue, 28 }  ,draw opacity=1 ][line width=1.5]    (452.52,213.26) -- (452.52,168.26) ;
%Straight Lines [id:da411770141639902]
\draw [draw opacity=0][fill=uuuuuu,fill opacity=0.4 ]   (478.5,122.5) -- (423.58,89.38) -- (452.5,168.5) -- cycle ;
%Straight Lines [id:da11981404960252684]
\draw [draw opacity=0][fill=uuuuuu,fill opacity=0.4 ]   (368.5,122.5) -- (423.5,89.5) -- (393.5,168.5) -- cycle ;
%Straight Lines [id:da639492972049809]
\draw [draw opacity=0][fill=uuuuuu,fill opacity=0.3 ]   (195.07,107) -- (244.58,140.86) -- (225.67,195.64) -- cycle ;
%Straight Lines [id:da363227301222357]
\draw [draw opacity=0][fill=uuuuuu,fill opacity=0.3 ]   (145.55,140.86) -- (195.07,107) -- (244.58,140.86) -- cycle ;
%Straight Lines [id:da43115027009263773]
\draw [draw opacity=0][fill=uuuuuu,fill opacity=0.3 ]   (145.55,140.86) -- (164.46,195.64) -- (195.07,107) -- cycle ;
%Straight Lines [id:da670402393742678]
\draw [draw opacity=0][fill=uuuuuu,fill opacity=0.3 ]   (164.46,195.64) -- (225.67,195.64) -- (145.55,140.86) -- cycle ;
%Straight Lines [id:da34034185731162125]
\draw [draw opacity=0][fill=uuuuuu,fill opacity=0.3 ]   (225.67,195.64) -- (164.46,195.64) -- (244.58,140.86) ;
%Shape: Circle [id:dp9196859963316548]
\draw  [fill={rgb, 255:red, 0; green, 0; blue, 0 }  ,fill opacity=1 ] (191.57,107) .. controls (191.57,105.07) and (193.13,103.5) .. (195.07,103.5) .. controls (197,103.5) and (198.57,105.07) .. (198.57,107) .. controls (198.57,108.93) and (197,110.5) .. (195.07,110.5) .. controls (193.13,110.5) and (191.57,108.93) .. (191.57,107) -- cycle ;
%Shape: Circle [id:dp2551128482044944]
\draw  [fill={rgb, 255:red, 0; green, 0; blue, 0 }  ,fill opacity=1 ] (241.08,140.86) .. controls (241.08,138.93) and (242.65,137.36) .. (244.58,137.36) .. controls (246.52,137.36) and (248.08,138.93) .. (248.08,140.86) .. controls (248.08,142.79) and (246.52,144.36) .. (244.58,144.36) .. controls (242.65,144.36) and (241.08,142.79) .. (241.08,140.86) -- cycle ;
%Shape: Circle [id:dp040046703546830154]
\draw  [fill={rgb, 255:red, 0; green, 0; blue, 0 }  ,fill opacity=1 ] (222.17,195.64) .. controls (222.17,193.71) and (223.74,192.14) .. (225.67,192.14) .. controls (227.6,192.14) and (229.17,193.71) .. (229.17,195.64) .. controls (229.17,197.57) and (227.6,199.14) .. (225.67,199.14) .. controls (223.74,199.14) and (222.17,197.57) .. (222.17,195.64) -- cycle ;
%Shape: Circle [id:dp6873231067984267]
\draw  [fill={rgb, 255:red, 0; green, 0; blue, 0 }  ,fill opacity=1 ] (142.05,140.86) .. controls (142.05,138.93) and (143.62,137.36) .. (145.55,137.36) .. controls (147.48,137.36) and (149.05,138.93) .. (149.05,140.86) .. controls (149.05,142.79) and (147.48,144.36) .. (145.55,144.36) .. controls (143.62,144.36) and (142.05,142.79) .. (142.05,140.86) -- cycle ;
%Shape: Circle [id:dp3401945485666713]
\draw  [fill={rgb, 255:red, 0; green, 0; blue, 0 }  ,fill opacity=1 ] (160.96,195.64) .. controls (160.96,193.71) and (162.53,192.14) .. (164.46,192.14) .. controls (166.4,192.14) and (167.96,193.71) .. (167.96,195.64) .. controls (167.96,197.57) and (166.4,199.14) .. (164.46,199.14) .. controls (162.53,199.14) and (160.96,197.57) .. (160.96,195.64) -- cycle ;
%Shape: Circle [id:dp4761265641618351]
\draw  [fill={rgb, 255:red, 0; green, 0; blue, 0 }  ,fill opacity=1 ] (420,89.5) .. controls (420,87.57) and (421.57,86) .. (423.5,86) .. controls (425.43,86) and (427,87.57) .. (427,89.5) .. controls (427,91.43) and (425.43,93) .. (423.5,93) .. controls (421.57,93) and (420,91.43) .. (420,89.5) -- cycle ;
%Shape: Circle [id:dp07329537787538487]
\draw  [color={rgb, 255:red, 208; green, 2; blue, 27 }  ,draw opacity=1 ][fill={rgb, 255:red, 208; green, 2; blue, 27 }  ,fill opacity=1 ] (365,122.5) .. controls (365,120.57) and (366.57,119) .. (368.5,119) .. controls (370.43,119) and (372,120.57) .. (372,122.5) .. controls (372,124.43) and (370.43,126) .. (368.5,126) .. controls (366.57,126) and (365,124.43) .. (365,122.5) -- cycle ;
%Shape: Circle [id:dp6089303892645344]
\draw  [color={rgb, 255:red, 248; green, 231; blue, 28 }  ,draw opacity=1 ][fill={rgb, 255:red, 248; green, 231; blue, 28 }  ,fill opacity=1 ] (475,122.5) .. controls (475,120.57) and (476.57,119) .. (478.5,119) .. controls (480.43,119) and (482,120.57) .. (482,122.5) .. controls (482,124.43) and (480.43,126) .. (478.5,126) .. controls (476.57,126) and (475,124.43) .. (475,122.5) -- cycle ;
%Shape: Circle [id:dp4057062141312784]
\draw  [fill={rgb, 255:red, 0; green, 0; blue, 0 }  ,fill opacity=1 ] (390,168.5) .. controls (390,166.57) and (391.57,165) .. (393.5,165) .. controls (395.43,165) and (397,166.57) .. (397,168.5) .. controls (397,170.43) and (395.43,172) .. (393.5,172) .. controls (391.57,172) and (390,170.43) .. (390,168.5) -- cycle ;
%Shape: Circle [id:dp5101430786252492]
\draw  [fill={rgb, 255:red, 0; green, 0; blue, 0 }  ,fill opacity=1 ] (449,168.5) .. controls (449,166.57) and (450.57,165) .. (452.5,165) .. controls (454.43,165) and (456,166.57) .. (456,168.5) .. controls (456,170.43) and (454.43,172) .. (452.5,172) .. controls (450.57,172) and (449,170.43) .. (449,168.5) -- cycle ;
%Shape: Circle [id:dp42368629428280746]
\draw  [fill={rgb, 255:red, 0; green, 0; blue, 0 }  ,fill opacity=1 ] (390,213.5) .. controls (390,211.57) and (391.57,210) .. (393.5,210) .. controls (395.43,210) and (397,211.57) .. (397,213.5) .. controls (397,215.43) and (395.43,217) .. (393.5,217) .. controls (391.57,217) and (390,215.43) .. (390,213.5) -- cycle ;
%Shape: Circle [id:dp9200155229801568]
\draw  [fill={rgb, 255:red, 0; green, 0; blue, 0 }  ,fill opacity=1 ] (449,213.5) .. controls (449,211.57) and (450.57,210) .. (452.5,210) .. controls (454.43,210) and (456,211.57) .. (456,213.5) .. controls (456,215.43) and (454.43,217) .. (452.5,217) .. controls (450.57,217) and (449,215.43) .. (449,213.5) -- cycle ;
%Right Arrow [id:dp11427495087628126]
\draw   (298,158) -- (312.19,158) -- (312.19,155) -- (321.66,161) -- (312.19,167) -- (312.19,164) -- (298,164) -- cycle ;

% Text Node
\draw (430,75) node [anchor=north west][inner sep=0.75pt]   [align=left] {2};
% Text Node
\draw (485,115) node [anchor=north west][inner sep=0.75pt]   [align=left] {3};
% Text Node
\draw (350,115) node [anchor=north west][inner sep=0.75pt]   [align=left] {1};
% Text Node
\draw (375,162) node [anchor=north west][inner sep=0.75pt]   [align=left] {5};
% Text Node
\draw (460,162) node [anchor=north west][inner sep=0.75pt]   [align=left] {4};
% Text Node
\draw (375,205) node [anchor=north west][inner sep=0.75pt]   [align=left] {$5'$};
% Text Node
\draw (460,205) node [anchor=north west][inner sep=0.75pt]   [align=left] {$4'$};
% Text Node
\draw (190,86) node [anchor=north west][inner sep=0.75pt]   [align=left] {2};
% Text Node
\draw (250,132) node [anchor=north west][inner sep=0.75pt]   [align=left] {3};
% Text Node
\draw (130,132) node [anchor=north west][inner sep=0.75pt]   [align=left] {1};
% Text Node
\draw (231,198) node [anchor=north west][inner sep=0.75pt]   [align=left] {4};
% Text Node
\draw (150,198) node [anchor=north west][inner sep=0.75pt]   [align=left] {5};
% Text Node
%\draw (185,225) node [anchor=north west][inner sep=0.75pt]   [align=left] {$C_5^3$};
%% Text Node
%\draw (390,225) node [anchor=north west][inner sep=0.75pt]   [align=left] {$C_5^3\boxplus \{4, 5\}$};
\end{tikzpicture}
\caption{$C_{5}^3$ and $C_{5}^3\boxplus \{4,5\}$.}
\end{figure}
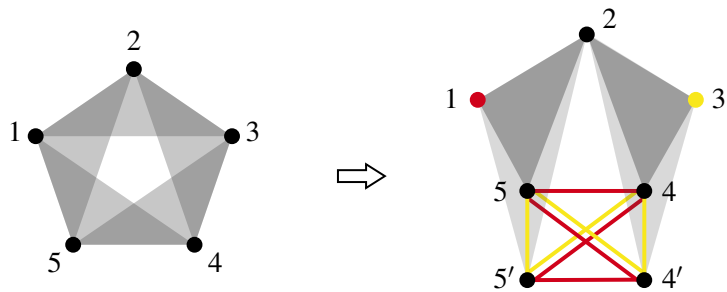

Using Propositions~\ref{PROP:nonminal-finite-family} and~\ref{PROP:nonminimal-corss-blowup} and some further argument
we obtain the following theorem.

\begin{theorem}\label{THM:nonminimal-general-turan-stability}
Let $\mathcal{G}$ be a minimal $3$-graph.
Suppose that $\{v_1,v_2\}\subset V(\mathcal{G})$ is symmetric in $\mathcal{G}$ and $d(v_1,v_2)\ge 1$.
Then there exists a finite family $\mathcal{F}$ of $3$-graphs with $\xi(\mathcal{F}) = \infty$
such that $\pi(\mathcal{F}) = 6 \lambda(\mathcal{G})$,
and for every $n\in\mathbb{N}$ there exists a maximum $\mathcal{F}$-free $3$-graph
on $n$-vertices which is a blowup of $\mathcal{G}$.
\end{theorem}
%%%%%%%%%%%%%%%%%%%%%%%%%%%%%%%%%%%%%%%%%%%%%%%%%
\subsection{Feasible region}\label{SUBSEC:feasible-region}
Theorem~\ref{THM:main-sec1.1} has some interesting application to the so-called feasible region of hypergraphs.

Given an $r$-graph $\mathcal{H}$ the \emph{shadow} of $\mathcal{H}$ is defined as
\begin{align}
\partial\mathcal{H}
=
\left\{A\in \binom{V(\mathcal{H})}{r-1}\colon \text{there is } B\in \mathcal{H} \text{ such that }
	A\subseteq B\right\}. \notag
\end{align}
The \emph{edge density} of $\mathcal{H}$ is defined as $\rho(\mathcal{H})= |\mathcal{H}|/\binom{v(\mathcal{H})}{r}$, and
the \emph{shadow density} of $\mathcal{H}$ is defined as $\rho(\partial\mathcal{H})= |\partial\mathcal{H}|/\binom{v(\mathcal{H})}{r-1}$.

For a family $\mathcal{F}$ the \emph{feasible region} $\Omega(\mathcal{F})$ of $\mathcal{F}$
is the set of points $(x,y)\in [0,1]^2$ such that there exists a sequence of $\mathcal{F}$-free $r$-graphs
$\left( \mathcal{H}_{k}\right)_{k=1}^{\infty}$ with
$$\lim_{k \to \infty}v(\mathcal{H}_{k}) = \infty,\quad
\lim_{k \to \infty}\rho(\partial\mathcal{H}_{k}) = x, \quad\text{and}\quad \lim_{k \to \infty}\rho(\mathcal{H}_{k}) = y.$$
The feasible region unifies and generalizes many classical problems
such as the  Kruskal--Katona theorem~\cite{KR63,KA68} and the Tur\'{a}n problem.
It was introduced in \cite{LM1} to understand the extremal properties of $\mathcal{F}$-free hypergraphs
beyond just the determination of $\pi(\mathcal{F})$. The general shape of $\Omega(\mathcal{F})$ was analyzed
in~\cite{LM1} as follows: For some constant $c(\mathcal{F})\in [0, 1]$ the projection to the first
coordinate,
\[
	{\rm proj}\Omega(\mathcal{F})
	=
	\left\{ x \colon  \text{there is $y \in [0,1]$ such that $(x,y) \in \Omega(\mathcal{F})$} \right\},
\]
is the interval $[0, c(\mathcal{F})]$ .
Moreover, there is a left-continuous almost everywhere differentiable function
$g(\mathcal{F})\colon {\rm proj}\Omega(\mathcal{F}) \to [0,1]$ such that
\[
	\Omega(\mathcal{F})
	=
	\bigl\{(x, y)\in [0, c(\mathcal{F})]\times [0, 1]\colon 0\le y\le g(\mathcal{F})(x)\bigr\}\,.
\]
Let us call $g(\mathcal{F})$ the \emph{feasible region function} of $\mathcal{F}$.
It was shown in~\cite{LM1} that $g(\mathcal{F})$ is not necessarily continuous,
and in~\cite{Liu20a}, it was shown that $g(\mathcal{F})$ can
have infinitely many local maxima even for some simple and natural families $\mathcal{F}$.

The stability number of $\mathcal{F}$ can give information about the number of global maxima of $g(\mathcal{F})$ (e.g. see~\cite{LMR1}).
The family $\mathcal{M}$ of triple systems from~\cite{LM22} for which $\xi(\mathcal{M})=2$ has the following
additional property: not only are the two near extremal constructions for $\mathcal{M}$ far from each
other in edit-distance, but the same is true of their shadows.
As a consequence, in addition to $\xi(\mathcal{M}) = 2$, the function~$g(\mathcal{M})$ has exactly two global maxima.
Similarly, it was proved in~\cite{LMR1} that for the $t$-stable family $\mathcal{M}_t$
the function $g(\mathcal{M}_t)$ has exactly $t$ global maxima for every positive integer $t$.

It was left as an open question in~\cite{LMR1}
whether there exists a finite family $\mathcal{F}$
so that the function $g(\mathcal{F})$ has infinitely many global maxima.
Here we would like to remind the reader that even though there are
infinitely many extremal constructions for Tur\'{a}n's tetrahedron conjecture (if it is true),
the shadow of all these constructions is complete.
Hence, solving Tur\'{a}n's tetrahedron conjecture will not answer the question asked in~\cite{LMR1}.

Our next result shows that the same family $\mathcal{F}_t$ as in Theorem~\ref{THM:main-sec1.1}
has the property that $g(\mathcal{F}_t)$ attains its maximum on a nontrivial interval,
thus giving a positive answer to the question in~\cite{LMR1}.

\begin{theorem}\label{THM:feasibe-region}
For every integer $t\ge 3$ we have $\mathrm{proj}\Omega(\mathcal{F}_t) = \left[0, \frac{t+1}{t+2}\right]$,
and $g(\mathcal{F}_t, x)\leq {(t-2)(t-1)}/{t^2}$ for all $x\in \mathrm{proj}\Omega(\mathcal{F}_t)$.
Moreover, if $t\ge 4$, then $g(\mathcal{F}_t, x)={(t-2)(t-1)}/{t^2}$
iff $x\in \left[\frac{t-1}{t}, \frac{t-1}{t}+\frac{1}{t^2}\right]$.
\end{theorem}
\noindent\textbf{Remark.}
%Since the interval $\left[\frac{t-1}{t}, \frac{t-1}{t}+\frac{1}{t^2}\right]$ contains a transcendental number,
Fox (see~\cite{PI14}) asked whether there exists a finite family $\mathcal{F}$ of $r$-graphs for some $r\ge 3$
such that the Tur\'{a}n density $\pi(\mathcal{F})$ is a transcendental number.
Though we did not answer his question, Theorem~\ref{THM:feasibe-region} does imply that there exists
a sequence of (near) extremal $\mathcal{F}_t$-free $3$-graphs whose shadow densities approach a transcendental number
(because the interval $\left[\frac{t-1}{t}, \frac{t-1}{t}+\frac{1}{t^2}\right]$ contains a transcendental number).

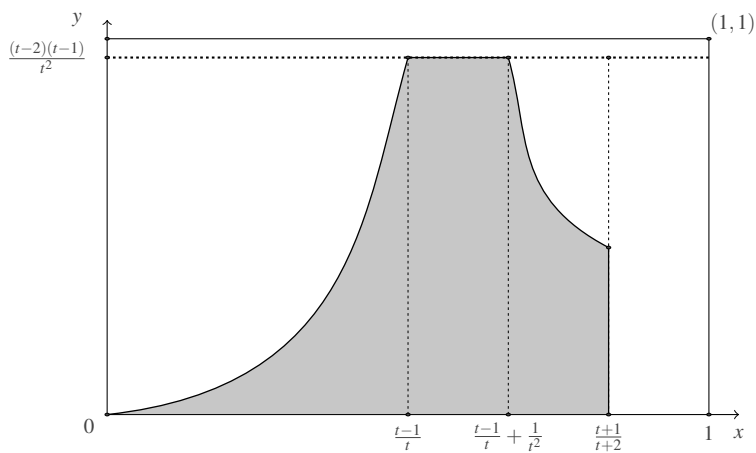
\begin{figure}[htbp]
\centering
\begin{tikzpicture}[xscale=8,yscale=5]
\draw [->] (0,0)--(1+0.05,0);
\draw [->] (0,0)--(0,1+0.05);
\draw (0,1)--(1,1);
\draw (1,0)--(1,1);
\draw [line width=0.8pt,dash pattern=on 1pt off 1.2pt,domain=0:1] plot(\x,1-0.05);
\draw [line width=0.4pt,dash pattern=on 1pt off 1.2pt] (1/2,0) -- (1/2,1-0.05);
\draw [line width=0.4pt,dash pattern=on 1pt off 1.2pt] (2/3,0) -- (2/3,1-0.05);
%\draw [line width=0.4pt,dash pattern=on 1pt off 1.2pt] (3/4,0) -- (3/4,1-0.05);
%\draw [line width=0.4pt,dash pattern=on 1pt off 1.2pt] (4/5,0) -- (4/5,1-0.05);
\draw [line width=0.4pt,dash pattern=on 1pt off 1.2pt] (5/6,0) -- (5/6,1-0.05);

\draw[fill=sqsqsq,fill opacity=0.25, line width=0.5pt]
(0,0)  to [out = 10, in = 260] (1/2,1-0.05) to (2/3, 1-0.05)
to [out = 280, in = 140] (5/6,4/9) to (5/6,0);

\begin{scriptsize}
\draw [fill=uuuuuu] (1,0) circle (0.1pt);
\draw[color=uuuuuu] (1,0-0.05) node {$1$};
\draw [fill=uuuuuu] (0,0) circle (0.1pt);
\draw[color=uuuuuu] (0-0.03,0-0.03) node {$0$};
\draw [fill=uuuuuu] (0,1) circle (0.1pt);
\draw[color=uuuuuu] (0-0.05,1+0.05) node {$y$};
\draw[color=uuuuuu] (1+0.05,0-0.05) node {$x$};
\draw [fill=uuuuuu] (0,1-0.05) circle (0.1pt);
\draw[color=uuuuuu] (0-0.1,1-0.05) node {$\frac{(t-2)(t-1)}{t^2}$};

\draw [fill=uuuuuu] (1/2,1-0.05) circle (0.1pt);
\draw [fill=uuuuuu] (1/2,0) circle (0.1pt);
\draw[color=uuuuuu] (1/2,0-0.06) node {$\frac{t-1}{t}$};

\draw [fill=uuuuuu] (2/3,1-0.05) circle (0.1pt);
\draw [fill=uuuuuu] (2/3,0) circle (0.1pt);
\draw[color=uuuuuu] (2/3,0-0.06) node {$\frac{t-1}{t}+\frac{1}{t^2}$};

\draw [fill=uuuuuu] (5/6,1-0.05) circle (0.1pt);
\draw [fill=uuuuuu] (5/6,0) circle (0.1pt);
\draw[color=uuuuuu] (5/6,0-0.06) node {$\frac{t+1}{t+2}$};

\draw [fill=uuuuuu] (1,1) circle (0.1pt);
\draw[color=uuuuuu] (1+0.04,1+0.04) node {$(1,1)$};

\draw [fill=uuuuuu] (5/6,4/9) circle (0.1pt);
%%%%%%%%%%%%%%%%%%%%%%%%%%%%%%%%%%%%%%%%%%%
\end{scriptsize}
\end{tikzpicture}
\caption{The function $g(\mathcal{F}_t)$ attains its maximum on the interval $\left[\frac{t-1}{t}, \frac{t-1}{t}+\frac{1}{t^2}\right]$.}
\end{figure}

The remainder of this paper is organized as follows.
In Section~\ref{SEC:Preliminaries} we introduce some preliminary definitions and results.
In Section~\ref{SEC:Proof-multilinear-polynomials} we prove Proposition~\ref{PROP:multilinear-polynomial}.
In Section~\ref{SEC:Proof-nonminimal-hypergraphs} we prove Propositions~\ref{PROP:nonminal-finite-family} and~\ref{PROP:nonminimal-corss-blowup} and Theorem~\ref{THM:nonminimal-general-turan-stability}.
In Section~\ref{SEC:proof-turan-number} we prove Theorem~\ref{THM:main-sec1.1}~(a),~(b), and~(c).
In Section~\ref{SEC:proof-stability} we prove Theorem~\ref{THM:main-sec1.1}~(d).
In Section~\ref{SEC:proof-feasible-region} we prove Theorem~\ref{THM:feasibe-region}.
The final section contains some concluding remarks.

%%%%%%%%%%%%%%%%%%%%%%%%%%%%%%%%%%%%%%%%%%%%%%%%%
%%%%%%%%%%%%%%%%%%%%%%%%%%%%%%%%%%%%%%%%%%%%%%%%%
\section{Preliminaries}\label{SEC:Preliminaries}
We present some definitions and useful results in this section.

Recall that the shadow of an $r$-graph $\mathcal{H}$ is
\begin{align}
\partial\mathcal{H}
=
\left\{A\in \binom{V(\mathcal{H})}{r-1}\colon \text{there is } B\in \mathcal{H} \text{ such that }
	A\subseteq B\right\}. \notag
\end{align}
By setting $\partial_1\mathcal{H} = \partial\mathcal{H}$
for every $i \in [r-1]$ we define $\partial_{i}\mathcal{H}$ inductively
as $\partial_{i}\mathcal{H} = \partial \partial_{i-1}\mathcal{H}$.
In particular, $\partial_{r-2}\mathcal{H}$ is a graph.

Given an $r$-graph $\mathcal{H}$ and a vertex $v\in V(\mathcal{H})$
the \emph{neighborhood} of $v$ is defined as
\[
N_{\mathcal{H}}(v)
= \left\{u\in V(\mathcal{H})\setminus\{v\}\colon
     \exists E\in\mathcal{H}\text{ such that }\{u,v\}\subset E\right\}.
\]
Recall that the \emph{link} of $v$ in $\mathcal{H}$ is
$L_\mathcal{H}(v)
=\left\{e\in \partial\mathcal{H}\colon
     e\cup \{v\}\in \mathcal{H}\right\}.$
The \emph{degree} of $v$ is $d_{\mathcal{H}}(v)=|L_\mathcal{H}(v)|$.
Denote by $\delta(\mathcal{H})$ and $\Delta(\mathcal{H})$
the minimum degree and maximum degree of $\mathcal{H}$, respectively.
Recall that the neighborhood of $\{u,v\}\subset V(\mathcal{H})$ in $\mathcal{H}$ is
$N_{\mathcal{H}}(u,v)
=\left\{w\in V(\mathcal{H})\colon \{u,v,w\}\in \mathcal{H}\right\}.$
The codegree of  $\{u,v\}$ is $d_{\mathcal{H}}(u,v) = |N_{\mathcal{H}}(u,v)|$.
Denote by $\delta_2(\mathcal{H})$ and $\Delta_2(\mathcal{H})$
the minimum degree and maximum codegree of $\mathcal{H}$, respectively.
We will omit the subscript $\mathcal{H}$ from our notations if it is clear from the context.

Let $\mathcal{G}$ be an $r$-graph with the vertex set $[m]$, and $V_1,\ldots, V_m$ be $m$ disjoint sets
(each $V_i$ is allowed to be empty).
The blowup $\mathcal{G}[V_1,\ldots, V_m]$ of $\mathcal{G}$ is obtained from $\mathcal{G}$
by replacing each vertex $i$ with the set $V_i$ and replacing each edge $i_1\cdots i_r$
with the complete $r$-partite $r$-graph with parts $V_{i_1},\ldots,V_{i_r}$.

For a hypergraph $\mathcal{G}$ the maximum number of edges in a blowup of $\mathcal{G}$
is related to $\lambda(\mathcal{G})$
(e.g. see Frankl and F\"uredi~\cite{FF89} or Keevash's survey~\cite[{Section~3}]{KE11}).

\begin{lemma}[\cite{FR84,KE11}]\label{LEMMA:|H|<=lambda-T-n^r}
Suppose that $\mathcal{G}$ is an $r$-graph and $V_1,\ldots, V_m$ are $m$ disjoint sets with $\sum_{i\in[m]}|V_i| = n$.
Then $|\mathcal{G}[V_1,\ldots, V_m]| = p_{\mathcal{G}}\left(|V_1|/n,\ldots,|V_m|/n\right)n^r$.
In particular, $|\mathcal{G}[V_1,\ldots, V_m]|\le \lambda(G)n^r$.
\end{lemma}

Given an $r$-graph $F$ we say $\mathcal{H}$ is \emph{$F$-hom-free}
if there is no homomorphism from $F$ to $\mathcal{H}$.
This is equivalent to say that every blowup of $\mathcal{H}$ is $F$-free.
For a family $\mathcal{F}$ of $r$-graphs we say $\mathcal{H}$ is \emph{$\mathcal{F}$-hom-free}
if it is $F$-hom-free for all $F \in \mathcal{F}$.
An easy observation is that %for a $2$-covered $r$-graph $\mathcal{F}$
if an $r$-graph $F$ is $2$-covered, then $\mathcal{H}$ is $F$-free iff it is $F$-hom-free.

\begin{definition}[Blowup-invariance]
A family $\mathcal{F}$ of $r$-graphs is blowup-invariant if every
$\mathcal{F}$-free $r$-graph is also $\mathcal{F}$-hom-free.
\end{definition}

%\begin{definition}[Equivalence class]
Let $\mathcal{H}$ be an $r$-graph and $\{u,v\}\subset V(\mathcal{H})$ be two non-adjacent vertices
(i.e., no edge contains both $u$ and $v$).
We say $u$ and $v$ are \emph{equivalent} if $L_{\mathcal{H}}(u)=L_{\mathcal{H}}(v)$ (in particular, two equivalent vertices are non-adjacent).
Otherwise we say they are \emph{non-equivalent}.
An equivalence class of $\mathcal{H}$ is a maximal vertex set in which every pair of vertices are equivalent.
We say $\mathcal{H}$ is \emph{symmetrized} if it does not contain non-equivalent pairs of vertices.
In~\cite{LMR2}, the authors summarized the well known method of Zykov~\cite{Zy} symmetrization for solving Tur\'{a}n problems
into the following statement.

\begin{theorem}[see e.g.~\cite{LMR2}]\label{THEOREM:ex(n,F)=max-blowup-of-H}
Suppose that $\mathcal{F}$ is a blowup-invariant family of $r$-graphs. If $\mathfrak{H}$ denotes the class of all symmetrized $\mathcal{F}$-free $r$-graphs, then ex$(n, \mathcal{F})=\mathfrak{h}(n)$  holds for every $n \in \mathbb{N}^{+}$,
where $\mathfrak{h}(n)=\max\{|\mathcal{H}|\colon\mathcal{H}\in \mathfrak{H} \text{ and } v(\mathcal{H})=n\}$.
\end{theorem}
%%%%%%%%%%%%%%%%%%%%%%%%%%%%%%%%%%%%%%%%%%%%%%%%%
%%%%%%%%%%%%%%%%%%%%%%%%%%%%%%%%%%%%%%%%%%%%%%%%%
\section{Multilinear polynomials}\label{SEC:Proof-multilinear-polynomials}
We prove Proposition~\ref{PROP:multilinear-polynomial} in this section.
First, we need the following simple lemma.

\begin{lemma}\label{LEMMA:max-p-symmetric}
Suppose that $p(X_1,\ldots,X_m) = p_1 + p_2(X_i+X_j) + p_3X_iX_j$
is an $m$-variable multilinear polynomial that is symmetric with respect to $X_i$ and $X_j$,
and $p_3$ is nonnegative.
Then for every $(x_1, \ldots, x_m)\in Z(p)$ we have
$(x_1, \ldots,x_{i-1}, (x_i+x_j)/2, x_{i+1}, \ldots, x_{j-1}, (x_i+x_j)/2, x_{j+1}, \ldots, x_m)\in Z(p)$.
In particular,
$$\max\{p(y_1, \ldots, y_m)\colon (y_1, \ldots, y_m)\in \Delta_{m-1} \text{ and }y_i = y_j\} = \lambda(p).$$
\end{lemma}
\begin{proof}
Using the AM-GM inequality we obtain
\begin{align}
& p(x_1, \ldots,x_{i-1}, \frac{x_i+x_j}{2}, x_{i+1}, \ldots, x_{j-1}, \frac{x_i+x_j}{2}, x_{j+1}, \ldots, x_m) \notag\\
& = p_1 + p_2\left(\frac{x_i+x_j}{2}+ \frac{x_i+x_j}{2}\right) + p_3\left(\frac{x_i+x_j}{2}\right)^2  \notag\\
& \ge p_1 + p_2(x_i+x_j) + p_3 x_ix_j
= \lambda(p), \notag
\end{align}
where the last equality follows from our assumption that $(x_1, \ldots, x_m)\in Z(p)$.
By the definition of $\lambda(p)$, the inequality above is actually an equality, so we have
$(x_1, \ldots,x_{i-1}, (x_i+x_j)/2, x_{i+1}, \ldots, x_{j-1}, (x_i+x_j)/2, x_{j+1}, \ldots, x_m)\in Z(p)$.
\end{proof}

Now we are ready to prove Proposition~\ref{PROP:multilinear-polynomial}.

\begin{proof}[Proof of Proposition~\ref{PROP:multilinear-polynomial}]
Let $p, p_1, p_2, p_3, p_4, p_5, \hat{p}$ be polynomials that satisfy the assumptions in Proposition~\ref{PROP:multilinear-polynomial}.
By symmetry, we may assume that $\{i,j\} = \{1,2\}$.
Let $\widehat{X}_1 = X_1+ X_1'$ and $\widehat{X}_2 = X_2+X_2'$.
It follows from the AM-GM inequality that
\begin{align}\label{equ:hat-p-and-p}
& \hat{p}(X_1,X'_1,X_2,X'_2,X_3,\ldots,X_m)  \notag\\
&  = p_1 + p_2(X_1+X'_1+X_2+X'_2)
            + p_{4}(X_1+X'_1)(X_2+X'_2)
            + p_{5}(X_1+X_2)(X'_1+X'_2) \notag\\
& \le p_1 + p_2(\widehat{X}_1+\widehat{X}_2)
            + p_{4}\left(\frac{\widehat{X}_1+\widehat{X}_2}{2}\right)^2
            + p_{5}\left(\frac{\widehat{X}_1+\widehat{X}_2}{2}\right)^2 \notag\\
& = p_1 + p_2(\widehat{X}_1+\widehat{X}_2)
            + p_{3}\left(\frac{\widehat{X}_1+\widehat{X}_2}{2}\right)^2 \notag\\
& = p((\widehat{X}_1+\widehat{X}_2)/2, (\widehat{X}_1+\widehat{X}_2)/2, X_3, \ldots, X_m).
\end{align}
This implies that $\lambda(\hat{p}) \le \lambda(p)$.

Let $\alpha \in [0,1]$ and
\begin{align}
\vec{y} = \alpha \cdot (\frac{x_1+x_2}{2},0,0,\frac{x_1+x_2}{2}, x_3, \ldots, x_m)
+ (1-\alpha) \cdot (0,\frac{x_1+x_2}{2},\frac{x_1+x_2}{2},0, x_3, \ldots, x_m). \notag
\end{align}
Then it follows from Equation~(\ref{equ:hat-p-and-p}) that
\begin{align}
 \hat{p}(\vec{y})  = p((x_1+x_2)/2, (x_1+x_2)/2, x_3, \ldots, x_m). \notag
\end{align}
By Lemma~\ref{LEMMA:max-p-symmetric}, we have
\begin{align}
p((x_1+x_2)/2, (x_1+x_2)/2, x_3, \ldots, x_m)
= p(x_1, x_2, x_3, \ldots, x_m) = \lambda(p). \notag
\end{align}
Therefore, $\vec{y}\in Z(\hat{p})$.
\end{proof}
%%%%%%%%%%%%%%%%%%%%%%%%%%%%%%%%%%%%%%%%%%%%%%%%%
%%%%%%%%%%%%%%%%%%%%%%%%%%%%%%%%%%%%%%%%%%%%%%%%%
\section{Nonmininal hypergraphs}\label{SEC:Proof-nonminimal-hypergraphs}
In this section, we prove the results that were stated in Section~\ref{SUBSEC:nonminimal-hypergraphs}.

%%%%%%%%%%%%%%%%%%%%%%%%%%%%%%%%%%%%%%%%%%%%%%%%
\subsection{Proof for Proposition~\ref{PROP:nonminal-finite-family}}
In this subsection we prove Proposition~\ref{PROP:nonminal-finite-family}.

For every $r$-graph $\mathcal{G}$ let $\mathcal{F}_{\infty}(\mathcal{G})$ be the (infinite) family of all $r$-graphs that are not $\mathcal{G}$-colorable, i.e.
\begin{align}
\mathcal{F}_{\infty}(\mathcal{G})
= \left\{\text{$r$-graph $F$} \colon \text{ and $F$ is not $\mathcal{G}$-colorable}\right\}.\notag
\end{align}
For every positive integer $M$ define the family $\mathcal{F}_{M}(\mathcal{G})$
of $r$-graphs as
\begin{align}
\mathcal{F}_{M}(\mathcal{G})
= \left\{F \in \mathcal{F}_{\infty}(\mathcal{G})\colon v(F)\le M \right\}.\notag
\end{align}
It is clear that $\mathcal{F}_{M}(\mathcal{G})$ is nonempty if $M\ge v(\mathcal{G})+1$
since it contains all $r$-graphs $F$ with at most $M$ vertices and $K_{v(\mathcal{G})+1}\subset \partial_{r-2}F$.
It is also clear from the definition that $\mathcal{F}_{M}(\mathcal{G}) \subset \mathcal{F}_{M'}(\mathcal{G})$ for all $M'\ge M$.

It follows easily from the definition that every $\mathcal{F}_{\infty}(\mathcal{G})$-free $r$-graph is $\mathcal{G}$-colorable.
In particular, every maximum $\mathcal{F}_{\infty}(\mathcal{G})$-free $r$-graph on $n$ vertices is $\mathcal{G}$-colorable.
One of the main results in~\cite{PI14} implies the following compactness result.
For every minimum $r$-graph $\mathcal{G}$ there exists a constant $M_0 \in \mathbb{N}$ such that every maximum $\mathcal{F}_{M_0}(\mathcal{G})$-free $r$-graph on $n$ vertices is $\mathcal{G}$-colorable.
Proposition~\ref{PROP:nonminal-finite-family} actually shows that a similar compactness result holds for all $r$-graphs $\mathcal{G}$.

The following simple lemma is a special case of Lemma~8 in~\cite{PI14}.

\begin{lemma}[\cite{PI14}]\label{LEMMA:F-free-equ-F-homo-free}
For every $r$-graph $\mathcal{G}$ and every positive integer $M$ the family $\mathcal{F}_{M}(\mathcal{G})$ is blowup-invariant.
\end{lemma}

The following lemma extends Observation~2.3~(a) in~\cite{LIU19}.

\begin{lemma}\label{LEMMA:F-free-shadow-Km-free}
Let $\mathcal{G}$ be an $r$-graph.
Let $m$ denote the order of a maximum complete subgraph in $\partial_{r-2}\mathcal{G}$
and let $M = m+1+(r-2)\binom{m+1}{2}$.
Suppose that an $r$-graph $\mathcal{H}$ is $\mathcal{F}_{M}(\mathcal{G})$-free.
Then $\partial_{r-2}\mathcal{H}$ is $K_{m+1}$-free.
\end{lemma}
\begin{proof}
Suppose to the contrary that there exists a set $C \subset V(\mathcal{H})$ of size $m+1$
such that the induced subgraph of $\partial\mathcal{H}$ on $C$ is complete.
For every pair $\{u,v\}\subset C$ let $E_{u,v}\in \mathcal{H}$ be an edge that contains $\{u,v\}$.
Let $F = \{E_{u,v}\colon \{u,v\}\subset C\}$.
It is clear that $v(F) \le m+1+(r-2)\binom{m+1}{2} = M$.
So it follows from our assumption that $F$ is $\mathcal{G}$-colorable.
In other words, there exists a homomorphism $\phi\colon V(F)\to V(\mathcal{G})$ from $F$ to $\mathcal{G}$.
Since every pair $\{u,v\}\subset C$ is contained in an edge of $F$, we have $\phi(u) \neq \phi(v)$.
Therefore, $|\phi(C)| = |C| = m+1$.
However, $u\sim v$ in $\partial_{r-2}F$ implies that $\phi(u)\sim \phi(v)$ in $\partial_{r-2}\mathcal{G}$,
which means that the induced subgraph of $\partial_{r-2}\mathcal{G}$ on $\phi(C)$ is complete, a contradiction.
\end{proof}

We will prove the following statement which implies
Proposition~\ref{PROP:nonminal-finite-family}.

\begin{proposition}\label{PROP:nonminal-finite-family+}
Suppose that $r\ge 3$, $\mathcal{G}$ is an $r$-graph, and $M = v(\mathcal{G})+1+(r-2)\binom{v(\mathcal{G})+1}{2}$.
Then $\mathrm{ex}(n,\mathcal{F}_{M}(\mathcal{G})) = \max\{|\mathcal{H}|\colon
\text{$v(\mathcal{H}) = n$ and $\mathcal{H}$ is $\mathcal{G}$-colorable}\}$
for every $n\in \mathbb{N}$.
\end{proposition}
\begin{proof}
Let $\mathcal{G}$ be an $r$-graph and let $m = v(\mathcal{G})$.
Let $\mathfrak{G}$ be the collection of all $\mathcal{G}$-colorable $r$-graphs.
Define $\mathfrak{g}(n) = \max\{|\mathcal{H}|\colon \mathcal{H} \in \mathfrak{G} \text{ and }v(\mathcal{H}) = n\}$.
Let $M = m+1 + (r-2)\binom{m+1}{2}$ and $\mathcal{F} = \mathcal{F}_{M}(\mathcal{G})$.
It follows from Lemma~\ref{LEMMA:F-free-equ-F-homo-free} that $\mathcal{F}$ is blowup-invariant.
Therefore, by Theorem~\ref{THEOREM:ex(n,F)=max-blowup-of-H},
it suffices to prove the following claim.

\begin{claim}\label{CLAIM:symmetrized-stable}
Every symmetrized $\mathcal{F}$-free $r$-graph is contained in $\mathfrak{G}$.
\end{claim}
\begin{proof}
Let $\mathcal{H}$ be a symmetrized $\mathcal{F}$-free $r$-graph.
Let $C\subset V(\mathcal{H})$ be a set that contains exactly one vertex from each equivalence class of $\mathcal{H}$
and let $\mathcal{T}$ denote the induced subgraph of $\mathcal{H}$ on $C$.
It is clear from the definition of symmetrized hypergraphs that $\mathcal{H}$ is a blowup of $\mathcal{T}$.
Since $\partial_{r-2}\mathcal{T} \cong K_{|C|}$, it follows from
Lemma~\ref{LEMMA:F-free-shadow-Km-free} that $|C|\le m$.
Since $\mathcal{T}\not\in \mathcal{F}$, we know that $\mathcal{T}$ is $\mathcal{G}$-colorable.
This implies that $\mathcal{H}$ is also $\mathcal{G}$-colorable.
Therefore, $\mathcal{H} \in \mathfrak{G}$.
\end{proof}%CLAIM
This completes the proof of Proposition~\ref{PROP:nonminal-finite-family+}.
\end{proof}%PROP

%%%%%%%%%%%%%%%%%%%%%%%%%%%%%%%%%%%%%%%%%%%%
\subsection{Proofs for Proposition~\ref{PROP:nonminimal-corss-blowup} and Theorem~\ref{THM:nonminimal-general-turan-stability}}
In this subsection we prove Proposition~\ref{PROP:nonminimal-corss-blowup} and Theorem~\ref{THM:nonminimal-general-turan-stability}.

Let us assume that the vertex set of $\mathcal{G}$ is $\{v_1,v_2,\ldots, v_m\}$.
Let $v_1',v_2'$ denote the clones of $v_1,v_2$ in $\mathcal{G}\boxplus\{v_1,v_2\}$, respectively.
Let $U_1 = \{v_1,v_2',v_3,\ldots, v_m\}$ and $U_2 = \{v_1',v_2,v_3,\ldots, v_m\}$.
It is easy to see from the definition that the induced subgraph of $\mathcal{G}\boxplus\{v_1,v_2\}$ on
$U_1$ and $U_2$ are both isomorphic to $\mathcal{G}$, this proves~(a).
Statement~(b) follows easily from the definition of crossed blowup.
So we may focus on~(c).

We will prove the following statement which implies Proposition~\ref{PROP:nonminimal-corss-blowup}~(c).

\begin{proposition}\label{PROP:nonminimal-corss-blowup+}
Suppose that $\mathcal{G}$ is an $m$-vertex $3$-graph and $\{v_1,v_2\}\subset \mathcal{G}$ is a pair of vertices with $d(v_1,v_2)\ge 2$.
Suppose that $\{v_1,v_2\}$ is symmetric in $\mathcal{G}$.
Then $\lambda(\mathcal{G}\boxplus\{v_1,v_2\}) = \lambda(\mathcal{G})$.
Moreover, for every $(z_1,\ldots,z_{m})\in Z(\mathcal{G})$
we have
$$L\left((\bar{z},0,0,\bar{z},z_3,\ldots,z_{m}), (0,\bar{z},\bar{z},0,z_3,\ldots,z_{m})\right)
\subset Z(\mathcal{G}\boxplus\{v_1,v_2\}),$$
where $\bar{z} = (z_1+z_2)/2$.
\end{proposition}
\begin{proof}
Let $k = d(v_1,v_2)$ and assume that $N(v_1,v_2)=\{v_{i_1}, \ldots, v_{i_k}\}$.
Let $W = \{v_3,\ldots,v_m\}$. Define polynomials $p_1, \ldots, p_5$ as
\begin{align}
p_{1} & = \sum_{v_iv_jv_k\in \mathcal{G}[W]}X_iX_jX_k, \notag\\
p_{2} & = \sum_{v_iv_j\in L_{\mathcal{G}}(v_1)[W]}X_iX_j = \sum_{v_iv_j\in L_{\mathcal{G}}(v_2)[W]}X_iX_j,  \notag\\
p_{3} & = X_{i_1}+ \cdots + X_{i_k}, \quad
p_{4}  = X_{i_1}+ \cdots + X_{i_{k-1}}, \quad\text{and}\quad
p_{5} = X_{i_k}. \notag
\end{align}
It is easy to see that $p_{\mathcal{G}}(X_1, \ldots, X_{m})
 = p_1 + p_2(X_1 + X_2)+ p_{3}X_1X_2$.
%\begin{align}
%p_{\mathcal{G}}(X_1, \ldots, X_{m})
% = p_1 + p_2(X_1 + X_2)+ p_{3}X_1X_2.\notag
%\end{align}
On the other hand, it follows from the definition of crossed blowup that
\begin{align}
& p_{\mathcal{G}\boxplus\{v_1,v_2\}}(X_1,X_1',X_2,X_2',X_3 \ldots, X_{m}) \notag\\
& = p_1 + p_2(X_1 + X_2+X_1'+X_2')+ p_{4}(X_1+X_1')(X_2+X_2')+p_5(X_1+X_2)(X_1'+X_2').  \notag
\end{align}
So by Proposition~\ref{PROP:multilinear-polynomial} we have
$$L\left((\bar{z},0,0,\bar{z},z_3,\ldots,z_{m}), (0,\bar{z},\bar{z},0,z_3,\ldots,z_{m})\right)
\subset Z(\mathcal{G}\boxplus\{v_1,v_2\}).$$
\end{proof}

Now we are ready to prove Theorem~\ref{THM:nonminimal-general-turan-stability}.

\begin{proof}[Proof of Theorem~\ref{THM:nonminimal-general-turan-stability}]
If $d_{\mathcal{G}}(v_1,v_2) \ge 2$, then we let $\mathcal{G}' = \mathcal{G}\boxplus \{v_1,v_2\}$, and
it follows from Proposition~\ref{PROP:nonminimal-corss-blowup}~(c) that $\lambda(\mathcal{G}') = \lambda(\mathcal{G})$.

If $d_{\mathcal{G}}(v_1,v_2) = 1$, then we let $\{w\}= N_{\mathcal{G}}(v_1, v_2)$
and let $\mathcal{G}'$ be obtained as follows.
First, we take two vertex-disjoint copies of $\mathcal{G}$ and identify them on the set $V(\mathcal{G})\setminus \{w\}$
(i.e. we blowup $w$ into two vertices).
Denote by $\mathcal{G}^2$ the resulting $3$-graph.
For example, if $\mathcal{G} = \{125,345\}$ and $\{v_1,v_2\} = \{1,2\}$, then $\mathcal{G}^2 = \{125,345,126,346\}$.
Let $\mathcal{G}' = \mathcal{G}^2 \boxplus\{v_1,v_2\}$.
Since $\mathcal{G}^2$ is $\mathcal{G}$-colorable and $\mathcal{G}\subset \mathcal{G}^2$,
we have $\lambda(\mathcal{G}^2) = \lambda(\mathcal{G})$.
Notice that $\{v_1,v_2\}$ is still symmetric in $\mathcal{G}^2$,
so it follows from Proposition~\ref{PROP:nonminimal-corss-blowup}~(c) that
$\lambda(\mathcal{G}') = \lambda(\mathcal{G}^2) = \lambda(\mathcal{G})$.

Let $m = v(\mathcal{G}')$ and $M = m+1 + (r-2)\binom{m+1}{2}$.
Let $\mathcal{F} = \mathcal{F}_{M}(\mathcal{G}')$.
It follows from Proposition~\ref{PROP:nonminal-finite-family+} that
$\mathrm{ex}(n,\mathcal{F}) = \max\{|\mathcal{H}|\colon  \text{$v(\mathcal{H})=n$ and $\mathcal{H}$ is $\mathcal{G}'$-colorable}\}$.
Then by Lemma~\ref{LEMMA:|H|<=lambda-T-n^r}, we have  $\pi(\mathcal{F}) \le 6\lambda(\mathcal{G}') = 6\lambda(\mathcal{G})$.
The converse of this inequality follows from the fact that every blowup of $\mathcal{G}'$ is $\mathcal{F}$-free.
Therefore, we have $\pi(\mathcal{F}) = 6\lambda(\mathcal{G})$.

Now we prove that $\xi(\mathcal{F})=\infty$.
Suppose to the contrary that $\xi(\mathcal{F})=t$ for some $t\in \mathbb{N}^+$.
For every $n \in \mathbb{N}$ let $\mathcal{S}_1(n),\ldots, \mathcal{S}_t(n)$ be the $n$-vertex $3$-graphs that witness
the $t$-stability of $\mathcal{F}$.
For every $i\in [t]$ and $n\in \mathbb{N}$
let $\rho_{i,n} =  \rho(\mathcal{G}',\mathcal{S}_i(n)) = N(\mathcal{G}',\mathcal{S}_i(n))/\binom{n}{m}$ ,
where $N(\mathcal{G}',\mathcal{S}_i(n))$ is the number of copies of $\mathcal{G}'$ in $\mathcal{S}_i(n)$.

We will get a contradiction by first showing that for all sufficiently large $n$
there exists a blowup $\widehat{\mathcal{G}}$ of $\mathcal{G}$ on $n$ vertices whose edge density is close to $\pi(\mathcal{F})$ but
whose shadow density is far from $\rho_{i,n}$ for every $i\in [t]$.
Then we argue that since every shadow edge is contained $\Omega(n)$ edges in $\widehat{\mathcal{G}}$,
the edit-distance of $\widehat{\mathcal{G}}$ and $\mathcal{S}_i(n)$ is $\Omega(n^3)$ for all $i\in [t]$.
This contradicts the definition of $t$-stable.

Notice that $v(\mathcal{G}) = m-2$ if $d_{\mathcal{G}}(v_1, v_2)\ge 2$ and $v(\mathcal{G}^2) = m-2$ if $d_{\mathcal{G}}(v_1, v_2) = 1$.
Let $(z_1,z_2,z_3,\ldots,z_{m-2})\in Z(\mathcal{G})$ (or $Z(\mathcal{G}^2)$ if $d_{\mathcal{G}}(v_1,v_2) = 1$)
be a vector such that $\zeta := \min\{z_1,z_2,z_3,\ldots,z_{m-2}\}$ is maximized.
By Lemma~\ref{LEMMA:max-p-symmetric}, we may assume that $z_1 = z_2$
since otherwise we can replace $z_1$ and $z_2$ by $(z_1+z_2)/2$.

\begin{claim}\label{CLAIM:zeta>0}
%There exists $(z_1,z_2,z_3,\ldots,z_{m-2})\in Z(\mathcal{G})$ (or $Z(\mathcal{G}^2)$ if $d_{\mathcal{G}}(v_1,v_2) = 1$)
%such that $\min\{z_1,z_2,z_3,\ldots,z_{m-2}\} > 0$.
We have $\zeta>0$.
\end{claim}
\begin{proof}
If $d(v_1,v_2) \ge 2$, then the minimality of $\mathcal{G}$ implies that
every vector in $Z(\mathcal{G})$ has only positive coordinates.
Therefore, $\zeta > 0$.

Suppose that $d(v_1,v_2) = 1$.
Without loss of generality let us assume that $N_{\mathcal{G}}(v_1,v_2) = \{v_3\}$
and $N_{\mathcal{G}^2}(v_1,v_2) = \{v_3,v_4\}$.
Since $v_4$ is a clone of $v_3$ in $\mathcal{G}^2$,
$(z_1,z_2,z_3,z_4,\ldots,z_{m-2})\in Z(\mathcal{G}^2)$ implies
that $(z_1,z_2,(z_3+z_4)/2,(z_3+z_4)/2,\ldots,z_{m-2})\in Z(\mathcal{G}^2)$.
For the same reason, $(z_1,z_2,z_3,z_4,\ldots,z_{m-2})\in Z(\mathcal{G}^2)$ implies that
$(z_1,z_2,z_3+z_4,z_5,\ldots,z_{m-2})\in Z(\mathcal{G})$.
The minimality of $\mathcal{G}$ implies that each coordinate of $(z_1,z_2,z_3+z_4,z_5,\ldots,z_{m-2})$ is positive.
Therefore, every coordinate in $(z_1,z_2,(z_3+z_4)/2,(z_3+z_4)/2,\ldots,z_{m-2})$ is positive,
and hence, $\zeta > 0$.
\end{proof}

Let $\vec{x}_{\alpha} = \alpha\cdot (z_1,0,0,z_2,z_3,\ldots,z_{m-2})+(1-\alpha)\cdot (0,z_1,z_2,0,z_3,\ldots,z_{m-2})$
for $\alpha \in [0,1]$.
Let $\widehat{\mathcal{G}}_{\alpha}(n) = \mathcal{G}'[V_1,V_1',V_2,V_2',V_3,\cdots,V_{m-2}]$ be a blowup of $\mathcal{G}'$,
where $|V_1| = \lfloor \alpha z_1 n \rfloor$, $|V_1'| = \lfloor (1-\alpha) z_1 n \rfloor$,
$|V_2| = \lfloor \alpha z_2 n \rfloor$, $|V_2'| = \lfloor (1-\alpha) z_2 n\rfloor$,
and $|V_i| = \lfloor z_i n \rfloor$ for $i\in [3,m-2]$.
Let $W = \{v_3,\ldots, v_{m-2}\}$.
Then
\begin{align}
|\partial\widehat{\mathcal{G}}_{\alpha}(n)|
& = \sum_{v_iv_j\in\partial\mathcal{G}\cap\binom{W}{2}}|V_i||V_j|
    + \sum_{v_i\in N_{\mathcal{G}}(v_1)}|V_i|\left(|V_1|+|V_1|'\right)
    + \sum_{v_i\in N_{\mathcal{G}}(v_2)}|V_i|\left(|V_2|+|V_2|'\right) \notag\\
&  \quad +\left(|V_1|+|V_1'|\right)\left(|V_2|+|V_2'|\right)+|V_1||V_1'|+|V_2||V_2'|. \notag
\end{align}
Taking the limit we obtain
\begin{align}
\lim_{n\to \infty}\rho\left(\partial\widehat{\mathcal{G}}_{\alpha}(n)\right)
& = 2\alpha(1-\alpha)(z_1+z_2) + 2C, \notag
\end{align}
where
\begin{align}
C
& = \sum_{v_iv_j\in\partial\mathcal{G}\cap\binom{W}{2}}z_iz_j
    + \sum_{v_i\in N_{\mathcal{G}}(v_1)}z_iz_1
    + \sum_{v_i\in N_{\mathcal{G}}(v_2)}z_iz_2 + z_1z_2 \notag
\end{align}
is independent of $\alpha$.

For $i\in [t+1]$ let $\alpha_i = \frac{i}{2(t+1)}$.
Since $z_1+z_2 > 0$ and $\alpha(1-\alpha)$ is increasing in $\alpha$ for $\alpha \in [0,1/2]$,
there exists $\epsilon>0$ and $N_0$ such that
\begin{align}
\left|\partial\widehat{\mathcal{G}}_{\alpha_j}(n)\right|
        -\left|\partial\widehat{\mathcal{G}}_{\alpha_i}(n)\right|
> \frac{6\epsilon n^2}{\zeta} \notag
\end{align}
for all $n \ge N_0$ and for all $1\le i < j \le t+1$.
Fix such an $\epsilon>0$.
By the definition of $t$-stable there exists $\delta= \delta(\epsilon)$ and $N_1 = N_1(\epsilon)$
such that   every $n$-vertex $\mathcal{F}$-free $3$-graph $\mathcal{H}$
with $|\mathcal{H}| \ge (1-\delta)\mathrm{ex}(n,\mathcal{F})$,
satisfies $|\mathcal{H}\triangle \mathcal{S}_i(n)| \le \epsilon n^3$ for some $i\in [t]$.
Since $\lim_{n\to \infty}\rho\left(\widehat{\mathcal{G}}_{\alpha}(n)\right) = \pi(\mathcal{F})$, there exists $N_2$ such that
$\left| \widehat{\mathcal{G}}_{\alpha}(n)\right| \ge (1-\delta)\mathrm{ex}(n,\mathcal{F})$ for all $\alpha\in[0,1]$ and $n\ge N_2$.
Therefore,
for every $j\in [t+1]$ there exists $i\in [t]$ such that
$|\widehat{\mathcal{G}}_{\alpha_j}(n)\triangle\mathcal{S}_i(n)|\le \epsilon n^3$
for all $n\ge \max\{N_0,N_1,N_2\}$.
By the Pigeonhole principle, there exists a pair $\{j,k\}$ such that
$|\widehat{\mathcal{G}}_{\alpha_j}(n)\triangle\mathcal{S}_i(n)|\le \epsilon n^3$
and $|\widehat{\mathcal{G}}_{\alpha_k}(n)\triangle\mathcal{S}_i(n)|\le \epsilon n^3$.
By the triangle inequality, we have $|\widehat{\mathcal{G}}_{\alpha_j}(n)\triangle\widehat{\mathcal{G}}_{\alpha_k}(n)|\le 2\epsilon n^3$.

Without loss of generality we may assume that $\alpha_{j}< \alpha_k$.
Observe that every edge $e\in \partial\widehat{\mathcal{G}}_{\alpha_j}(n)\triangle\partial\widehat{\mathcal{G}}_{\alpha_k}(n)$
is contained in the set $V_1\cup V_1'\cup V_2\cup V_2'$.
Let $\{\ell_1,\ell_2\}\subset [3,m-2]$ be the pair such that $N_{\mathcal{G}^2}=\{\ell_1,\ell_2\}$.
Then we have $N_{\widehat{\mathcal{G}}_{\alpha}(n)}(e) \ge \min\{|V_{\ell_1}|,|V_{\ell_2}|\} \ge \lfloor \zeta n\rfloor> \zeta n/2$.
Therefore, in order to transform $\widehat{\mathcal{G}}_{\alpha_j}(n)$ into $\widehat{\mathcal{G}}_{\alpha_k}(n)$
we need to remove at least
\begin{align}
\left(|\partial\widehat{\mathcal{G}}_{\alpha_k}(n)|-|\partial\widehat{\mathcal{G}}_{\alpha_j}(n)|\right)\frac{\zeta n}{2}
> \frac{6\epsilon n^2}{\zeta} \frac{\zeta n}{2} > 2\epsilon n^3 \notag
\end{align}
edges, a contradiction.
Therefore, $\xi(\mathcal{F}) = \infty$.
\end{proof}
%%%%%%%%%%%%%%%%%%%%%%%%%%%%%%%%%%%%%%%%%%%%%%%%%
%%%%%%%%%%%%%%%%%%%%%%%%%%%%%%%%%%%%%%%%%%%%%%%%%
\section{Proof of Theorem~\ref{THM:main-sec1.1}~(a),~(b), and~(c)}\label{SEC:proof-turan-number}
In this section we prove Theorem~\ref{THM:main-sec1.1}~(a),~(b), and~(c).
First let us define $\Gamma_t$ and $\mathcal{F}_t$ that were mentioned in Theorem~\ref{THM:main-sec1.1}.
For convenience, we will use $t+2$ instead of $t$ in the rest of this paper.

\begin{definition}
Let $t\ge 1$ be an integer.
\begin{enumerate}[label=(\alph*)]
\item Let
\begin{align}
\Gamma_{t+2} =
\begin{cases}
\{134, 234\}\boxplus \{3,4\} & \text{if } t=1, \\
K_{t+2}^{3}\boxplus \{t+1,t+2\} & \text{if } t\geq 2.
\end{cases} \notag
\end{align}
\item Let $\mathfrak{\Gamma}_{t+2}$ be the collection of all $\Gamma_{t+2}$-colorable $3$-graphs.
\item Let $\gamma_{t+2}(n) = \max\{|\mathcal{H}|\colon v(\mathcal{H}) = n \text{ and } \mathcal{H}\in\mathfrak{\Gamma}_{t+2}\}$.
\item Let $\mathcal{F}_{t+2} = \left\{F\colon v(F)\le 4(t+4)^2 \text{ and } F\not\in \mathfrak{\Gamma}_{t+2}\right\}$.
\end{enumerate}
\end{definition}
\noindent\textbf{Remark.}
In the rest of the paper, we always assume that for $t\ge 2$ the vertex set of $\Gamma_{t+2}$ is $[t+4]$, and $t+3, t+4$ are clones of $t+1, t+2$, respectively.

It follows from Proposition~\ref{PROP:nonminimal-corss-blowup} and Lemma~\ref{LEMMA:|H|<=lambda-T-n^r}
that $\gamma_{t+2}(n) \sim \lambda(K_{t+2}^3)n^3 = \frac{t(t+1)}{6(t+2)^2}n^3$
(for $t=1$ we also used $\lambda(\{134,234\}) = \lambda(K_{3}^3)$).

%Theorem~\ref{THM:main-sec1.1}~(a)
%follows easily from Proposition~\ref{PROP:nonminal-finite-family}.
Theorem~\ref{THM:main-sec1.1}~(a) and~(c) follow easily from Theorem~\ref{THM:nonminimal-general-turan-stability}.
So we just need to prove Theorem~\ref{THM:main-sec1.1}~(b).

\begin{proof}[Proof of Theorem~\ref{THM:main-sec1.1}~(b)]
Let $t\ge 2$ and $n$ be an integer satisfying $(t+2)\mid n$.
It is easy to see that $\mathrm{ex}(n,\mathcal{F}_{t+2}) = \frac{t(t+1)}{6(t+2)^2}n^3$.
Let $\mathcal{H} = \Gamma_{t+2}[V_1,\cdots,V_{t+4}]$ be a blowup of $\Gamma_{t+2}$ with
$|V_1| = \cdots = |V_t| = \frac{n}{t+2}$, $|V_{t+1}| = |V_{t+4}| = \frac{\alpha n}{t+2}$,
and $|V_{t+2}| = |V_{t+3}| = \frac{(1-\alpha) n}{t+2}$, where $\alpha \in [0,1/2]$ satisfies that $\frac{\alpha n}{t+2}\in \mathbb{N}$.
It is easy to see that $|\mathcal{H}| = \frac{t(t+1)}{6(t+2)^2}n^3$.
So it suffices to show that the set $S = \left\{\alpha \in [0,1/2]\colon \frac{\alpha n}{t+2}\in \mathbb{N}\right\}$
has size at least $\frac{n}{2(t+2)}$.
Indeed, suppose that $\alpha = \frac{p}{q}\in [0,1/2]$, where $p\le q/2$ are integers that are coprime.
Notice that $\frac{p}{q}\frac{n}{t+2} \in \mathbb{N}$ if $q\mid \frac{n}{t+2}$.
Since $p$ is coprime with $q$ iff $q-p$ is coprime with $q$,
we have $|\{p\colon (p,q)=1,\ p\le q/2\}| \ge \varphi(q)/2$, where $\varphi(q)$ is the Euler totient function
that denotes the number of positive integers $p\le q$ that are coprime with $q$.
Therefore,  $|S| \ge \frac{1}{2}\sum_{q\mid \frac{n}{t+2}}\varphi(q)$.
It follows from a well known result of Gauss that $\sum_{q\mid \frac{n}{t+2}}\varphi(q) = \frac{n}{t+2}$.
Therefore, $|S| \ge \frac{n}{2(t+2)}$.
This implies that there are at least $\frac{n}{2(t+2)}$ nonisomorphic extremal $\mathcal{F}_t$-free $3$-graphs on $n$ vertices.
The proof for the case $t = 1$ is similar, so we omit it here.
\end{proof}

%%%%%%%%%%%%%%%%%%%%%%%%%%%%%%%%%%%%%%%%%%%%%%%%%
%%%%%%%%%%%%%%%%%%%%%%%%%%%%%%%%%%%%%%%%%%%%%%%%%
\section{Stability}\label{SEC:proof-stability}
\subsection{Preparations}\label{SUBSEC:prepare}
In this section we present some useful theorems and lemmas that will be used later
in the proof of Theorem~\ref{THM:main-sec1.1}~(d).

\begin{definition}[Vertex-extendibility]\label{DFN:vertex-extendable}
Let $\mathcal{F}$ be a family of $r$-graphs and let $\mathfrak{H}$ be a class of $\mathcal{F}$-free $r$-graphs.
We say that $\mathcal{F}$ is \emph{vertex-extendable}
with respect to $\mathfrak{H}$ if there exist $\zeta>0$ and $N_0\in\mathbb{N}$ such that
for every $\mathcal{F}$-free $r$-graph $\mathcal{H}$ on $n\ge N_0$
vertices satisfying $\delta(\mathcal{H})\ge \bigl(\pi(\mathcal{F})/(r-1)!-\zeta\bigr)n^{r-1}$
the following holds: if $\mathcal{H}-v$ is a subgraph of a member of $\mathfrak{H}$ for some
vertex $v\in V(\mathcal{H})$, then $\mathcal{H}$ is a subgraph of a member of $\mathfrak{H}$ as well.
\end{definition}

In~\cite{LMR2}, the authors developed a machinery that reduces the proof of stability of certain families $\mathcal{F}$
to the simpler question of checking that an $\mathcal{F}$-free hypergraph $\mathcal{H}$ with large minimum degree is vertex-extendable.

\begin{theorem}[\cite{LMR2}]\label{THM:Psi-trick:G-extendable-implies-degree-stability}
Suppose that $\mathcal{F}$ is a blowup-invariant nondegenerate family of $r$-graphs and that
$\mathfrak{H}$ is a hereditary class of $\mathcal{F}$-free $r$-graphs.
If $\mathfrak{H}$ contains all symmetrized $\mathcal{F}$-free $r$-graphs and
$\mathcal{F}$ is vertex-extendable with respect to $\mathfrak{H}$, then the following statement holds.
There exist $\epsilon>0$ and $N_0$ such that every $\mathcal{F}$-free $r$-graph on $n\ge N_0$ vertices with minimum
degree at least $\left(\pi(\mathcal{F})/(r-1)!-\epsilon\right)n^{r-1}$ is contained in $\mathfrak{H}$.
\end{theorem}

The following lemma will be used extensively in the proof of Theorem~\ref{THM:main-sec1.1}~(d).

\begin{lemma}[see e.g. \cite{LMR1}]\label{LEMMA:greedily-embedding-Gi}
Fix a real $\eta \in (0, 1)$ and integers $m, n\ge 1$.
Let $\mathcal{G}$ be a $3$-graph with vertex set~$[m]$ and let $\mathcal{H}$ be a further $3$-graph
with $v(\mathcal{H})=n$.
Consider a vertex partition $V(\mathcal{H}) = \bigcup_{i\in[m]}V_i$ and the associated
blow-up $\widehat{\mathcal{G}} = \mathcal{G}[V_1,\ldots,V_{m}]$ of $\mathcal{G}$.
If two sets $T \subseteq [m]$ and $S\subseteq \bigcup_{j\not\in T}V_j$
have the properties
\begin{enumerate}[label=(\alph*)]
\item\label{it:47a} $|V_{j}| \ge (|S|+1)|T|\eta^{1/3} n$  for all $j \in T$,
\item\label{it:47b} $|\mathcal{H}[V_{j_1},V_{j_2},V_{j_3}]| \ge |\widehat{\mathcal{G}}[V_{j_1},V_{j_2},V_{j_3}]|
		- \eta n^3$ for all $\{j_1,j_2,j_3\} \in \binom{T}{3}$, and
\item\label{it:47c} $|L_{\mathcal{H}}(v)[V_{j_1},V_{j_2}]| \ge |L_{\widehat{\mathcal{G}}}(v)[V_{j_1},V_{j_2}]|
		- \eta n^2$ for all $v\in S$ and $\{j_1,j_2\} \in \binom{T}{2}$,
\end{enumerate}
then there exists a selection of vertices $u_j\in V_j$ for all $j\in [T]$
such that $U = \{u_j\colon j\in T\}$ satisfies
$\widehat{\mathcal{G}}[U] \subseteq \mathcal{H}[U]$ and
$L_{\widehat{\mathcal{G}}}(v)[U] \subseteq L_{\mathcal{H}}(v)[U]$ for all $v\in S$.
In particular, if $\mathcal{H} \subseteq \widehat{\mathcal{G}}$,
then $\widehat{\mathcal{G}}[U] = \mathcal{H}[U]$ and
$L_{\widehat{\mathcal{G}}}(v)[U] = L_{\mathcal{H}}(v)[U]$ for all $v\in S$.
\end{lemma}

Next we prove some basic properties about $\Gamma_{t+2}$.
%Define the following graphs.
%\begin{align}
%B_1 & = \binom{[t-1]}{2}\cup \left\{ij\colon i\in [t-1], j\in [t,t+3]\right\}
%        \cup \left\{ij\colon i\in \{t,t+1\}, j\in \{t+2,t+3\}\right\},      \notag\\
%B_2 & = \binom{[t]}{2}\cup \left\{ij\colon i\in [t], j\in \{t+1,t+2\}\right\}
%        \cup \left\{ij\colon i=t, j\in \{t+1,t+3\}\right\}, \quad\text{and} \notag\\
%B_2' & = \binom{[t]}{2}\cup \left\{ij\colon i\in [t], j = t+1\right\}
%        \cup \left\{t,t+2\right\}.                                          \notag
%\end{align}

\begin{observation}\label{OBS:structure-link}
The following statements hold.
\begin{enumerate}[label=(\alph*)]
\item If a set $U\subset [t+4]$ satisfies that $\Gamma_{t+2}[U]\cong K_{t+2}^3$,
        then either $U = [t]\cup \{t+1,t+4\}$ or $U = [t]\cup \{t+2,t+3\}$.
\item %For every $i\in [t]$ the link $L_{\Gamma_{t+2}}(i)$ is isomorphic to $B_1$.
        We have $L_{\Gamma_{t+2}}(1) \cong \cdots \cong L_{\Gamma_{t+2}}(t)$, and $L_{\Gamma_{t+2}}(t+1) \cong \cdots \cong L_{\Gamma_{t+2}}(t+4).$
%\item For every $i\in [t+1,t+4]$ the link $L_{\Gamma_{t+2}}(i)$ is isomorphic to $B_2$.
\item We have
\begin{align}
d_{\Gamma_{t+2}}(i,j) =
\begin{cases}
t+2 & \text{if } \{i,j\} \in \binom{[t]}{2}, \\
t+1 & \text{if } i\in [t] \text{ and } j\in [t+1,t+4], \\
t   & \text{if } \{i,j\} \in \{\{t+1,t+4\},\{t+2,t+3\}\}, \\
t-1 & \text{if } \{i,j\} \in \{\{t+1,t+3\},\{t+2,t+4\}\}, \\
1   & \text{if } \{i,j\} \in \{\{t+1,t+2\},\{t+3,t+4\}\}.
\end{cases} \notag
\end{align}
\end{enumerate}
\end{observation}

Using Observation~\ref{OBS:structure-link} (a) we prove the following lemma.

\begin{lemma}\label{LEMMA:two-K_t+2}
Suppose that $\mathcal{G}$ is a blowup of $\Gamma_{t+2}$ and $U,U'\subset V(\mathcal{G})$
are two vertex sets of size $t+2$ with $|U\cap U'| = t+1$.
If $\mathcal{G}[U] \cong \mathcal{G}[U'] \cong K_{t+2}^3$,
then the set $U\triangle U'$ is not contained in any edge of $\mathcal{G}$.
\end{lemma}
\begin{proof}
Suppose that $U=\{u_1,\ldots,u_{t+1},u\}$ and $U'=\{u_1,\ldots,u_{t+1},u'\}$.
Let $\phi\colon U\cup U' \to V(\Gamma_{t+2})$ be a map such that $\phi|_{U}$ and $\phi|_{U'}$
are homomorphisms from $\mathcal{G}[U]$ and $\mathcal{G}[U']$ to $\Gamma_{t+2}$, respectively.
Since $|U\cap U'| = t+1$, it follows from Observation~\ref{OBS:structure-link} (a)
that either $\phi(U) = \phi(U') = [t]\cup \{t+1,t+4\}$ or $\phi(U) = \phi(U') = [t]\cup \{t+2,t+3\}$.
In either case, we have $\phi(u) = \phi(u')$, which implies that $u$ and $u'$ lie in the same part of $\mathcal{G}$.
So $\{u,u'\}$ is not contained in any edge of $\mathcal{G}$.
\end{proof}

The following lemma follows easily from Observation~\ref{OBS:structure-link} (c).

\begin{lemma}\label{LEMMA:gamma-t-automorphism}
Suppose that $\phi\colon V(\Gamma_{t+2})\to V(\Gamma_{t+2})$ is a homomorphism from $\Gamma_{t+2}$ to $\Gamma_{t+2}$.
Then $\phi$ is bijective,
$\phi([t]) = [t]$, and $$\left\{ \phi(\{t+1,t+4\}),\phi(\{t+2,t+3\}) \right\} = \left\{ \{t+1,t+4\},\{t+2,t+3\} \right\}.$$
If $t\ge 3$, then additionally, $\phi([t-1]) = [t-1]$.
%
%Then $\phi$ is bijective and the following statements hold.
%\begin{enumerate}[label=(\alph*)]
%\item If $t = 2$, then $\phi([t]) = [t]$, and
%        $$\left\{ \phi(\{t+1,t+4\}),\phi(\{t+2,t+3\}) \right\} = \left\{ \{t+1,t+4\},\{t+2,t+3\} \right\}.$$
%\item If $t\ge 3$, then $\phi([t-1]) = [t-1]$, $\phi(t) = t$, and
%        $$\left\{ \phi(\{t+1,t+4\}),\phi(\{t+2,t+3\}) \right\} = \left\{ \{t+1,t+4\},\{t+2,t+3\} \right\}.$$
%\end{enumerate}
\end{lemma}

Using Observation~\ref{OBS:structure-link} (c) we further prove the following lemma.

\begin{lemma}\label{LEMMA:gamma-t-partial-embedding}
Suppose that $\phi\colon V(\Gamma_{t+2})\setminus \{t+3\} \to V(\Gamma_{t+2})$
is a homomorphism from $\Gamma_{t+2}\setminus \{t+3\}$ to $\Gamma_{t+2}$.
Then $\phi$ is injective, $\phi([t]) = [t]$, $\phi(\{t+1,t+2,t+4\})\subset [t+1,t+4]$,
and $\{\phi(t+1),\phi(t+4)\}\in \{\{t+1,t+4\},\{t+2,t+3\}\}$.
If $t\ge 3$, then additionally, $\phi([t-1]) = [t-1]$.
%
%and the following statements hold.
%\begin{enumerate}[label=(\alph*)]
%\item If $t = 2$, then $\phi([t]) = [t]$, $\phi(\{t+1,t+2,t+4\})\subset [t+1,t+4]$,
%        and $\{\phi(t+1),\phi(t+4)\}\in \{\{t+1,t+4\},\{t+2,t+3\}\}$.
%\item If $t\ge 3$, then $\phi([t-1]) = [t-1]$, $\phi(t) = t$, $\phi(\{t+1,t+2,t+4\})\subset [t+1,t+4]$,
%        and $\{\phi(t+1),\phi(t+4)\}\in \{\{t+1,t+4\},\{t+2,t+3\}\}$.
%\end{enumerate}
\end{lemma}
\begin{proof}
Let $F = \Gamma_{t+2}\setminus\{t+3\}$.
Since $\phi$ is an embedding, we have $d_{F}(i,j) \le d_{\Gamma_{t+2}}(\phi(i),\phi(j))$ for all pairs of vertices in $F$.
Observe that
\begin{align}
d_{F}(i,j) =
\begin{cases}
t+1 & \text{if } \{i,j\} \in \binom{[t]}{2} \text{ or } (i,j)\in [t-1]\times\{t+4\} \text{ or } (i,j) = (t,t+1), \\
t & \text{if } (i,j)\in [t-1] \times \{t+1,t+2\} \text{ or } (i,j)\in \{t\}\times \{t+2,t+4\}\\
 & \text{or } \{i,j\} =\{t+1,t+4\}, \\
t-1 & \text{if } \{i,j\} =\{t+2,t+4\}, \\
1   & \text{if } \{i,j\} =\{t+1,t+2\}.
\end{cases} \notag
\end{align}
Since the induced subgraph of $F$ on $[t]\cup \{t+1,t+4\}$ is isomorphic to $K_{t+2}^{3}$,
it follows from from Observation~\ref{OBS:structure-link} (a) that
$\phi([t]\cup \{t+1,t+4\}) \in \{[t]\cup \{t+1,t+4\},[t]\cup \{t+2,t+3\}\}$.
By symmetry we may assume that $\phi([t]\cup \{t+1,t+4\}) = [t]\cup \{t+1,t+4\}$.
Then we have $\phi(t+2) \in [t+1,t+4]\setminus \phi([t]\cup \{t+1,t+4\}) = \{t+2,t+3\}$
By symmetry we may assume that $\phi(t+2) = t+2$.
Since $d_{\Gamma_{t+2}}(\phi(t+2),t+1) = 1$ and $d_{\Gamma_{t+2}}(\phi(t+2),t+4) = t-1$,
we have $\phi(t+1) = t+1$ and $\phi(t+4) = t+4$
(otherwise we would have $\phi(i)\in \{t+1,t+4\}$ for some $i\in [t]$,
but this implies that $t = d_{F}(i,t+2) \le d_{\Gamma_{t+2}}(\phi(i),\phi(t+2)) \le t-1$, a contradiction).
Consequently, $\phi([t]) = [t]$.

Suppose that $t\ge 3$.
Let $S = \{t+1,t+2,t+4\}$
Observe that $L_{F}(i)[S] = L_{F}(j)[S] \neq L_{F}(t)(S)$ and
$L_{\Gamma_{t+2}}(i)[\phi(S)] = L_{\Gamma_{t+2}}(j)[\phi(S)] \neq L_{\Gamma_{t+2}}(t)(\phi(S))$ and for all $i,j\in [t-1]$.
So, we should have $\phi([t-1]) = [t-1]$ and $\phi(t) = \phi(t)$.
\end{proof}

The following lemma characterizes the vectors $(x_1,\ldots, x_{t+4}) \in \Delta_{t+3}$ for which
the value $p_{\Gamma_{t+2}}(x_1,\ldots, x_{t+4})$ is close to the maximum.

\begin{lemma}\label{LEMMA:Lagrangian-stability}
For every integer $t\ge 2$ and every real number $\epsilon>0$ there exists a constant $\delta>0$ such that
if a vector $\vec{x}\in \Delta_{t+3}$ satisfies $p_{\Gamma_{t+2}}(x_1,\ldots, x_{t+4}) \ge \lambda(K_{t+2}^3) - \delta$,
then there exists a real number $\alpha \in [0,1]$ such that
\begin{align}\label{equ:weight-stable}
x_i =
\begin{cases}
\frac{1}{t+2} \pm \epsilon & \text{if } i\in [t], \\
\frac{\alpha}{t+2} \pm \epsilon & \text{if } i\in \{t+1, t+4\},\\
\frac{1-\alpha}{t+2} \pm \epsilon & \text{if } i\in\{t+2, t+3\}.
\end{cases}
\end{align}
\end{lemma}
\begin{proof}
We will use
\begin{align}\label{INEQU:Maclaurin-r=2}
%xy = \left(\frac{x+y}{2}\right)^2 - \frac{1}{2}\left(\left(x-\frac{x+y}{2}\right)^2+\left(y-\frac{x+y}{2}\right)^2\right).
xy = \left(\frac{x+y}{2}\right)^2 -  \left(\frac{x-y}{2}\right)^2,
\end{align}
and the following inequalities (see Lemma 2.2 in~\cite{LMR2}),
\begin{align}\label{INEQU:Maclaurin-r=3}
p_{K_{t+2}^3}(y_1,\ldots,y_{t+2})
\le \frac{t(t+1)}{2(t+2)^2}- \frac{t}{6(t+2)}\sum_{i=1}^{t+2}\left(y_i-\frac{1}{t+2}\right)^2
\end{align}
for all $(y_1,\ldots, y_{t+2})\in \Delta_{t+1}$.

Let $\delta = \left(\frac{\epsilon}{30t}\right)^2$ and note that $\epsilon=30t\delta^{1/2}$.
Let $p = p_{\Gamma_{t+2}}(x_1,\ldots, x_{t+4})$ and $\sigma_{i} = \sigma_i(x_1,\ldots, x_t) =\sum_{E\in \binom{[t]}{i}}\prod_{j\in E}x_j$
for $i\in \{1,2,3\}$.
Let $y_{12} = x_{t+1}+x_{t+2}$, $y_{34} = x_{t+3}+x_{t+4}$, $y_{13} = x_{t+1}+x_{t+3}$, $y_{24} = x_{t+2}+x_{t+4}$,
and $y_{t+1} = y_{t+2} = (x_{t+1}+x_{t+2}+x_{t+3}+x_{t+4})/2$.
Note that $y_{t+1} = y_{t+2} = (y_{12}+y_{34})/2 = (y_{13}+y_{24})/2$.
By Equation $(\ref{INEQU:Maclaurin-r=2})$, we have
\begin{align}
p
& = \sigma_{3} + \sigma_{2}\left(1-\sigma_1\right)+\sum_{i=1}^{t-1}x_iy_{12}y_{34} + x_{t}y_{13}y_{24} \notag\\
& \le \sigma_{3} + \sigma_{2}\left(1-\sigma_1\right) + \sigma_{1}y_{t+1}y_{t+2}
    - (\sigma_1-x_{t})\left(\frac{y_{12}-y_{34}}{2}\right)^2 - x_{t}\left(\frac{y_{13}-y_{24}}{2}\right)^2 \notag\\
& = p_{K_{t+2}^{3}}(x_1,\ldots,x_{t},y_{t+1},y_{t+2})
        - (\sigma_1-x_{t})\left(\frac{y_{12}-y_{34}}{2}\right)^2 - x_{t}\left(\frac{y_{13}-y_{24}}{2}\right)^2. \notag
\end{align}
Since $p \ge \frac{t(t+1)}{2(t+2)^2} - \delta$, it follows from Equation $(\ref{INEQU:Maclaurin-r=3})$
and the inequality above that
\begin{align}\label{equ:delta}
\delta
& \ge \frac{t}{6(t+2)}\sum_{i=1}^{t}\left(x_i-\frac{1}{t+2}\right)^2+\sum_{i=t+1}^{t+2}\left(y_i-\frac{1}{t+2}\right)^2 \notag\\
& + (\sigma_1-x_{t})\left(\frac{y_{12}-y_{34}}{2}\right)^2+x_{t}\left(\frac{y_{13}-y_{24}}{2}\right)^2.
\end{align}
First, it follows that
\begin{align}
\left|x_i-\frac{1}{t+2}\right| & \le\sqrt{\frac{6(t+2)}{t}\delta}\le 6\delta^{1/2} \quad\text{for }i\in [t],\text{ and}\notag\\
\left|y_i-\frac{1}{t+2}\right| & \le\sqrt{\frac{6(t+2)}{t}\delta}\le 6\delta^{1/2} \quad\text{for }i\in \{t+1,t+2\}. \notag
\end{align}
In particular, the first inequality implies that $x_t \ge 1/(t+2)- 6\delta^{1/2} \ge 1/4t^2$ and  $\sigma_1 -x_t = \sum_{i=1}^{t-1}x_i \ge 1/4t^2$.
So Inequality $(\ref{equ:delta})$ implies that
$$\max\{|y_{12}-y_{34}|,|y_{13}-y_{24}|\}\le \max\left\{2\sqrt{\delta/x_{t}},2\sqrt{\delta/(\sigma_1-x_{t})} \right\} \le 4t\delta^{1/2}.$$
Since $y_{12}+y_{34} = y_{13}+y_{24} = 2y_{t+1}$, we obtain
\begin{align}
\max\left\{\left|y_{12}-\frac{1}{t+2}\right|,\left|y_{23}-\frac{1}{t+2}\right|,
\left|y_{13}-\frac{1}{t+2}\right|,\left|y_{24}-\frac{1}{t+2}\right|\right\}
\le 4t\delta^{1/2}+6\delta^{1/2} \le 10t\delta^{1/2}. \notag
\end{align}
Consequently,
\begin{align}
|x_{t+1}-x_{t+4}| & = |y_{12}-y_{24}| \le 20t\delta^{1/2}, \quad\text{and}\quad
|x_{t+2}-x_{t+3}| & = |y_{12}-y_{13}| \le 20t\delta^{1/2}. \notag
\end{align}
Since $x_{t+1}+x_{t+2} = y_{12} = \frac{1}{t+2}\pm 10t\delta^{1/2}$,
there exists $\alpha \in [0,1]$ such that
\begin{align}
x_{t+1} = \frac{\alpha}{t+2}\pm 10t\delta^{1/2}, \quad\text{and}\quad
x_{t+2} = \frac{1-\alpha}{t+2}\pm 10t\delta^{1/2}. \notag
\end{align}
Consequently,
\begin{align}
x_{t+3} = \frac{1-\alpha}{t+2}\pm 30t\delta^{1/2}, \quad\text{and}\quad
x_{t+4} = \frac{\alpha}{t+2}\pm 30t\delta^{1/2}. \notag
\end{align}
This proves Lemma~\ref{LEMMA:Lagrangian-stability}.
\end{proof}%LEMMA

%%%%%%%%%%%%%%%%%%%%%%%%%%%%%%%%%%%%%%%%%%%%%%%%%%%%%%%%%%%%
\subsection{Proof of Theorem~\ref{THM:main-sec1.1}~(d)}\label{SUBSEC:proof-stability}
We prove Theorem~\ref{THM:main-sec1.1}~(d) in this section.
Recall from Lemma~\ref{LEMMA:F-free-equ-F-homo-free} that $\mathcal{F}_{t+2}$ is blowup-invariant,
and from Claim~\ref{CLAIM:symmetrized-stable} that the hereditary family $\mathfrak{\Gamma}_{t+2}$ contains all symmetrized $\mathcal{F}_{t+2}$-free $3$-graphs.
Therefore, by Theorem~\ref{THM:Psi-trick:G-extendable-implies-degree-stability}, it suffices to show that $\mathcal{F}_{t+2}$ is vertex-extendable with respect to $\mathfrak{\Gamma}_{t+2}$.

\begin{proof}[Proof of Theorem~\ref{THM:main-sec1.1}~(d)]
%Recall that $\mathfrak{\Gamma}_{t+2}$ is the collection of all $\Gamma_{t+2}$-colorable $3$-graphs.
%By Theorem~\ref{THM:Psi-trick:G-extendable-implies-degree-stability},
%it suffices to prove that $\mathcal{F}_{t+2}$ is vertex-extendable with respect to $\mathfrak{\Gamma}_{t+2}$.
Let $\epsilon>0$ be a sufficiently small constant and $n$ be a sufficiently large integer.
Let $\delta>0$ be the constant guaranteed by Lemma~\ref{LEMMA:Lagrangian-stability}.
We may assume that $\delta \le \epsilon$.
Let $\mathcal{H}$ be an $(n+1)$-vertex $\mathcal{F}_{t+2}$-free $3$-graph with minimum degree at least
$3(\lambda - \delta/2)(n+1)^2$, where $\lambda = \lambda(\Gamma_{t+2}) = \frac{t(t+1)}{6(t+2)^2}$.

Suppose that there exists a vertex $v\in V(\mathcal{H})$ such that $\mathcal{H} - v \in \mathfrak{\Gamma}_{t+2}$.
Let $V = V(\mathcal{H})\setminus \{v\}$ and $\mathcal{H}' = \mathcal{H} - v$.
Let $V = V_1\cup \cdots \cup V_{t+4}$ be a partition
such that the map $\phi\colon V(\mathcal{H}') \to V(\Gamma_{t+2})$
defined by $\phi(V_i) = i$ for $i\in [t+4]$ induces a homomorphism from $\mathcal{H}'$ to $\Gamma_{t+2}$.
Let $\mathcal{G} = \Gamma_{t+2}[V_1,\ldots,V_{t+4}]$ be a blowup of $\Gamma_{t+2}$.
It follows from the assumption that
\begin{align}\label{equ;G-H'-size}
    |\mathcal{G}\setminus\mathcal{H}'|
    \le \lambda n^3 - \frac{n}{3} \left(3\left(\lambda - \frac{\delta}{2}\right)(n+1)^2 - n\right)
    \le \delta n^3.
\end{align}
It is possible that some set $V_i$ is empty. So it will be convenient later to define the following graphs.

Let $w_1,\ldots, w_{t+4}$ be $t+4$ new vertices and let $W_i = V_i\cup \{w_i\}$ for $i\in [t+4]$.
Let $\widehat{\mathcal{G}} = \Gamma_{t+2}[W_1,\ldots,W_{t+4}]$ be a blowup of $\Gamma_{t+2}$.
For every $i\in [t+4]$ define $L_{\mathcal{G}}(i) = L_{\widehat{\mathcal{G}}}(w_i)[V]$
and $N_{\mathcal{G}}(i) = V\setminus V_i$.
Notice that $L_{\mathcal{G}}(i)$ is a graph on $N_{\mathcal{G}}(i)$.

We will use $N(u)$ and $L(u)$ to denote $N_{\mathcal{H}}(u)$ and $L_{\mathcal{H}}(u)$ respectively
for every vertex $u\in V(\mathcal{H})$ if it does not cause any confusion.

\begin{claim}\label{CLAIM:stability-proof-goal}
It suffices to prove that there exists some $i_0\in [t+4]$ such that
$L(v)\subset L_{\mathcal{G}}(i_0)$.
\end{claim}
\begin{proof}
Suppose that there exists some $i_0\in [t+4]$ such that
$L(v)\subset L_{\mathcal{G}}(i_0)$.
First notice that since $N_{\mathcal{G}}(i_0)\cap V_{i_0} = \emptyset$, we have $N(v)\cap V_{i_0} = \emptyset$ as well.
Let $\widehat{V}_{i_0} = V_{i_0}\cup\{v\}$ and $\widehat{V}_i = V_i$ for $i\in [t+4]\setminus \{i_0\}$.
It is clear that $\mathcal{H}\subset \Gamma_{t+2}[\widehat{V}_1,\ldots,\widehat{V}_{t+4}]$.
Therefore, $\mathcal{F}_{t+2}$ is vertex-extendable with respect to $\mathfrak{\Gamma_{t+2}}$.
This proves Theorem~\ref{THM:main-sec1.1}~(d).
\end{proof}

Let $x_i = |V_i|/n$ for $i\in [t+4]$.
By assumption, we have
\begin{align}\label{equ:delta-H'}
    \delta(\mathcal{H}') \ge \delta(\mathcal{H})-n \ge 3(\lambda - \delta/2)(n+1)^2-n \ge 3(\lambda - \delta)n^2,
\end{align}
% $$\delta(\mathcal{H}') \ge \delta(\mathcal{H})-n \ge 3(\lambda - \delta/2)(n+1)^2-n \ge 3(\lambda - \delta)n^2,$$
and hence $|\mathcal{H}'| \ge \delta(\mathcal{H}')n/3 \ge (\lambda-\delta)n^3$.
Since $p_{\Gamma_{t+2}}(x_1,\ldots,x_{t+4}) \ge |\mathcal{H}'|/n^3 \ge \lambda-\delta$,
it follows from Lemma~\ref{LEMMA:Lagrangian-stability} that
there exists a real number $\alpha \in [0,1]$ such that Equation $(\ref{equ:weight-stable})$ holds.

\begin{claim}\label{CLAIM:missing-link-few}
For every vertex $u\in V$ we have
$|L_{\mathcal{G}}(u)\setminus L(u)| \le 4t \epsilon n^2$.
\end{claim}
\begin{proof}
Let us assume that $u\in V_i$ where $i\in [t+4]$.

If $i\in \{1,\ldots, t\}$, then it follows from $(\ref{equ:weight-stable})$ that
\begin{align}
|L_{\mathcal{G}}(u)|/n^2
& \le \binom{t-1}{2}\left(\frac{1}{t+2} + \epsilon\right)^2
    + (t-1)\left(\frac{1}{t+2} + \epsilon\right)\left(\frac{2}{t+2} + 4\epsilon\right) \notag\\
&  \quad  + \left(\frac{1}{t+2} + 2\epsilon\right)\left(\frac{1}{t+2} + 2\epsilon \right) \notag\\
& \le \frac{t(t+1)}{2(t+2)^2} + \frac{t(t+3)}{t+2}\epsilon + \frac{t^2+5t+2}{2}\epsilon^2
 \le \frac{t(t+1)}{2(t+2)^2} + 2t\epsilon. \notag
\end{align}

If $i \in \{t+1,\ldots, t+4\}$,
then it follows from $(\ref{equ:weight-stable})$ that
\begin{align}
|L_{\mathcal{G}}(u)|/n^2
& \le \binom{t}{2}\left(\frac{1}{t+2} + \epsilon\right)^2
    + t\left(\frac{1}{t+2} + \epsilon\right)\left(\frac{1}{t+2} + 2\epsilon\right) \notag\\
& \le \frac{t(t+1)}{2(t+2)^2} + t\epsilon + \frac{t(t+3)}{2}\epsilon^2
\le \frac{t(t+1)}{2(t+2)^2} + 2t\epsilon. \notag
\end{align}

Recall that $\lambda = \frac{t(t+1)}{6(t+2)^2}$ and $\delta\le \epsilon$.
Therefore, by~\eqref{equ:delta-H'}, we have
\begin{align}
|L_{\mathcal{G}}(u)\setminus L(u)|/n^2
\le \frac{t(t+1)}{2(t+2)^2} + 2t\epsilon - 3\left(\lambda-\delta\right)
\le \frac{t(t+1)}{2(t+2)^2} + 2t\epsilon - \frac{t(t+1)}{2(t+2)^2} + 3\epsilon
\le
4t\epsilon. \notag
\end{align}
\end{proof}%CLAIM

\begin{claim}\label{CLAIM:missing-neighbors-small}
For every $i\in [t+4]$ and every  $u\in V_i$ we have $|V_j\setminus N(u)| \le 12t\epsilon n$ for all $j\in [t+4]\setminus\{i\}$.
\end{claim}
\begin{proof}
Fix $i\in [t+4]$ and let $u\in V_i$.
If $i \in [t]$, then it follows from Lemma~\ref{LEMMA:Lagrangian-stability} that
\begin{align}
\delta(L_{\mathcal{G}}(u))/n
& \ge \min\left\{(t-2)\left(\frac{1}{t+2}-\epsilon\right)
        +2\left(\frac{\alpha}{t+2}-\epsilon+\frac{1-\alpha}{t+2}-\epsilon\right),\right. \notag\\
&  \quad  \left. (t-1)\left(\frac{1}{t+2}-\epsilon\right)+\frac{\alpha}{t+2}-\epsilon+\frac{1-\alpha}{t+2}-\epsilon\right\}
 \ge \frac{t}{t+2}-(t+2)\epsilon
 \ge \frac{1}{3}. \notag
\end{align}
%%%%
If $i \in [t+1,t+4]$, then it follows from Lemma~\ref{LEMMA:Lagrangian-stability} that
\begin{align}
\delta(L_{\mathcal{G}}(u))/n
& \ge \min\left\{(t-1)\left(\frac{1}{t+2}-\epsilon\right)
        +\frac{\alpha}{t+2}-\epsilon+\frac{1-\alpha}{t+2}-\epsilon, t\left(\frac{1}{t+2}-\epsilon\right)\right\} \notag\\
& \ge \frac{t}{t+2}-(t+2)\epsilon
\ge \frac{1}{3}. \notag
\end{align}
Since each vertex in $V_j\setminus N(u)$ contributes at least $\delta(L_{\Gamma_{t+2}}(u))$ edges to
$L_{\mathcal{G}}(u)\setminus L(u)$,
it follows from Claim~\ref{CLAIM:missing-link-few} that
\begin{align}
|V_j\setminus N(u)|
\le \frac{4t\epsilon n^2}{n/3} = 12t\epsilon n. \notag
\end{align}
\end{proof}

By symmetry, we may assume that $\alpha \ge 1/2$.
Since $\epsilon$ is sufficiently small, we have
\begin{align}\label{equ:min-xi}
    \min\{x_1,\ldots, x_t, x_{t+1}, x_{t+4}\} \ge \frac{1}{4(t+2)}.
\end{align}

\begin{claim}\label{CLAIM:link-empty-inside}
We have $L(v) \cap \binom{V_i}{2} = \emptyset$ for all $i\in [t+4]$.
\end{claim}
\begin{proof}
Suppose to the contrary that there exists $i\in [t+4]$ such that $L(v) \cap \binom{V_i}{2} \neq \emptyset$.
Let us assume that $\{u_i,u_i'\} \in L(v) \cap \binom{V_i}{2}$.

First let us assume that  $i\in \{1,\ldots, t,t+1,t+4\}$.
Applying Lemma~\ref{LEMMA:greedily-embedding-Gi} with $\eta = 4t\epsilon$, $T = \{1,\ldots, t,t+1,t+4\}\setminus \{i\}$, and $S = \{u_i\}, \{u_i'\}$, respectively,
we obtain $u_j \in V_j$ for every $j\in T$ such that sets $U = \{u_j\colon j\in T\}\cup \{u_i\}$
and $U' = \{u_j\colon j\in T\}\cup \{u_i'\}$
satisfy
$$\mathcal{H}[U] = \mathcal{G}[U] \cong K_{t+2}^{3}\quad\text{and}\quad
    \mathcal{H}[U'] = \mathcal{G}[U']\cong K_{t+2}^{3}.$$
Note that Lemma~\ref{LEMMA:greedily-embedding-Gi}~(a) is guaranteed by~\eqref{equ:min-xi},
Lemma~\ref{LEMMA:greedily-embedding-Gi}~(b) is guaranteed by~\eqref{equ;G-H'-size},
and Lemma~\ref{LEMMA:greedily-embedding-Gi}~(c) is guaranteed by Claim~\ref{CLAIM:missing-link-few}.

Let $F = \mathcal{H}[U\cup U']\cup \{\{v,u_i,u_i'\}\}$.
It is clear that $v(F) = t+4\le 4(t+4)^2$.
Therefore, $F$ is contained as a subgraph in some blowup of $\Gamma_{t+2}$.
Since $\mathcal{H}[U]\cong K_{t+2}^{3}$, $\mathcal{H}[U']\cong K_{t+2}^{3}$, and $|U\cap U'| = t+1$,
it follows from Lemma~\ref{LEMMA:two-K_t+2} that $\{u_i,u_i'\}$ is not contained in any edge of $F$, a contradiction.

Now let us assume that $i\in \{t+2,t+3\}$.
Applying Lemma~\ref{LEMMA:greedily-embedding-Gi} with $\eta = 4t\epsilon$ and $T = \{1,\ldots, t,t+1,t+4\}$
we obtain $u_j \in V_j$ for every $j\in T$ such that the set $U = \{u_j\colon j\in T\}$ satisfies
\begin{align}
\mathcal{H}[U] = \mathcal{G}[U] \cong K_{t+2}^{3}\quad\text{and}\quad
L(u_i)[U]  = L_{\mathcal{G}}(u_i)[U] = L_{\mathcal{G}}(u_i')[U]= L(u_i')[U]. \notag
\end{align}

Let $F = \mathcal{H}[U\cup \{u_i,u_i'\}]\cup \{\{v,u_i,u_i'\}\}$.
It is clear that $v(F) = t+5 \le 4(t+4)^2$.
So there exists a homomorphism $\phi\colon V(F)\to V(\Gamma_{t+2})$ from $F$ to $\Gamma_{t+2}$.
Since $F[U]\cong K_{t+2}^{3}$, it follows from Observation~\ref{OBS:structure-link} (a) that either
$\phi(U) = [t]\cup \{t+1,t+4\}$ or $\phi(U) = [t]\cup \{t+2,t+3\}$.
By symmetry we may assume that $\phi(U) = [t]\cup \{t+1,t+4\}$.
Since $U\cup \{u_i,u_i'\}$ is $2$-covered in $F$,
we have $\{\phi(u_i),\phi(u_i')\} = [t+4]\setminus \phi(U) = \{t+2,t+3\}$.
However, notice that $|L_{F}(u_i)[U] \cap L_{F}(u_i')[U]| = \binom{t}{2}+t$,
while $|L_{\Gamma_{t+2}}(t+2)[\phi(U)] \cap L_{\Gamma_{t+2}}(t+3)[\phi(U)]| = \binom{t}{2}<\binom{t}{2}+t$, a contradiction.
\end{proof}%CLAIM

Define
\begin{align}
I_{small}  = \left\{i\in [t+4]\colon |N(v)\cap V_i|\le \epsilon^{1/4}n\right\}, \quad\text{and}\quad
I_{\emptyset}  = \left\{i\in [t+4]\colon N(v)\cap V_i = \emptyset\right\}. \notag
\end{align}
It follows from the definition that $I_{\emptyset} \subset I_{small}$.

%%%%%%%%%%%%%%%%%%%%%%%%%%%%%%%%%%%%%%%%%%%%%%%%%%%%%%%%%%%%%%%%%%%
Next we will consider two cases depending on the sizes of $V_{t+2}$ and $V_{t+3}$:
either $V_{t+2}$ and $V_{t+3}$ are both large (i.e. at least $50t^2\epsilon^{1/4}$) or $V_{t+2}$ and $V_{t+3}$ are both small.
The first case is quite easy to handle while the second case needs more work.

\bigskip

\noindent\textbf{Case 1}: $\alpha \le 1 - 101t^3\epsilon^{1/4}$.

Note that in this case we have
$$\min\left\{|V_{t+1}|,|V_{t+2}|,|V_{t+3}|,|V_{t+4}|\right\}\ge \frac{101t^3\epsilon^{1/4}}{t+2}n \ge 50t^2\epsilon^{1/4}n.$$

Define an auxiliary graph $R$ on $[t+4]$ in which two vertices $i$ and $j$ are adjacent
iff there exists $e\in L(v)$ such that $e\cap V_{i} \neq \emptyset$ and $e\cap V_{j} \neq \emptyset$.
An easy observation is that $L(v)\subset R[V_1,\ldots,V_{t+4}]$.

\begin{claim}\label{CALIM:case1-R1-structure}
We have $L(v) \subset L_{\mathcal{G}}(i)$ for some $i\in I_{small}$.
\end{claim}
\begin{proof}
For each edge $ij \in R$ let $e_{ij} \in L(v)$ be an edge
with one endpoint in $V_i$ and the other endpoint in $V_j$.
Let $U_1 = \bigcup_{ij\in R}e_{ij}$ be a vertex subset of $V$ and note that $|U_1| \le 2\binom{t+4}{2}$.
Define
\begin{align}
V_i' =
\begin{cases}
V_i \cap N(v) \cap \left(\bigcap_{u\in U_1\setminus V_i}N(u)\right) & \text{for } i \in [t+4]\setminus I_{small}, \\
V_i \cap \left(\bigcap_{u\in U_1\setminus V_i}N(u)\right) & \text{for } i \in I_{small}.
\end{cases} \notag
\end{align}
It follows from Claim~\ref{CLAIM:missing-neighbors-small} and our assumption that
\begin{align}
|V_i'|\ge \min\left\{\epsilon^{1/4}n - \left(|U_1|+1\right)\times 12t\epsilon n,\
                50t^2\epsilon^{1/4}n - |U_1|\times 12t\epsilon n\right\} \ge 4(t+4)^3\epsilon^{1/3}n. \notag
\end{align}
Applying Lemma~\ref{LEMMA:greedily-embedding-Gi} with $\eta = 4t\epsilon$ and $T = [t+4]$
we obtain $u_i \in V_i'$ for $i\in T$ such that the set $U_2 = \{u_i\colon i\in T\}$ satisfies
\begin{enumerate}[label=(\alph*)]
  \item $\mathcal{H}[U_2] = \mathcal{G}[U_2] \cong \Gamma_{t+2}$, and
  \item $L(u)[U_2] = L_{\mathcal{G}}(u)[U_2]$ for all $u\in U_1$.
\end{enumerate}
Let $F = \mathcal{H}[U_1\cup U_2 \cup \{v\}]$.
Note that $v(F) \le 2\binom{t+4}{2}+ t+5\le  3(t+4)^2$.
By assumption there exists a homomorphism $\phi\colon V(F) \to V(\Gamma_{t+2})$ from $F$ to $\Gamma_{t+2}$
(otherwise we would have $F\in\mathcal{F}_{t+2}$, a contradiction).
Since $F[U_2] \cong \Gamma_{t+2}$, it follows from Lemma~\ref{LEMMA:gamma-t-automorphism} that
$\phi(u_{i})\in [t]$ for $i\in [t]$ (and $\phi(u_t) = t$ if $t\ge 3$), and
\begin{align}
\left\{ \phi(\{u_{t+1},u_{t+4}\}),\phi(\{u_{t+2},u_{t+3}\}) \right\} = \left\{ \{t+1,t+4\},\{t+2,t+3\} \right\}. \notag
\end{align}
By symmetry we may assume that $\phi(u_i) = i$ for $i \in [t+4]$.
For every $i\in [t+4]$ and $u\in V_i\cap U_1$ since $u$ is adjacent to all vertices in $U_2\setminus \{u_i\}$ in $\partial F$,
we have $\phi(u) = i$.
Finally, since $v$ is adjacent to all vertices in $\{u_i\colon i\in [t+4]\setminus I_{small}\}$,
we have $\phi(v)\not\in \phi([t+4]\setminus I_{small})$.
In other words, $\phi(v)\in \phi(I_{small})$.
This means that there exists some $i\in I_{small}$ such that
$\phi(R)\subset L_{\Gamma_{t+2}}(\phi(i))$, and hence $L(v) \subset L_{\mathcal{G}}(i)$.
\end{proof}%CLAIM

\medskip

%%%%%%%%%%%%%%%%%%%%%%%%%%%%%%%%%%%%%%%%%%%%%%%%%%%%%%%%%%
\noindent\textbf{Case 2}: $\alpha \ge 1 - 101t^3\epsilon^{1/4}$.

Note that in this case we have
\begin{align}
\max\left\{|V_{t+2}|,|V_{t+3}|\right\} & \le \frac{101t^3\epsilon^{1/4}}{t+2}n+\epsilon n \le 102t^2\epsilon^{1/4}n, \quad\text{and}\notag\\
\min\left\{|V_{t+1}|,|V_{t+4}|\right\} & \ge \frac{1-100t^3\epsilon^{1/4}}{t+2}n-\epsilon n \ge \frac{n}{t+2}-50t^2\epsilon^{1/4}n. \notag
\end{align}

For convenience, let $S = [t]\cup \{t+1,t+4\}$.

\begin{claim}\label{CLAIM:case2-at-most-one-small-intersection}
We have $|I_{small}\cap S| \le 1$.
\end{claim}
\begin{proof}
Suppose to the contrary that $|I_{small}\cap S| \ge 2$.
Then it follows from Lemma~\ref{LEMMA:Lagrangian-stability} that
\begin{align}
|L(v)|/n^2
& \le \binom{t}{2}\left(\frac{1}{t+2}+\epsilon\right)^2 + 2\times 102t^2\epsilon^{1/4} + 2\epsilon^{1/4} \notag\\
& =    \frac{t(t+1)}{2(t+2)^2} - \frac{t}{(t+2)^2} +\frac{t(t-1)}{t+2}\epsilon
    +  \frac{t(t-1)}{2}\epsilon^2+204t^2\epsilon^{1/4} + 2\epsilon^{1/4} \notag\\
& < 3(\lambda-\delta), \notag
\end{align}
which contradicts~\eqref{equ:delta-H'}.
\end{proof}%CLAIM

Next we will consider two cases depending on the value of $|I_{small}\cap S|$: either $|I_{small}\cap S| = 1$ or $|I_{small}\cap S| = 0$.

\noindent\textbf{Case 2.1}: $|I_{small}\cap S| = 1$.

We may assume that $I_{small} = \{1\}$. The proof for other cases follows analogously.
Let $W = \bigcup_{i\in S\setminus I_{small}}V_i$.
It follows from our assumption that
\begin{align}
|L(v)[W]|/n^2
\ge \delta(\mathcal{H})/n^2 - 2\times 102t^2\epsilon^{1/4} - \epsilon^{1/4}
\ge \frac{t(t+1)}{2(t+2)^2} - 205t^2\epsilon^{1/4}. \notag
\end{align}
Notice that the induced subgraph of $L_{\mathcal{G}}(1)$ on $W$
is the blowup $K_{t+1}[V_2,\ldots,V_{t},V_{t+1},V_{t+4}]$ of $K_{t+1}$.
Using the same argument as in the proof of Claim~\ref{CLAIM:missing-link-few}, we get
\begin{align}
|L_{\mathcal{G}}(1)\setminus L(v)[W]|/n^2
\le \frac{t(t+1)}{2(t+2)^2} + 2t\epsilon - \left(\frac{t(t+1)}{2(t+2)^2} - 205t^2\epsilon^{1/4}\right)
\le 206t^2\epsilon^{1/4}. \notag
\end{align}
Also, using the same argument as in the proof of Claim~\ref{CLAIM:missing-neighbors-small}, we get
\begin{align}
|N(v)\cap V_{i}| \ge |V_i| - \frac{206t^2\epsilon^{1/4}n^2}{n/3}
        \ge \frac{n}{t+2} -50 t^2\epsilon^{1/4}n - 700t^2\epsilon^{1/4}n
        \ge \frac{n}{2(t+2)}. \notag
\end{align}

\begin{claim}\label{CLAIM:case2-no-neighbor-in-V1}
We have $N(v)\cap V_1 = \emptyset$.
\end{claim}
\begin{proof}
Suppose to the contrary that there exists a vertex $u_1\in N(v)\cap V_1$.
Let $V_i' = V_i\cap N(v) \cap N(u_1)$ for $i\in S\setminus \{1\}$.
Notice from Claim~\ref{CLAIM:missing-neighbors-small} that $|V_i'| \ge \frac{n}{2(t+2)} - 12t\epsilon n > 206(t+4)^4\epsilon^{1/12} n$.
Applying Lemma~\ref{LEMMA:greedily-embedding-Gi} with $\eta = 206t^2\epsilon^{1/4}$ and $T = S\setminus \{1\}$,
we obtain a vertex $u_i \in V_i'$ for every $i\in T$ such that the set $U = \{u_i\colon i\in T\}$ satisfies
\begin{align}
\mathcal{H}[U\cup \{u_1\}] = \mathcal{G}[U\cup \{u_1\}] \cong K_{t+2}^{3}
\quad\text{and}\quad
\mathcal{H}[U\cup \{v\}] = \mathcal{G}[U\cup \{v\}] \cong K_{t+2}^{3}. \notag
\end{align}
Let $e\in \mathcal{H}$ be an edge that contains $\{v,u_1\}$ and let $F = \mathcal{H}[U\cup\{v,u_1\}]\cup\{e\}$.
By assumption, $F$ should be contained in some blowup of $\Gamma_{t+2}$,
but this would contradict Lemma~\ref{LEMMA:two-K_t+2} since $\{v,u_1\}$ is contained in $e$.
\end{proof}%CLAIM

\begin{claim}\label{CLAIM:case2-put-v-in-V1}
We have $L(v) \subset L_{\mathcal{G}}(1)$.
\end{claim}
\begin{proof}
Suppose to the contrary that there exists an edge $e\in L(v)\setminus L_{\mathcal{G}}(1)$.
By symmetry we may assume that $e = \{u_{t+1},u_{t+2}\}$ and $(u_{t+1},u_{t+2})\in V_{t+1}\times V_{t+2}$.
Let
\begin{align}
V_{i}' =
\begin{cases}
V_{i}\cap N(u_{t+1}) \cap N(u_{t+2}) & \text{if } i = 1, \\
V_{i}\cap N(v) \cap N(u_{t+2}) & \text{if } i = t+1, \\
V_{i}\cap N(v) \cap N(u_{t+1}) \cap N(u_{t+2}) & \text{if } i\in [2,t]\cup \{t+4\}.
\end{cases}\notag
\end{align}
Notice that $|V_i'| \ge \frac{n}{2(t+2)} - 2\times 12t\epsilon > 206(t+4)^4\epsilon^{1/12}$.
Applying Lemma~\ref{LEMMA:greedily-embedding-Gi} with $\eta = 206t^2\epsilon^{1/4}$ and $T = S$,
we obtain a vertex $u_i' \in V_i'$ for every $i\in T$ such that the set $U = \{u_i'\colon i\in T\}$ satisfies
\begin{enumerate}[label=(\alph*)]
\item $\mathcal{H}[U] = \mathcal{G}[U] \cong K_{t+2}^{3}$,
\item $L(v)[U\setminus\{u_1\}]\cong K_{t+1}$, and
\item $L(u_i)[U] = L_{\mathcal{G}}(u_i)[U]$ for $i\in \{t+1,t+2\}$.
\end{enumerate}

Let $F = \mathcal{H}[U\cup \{v,u_{t+1},u_{t+2}\}]$ and $U' = \left(U\cup \{u_{t+1}\}\right)\setminus \{u_{t+1}'\}$.
Since $v(F) = t+5 \le 4(t+4)^2$,
by assumption, there is a homomorphism $\phi\colon V(F) \to V(\Gamma_{t+2})$ from $F$ to $\Gamma_{t+2}$.
Since $\mathcal{H}[U] \cong \mathcal{H}[U'] \cong K_{t+2}^{3}$,
we have $\Gamma_{t+2}[\phi(U)] \cong \Gamma_{t+2}[\phi(U')] \cong K_{t+2}^3$.
Since $|\phi(U) \cap \phi(U')| = |U\cap U'| = t+1$,
it follows from Lemma~\ref{LEMMA:two-K_t+2} that $\phi(U) = \phi(U')$ and $\phi(u_{t+1}) = \phi(u_{t+1}')$.

Let $U_1 = U\cup\{u_{t+2}\}$.
Since $F[U_1]\cong \Gamma_{t+2}\setminus\{t+3\}$, it follows from Lemma~\ref{LEMMA:gamma-t-partial-embedding}
that $\phi(\{u_i'\colon i\in[t]\}) = [t]$ (and $\phi(u'_{t}) = t$ if $t\ge 3$),
$\phi(\{u_{t+1}',u_{t+2},u_{t+4}'\}) \subset [t+1,t+4]$,
and $\phi(\{u_{t+1}',u_{t+4}'\}) \in \{\{t+1,t+4\},\{t+2,t+3\}\}$.
By symmetry we may assume that $\phi(u_{i}') = i$ for $i\in S$ and $\phi(u_{t+2}) = t+2$.
Hence $\phi(u_{t+1}') = \phi(u_{t+1}) = t+1$.
Since $v$ is adjacent to all vertices in $U_1\setminus\{u_{1}\}$,
we have $\phi(v)\in [t+4]\setminus\phi(U_1\setminus\{u_{1}\}) = \{1,t+3\}$.
If $\phi(v) = t+3$, then $\{v,u_{t+1},u_{t+2}\}\in F$ implies that
$\{t+1,t+2,t+3\} = \{\phi(v),\phi(u_{t+1}),\phi(u_{t+2})\} \in \Gamma_{t+2}$, a contradiction.
Therefore, we may assume that $\phi(v) = 1$.
Then $\left\{\{u_1',u_{t+1}',u_{t+4}'\},\{u_1',u_{t+2},u_{t+4}'\},\{v,u_{t+1},u_{t+2}\}\}\right\}\subset F$ implies
that
\begin{align}
& \left\{\{1,t+1,t+4\},\{1,t+2,t+4\},\{1,t+1,t+2\}\}\right\} \notag\\
=&\left\{\{\phi(u_1'),\phi(u_{t+1}'),\phi(u_{t+4}')\},
        \{\phi(u_1'),\phi(u_{t+2}),\phi(u_{t+4}')\},\{\phi(v),\phi(u_{t+1}),\phi(u_{t+2})\}\right\}
\subset\Gamma_{t+2}. \notag
\end{align}
This means that the triangle $\{t+1,t+2,t+4\}$ is contained in the induced subgraph of $L_{\Gamma_{t+2}}(1)$ on $\{t+1, t+2, t+3, t+4\}$,
contradicts the fact that the induced subgraph of $L_{\Gamma_{t+2}}(1)$ on $\{t+1, t+2, t+3, t+4\}$ is a copy of $C_4$.
\end{proof}%CLAIM

\medskip

Recall that $S = [t]\cup \{t+1, t+2\}$.

\noindent\textbf{Case 2.2}: $I_{small}\cap S = \emptyset$.

Recall that $R$ is a graph on $[t+4]$ in which two vertices $i$ and $j$ are adjacent
iff there exists $e\in L(v)$ such that $e\cap V_{i} \neq \emptyset$ and $e\cap V_{j} \neq \emptyset$.
Recall the observation that $L(v)\subset R[V_1,\ldots,V_{t+4}]$.

\begin{claim}\label{CALIM:case2-Vj+2-nonempty}
If $V_{t+2}\cup V_{t+3}\neq\emptyset$, then
        either $L(v) \subset L_{\mathcal{G}}(t+2)$ or $L(v) \subset L_{\mathcal{G}}(t+3)$.
\end{claim}
\begin{proof}
For every $ij\in R$ let $e_{ij} \in L(v)$ be an edge such that $e_{ij}\cap V_{i}\neq \emptyset$
and $e_{ij}\cap V_{j}\neq \emptyset$.
Let $U_1 = \bigcup_{ij\in R_2}e_{ij}$ be a vertex subset of $V$. Note that $|U_1|\le 2\binom{t+4}{2}$.

By symmetry we may assume that $V_{t+2} \neq \emptyset$.
Fix a vertex $u_{t+2} \in V_{t+2}$ (it is possible that $u_{t+2}\in U_1$ as well).
For every $i\in S$ let $V_i' = V_i\cap N(v) \cap N(u_{t+2}) \cap \left(\bigcap_{u\in U_1\setminus V_i}N(u)\right)$.
Notice that $|V_i'| \ge \frac{n}{2(t+2)} - (2\binom{t+4}{2}+1)12t\epsilon > 4(t+4)^3\epsilon^{1/3}n$.
Applying Lemma~\ref{LEMMA:greedily-embedding-Gi} with $\eta = 4t\epsilon$ and $T = S$
we obtain $u_{i}\in V'_i$ for every $i\in T$ such that the set $U = \{u_i\colon i\in T\}$ satisfies
\begin{enumerate}[label=(\alph*)]
\item $\mathcal{H}[U] = \mathcal{G}[U] \cong K_{t+2}^{3}$,
\item $L(u)[U] = L_{\mathcal{G}}(u)[U]$ for all $u\in U_1$, and
\item $L(u_{t+2})[U] = L_{\mathcal{G}}(u_{t+2})[U]$.
\end{enumerate}

Let $F = \mathcal{H}[U\cup U_1\cup \{v,u_{t+2}\}]$. Since $v(F) \le 2\binom{t+4}{2}+t+5 \le 4(t+4)^3$,
by assumption there is a homomorphism $\phi\colon V(F) \to V(\Gamma_{t+2})$ from $F$ to $\Gamma_{t+2}$.
Since $F[U\cup \{u_{t+2}\}] \cong \Gamma_{t+2}\setminus\{t+3\}$, it follows from Lemma~\ref{LEMMA:gamma-t-partial-embedding}
that $\phi(\{u_i\colon i\in [t]\}) = [t]$ (and $\phi(u_{t}) = t$ if $t\ge 3$),
$\phi(\{u_{t+1},u_{t+2},u_{t+4}\}) \subset [t+1,t+4]$, and
$\phi(\{u_{t+1},u_{t+4}\})\in \{\{t+1,t+4\},\{t+2,t+3\}\}$.
By symmetry, we may assume that $\phi(u_{i}) = i$ for $i\in [t]\cup \{t+1,t+2,t+4\}$.
For every $i\in [t]\cup \{t+1,t+4\}$ and every $u\in U_{1}\cap V_{i}$,
since $L(u)[U] = L(u_i)[U] = L_{\mathcal{G}}(u_i)[U]$
and $u$ is adjacent to all vertices in $U\setminus \{u_i\}$,
we have $\phi(u) = \phi(u_i) = i$.
Finally, since $v$ is adjacent to all vertices in $\{u_i\colon i\in [t]\cup \{t+1,t+3\}\}$,
we have $\phi(v) \in [t+4]\setminus\phi(\{u_i\colon i\in [t]\cup \{t+1,t+3\}\}) = \{t+2,t+3\}$.
If $\phi(v) = t+2$, then $\phi(R) \subset L_{\Gamma_{t+2}}(\phi(v)) = L_{\Gamma_{t+2}}(t+2)$,
which means that $L(v)\subset L_{\mathcal{G}}(t+2)$.
If $\phi(v) = t+3$, then $\phi(R) \subset L_{\Gamma_{t+2}}(\phi(v)) = L_{\Gamma_{t+2}}(t+3)$,
which means that $L(v)\subset L_{\mathcal{G}}(t+3)$.
\end{proof}%CLAIM

Now we may assume that $V_{j+2}\cup V_{j+3}=\emptyset$.

\begin{claim}\label{CLAIM:case2-structure-of-R}
We have either $L(v)\subset L_{\mathcal{G}}(t+2)$ or $L(v)\subset L_{\mathcal{G}}(t+3)$.
\end{claim}
\begin{proof}
For each edge $ij \in R$ let $e_{i,j} \in L(v)$ be an edge with one endpoint in $V_i$
and the other endpoint in $V_j$.
Let $U_1 = \bigcup_{ij\in R}e_{i,j}$ be a vertex subset of $V$ and note that $|U_1| \le 2\binom{t+2}{2}$.
Define
\begin{align}
V_i' = V_i \cap N(v) \cap \left(\bigcap_{u\in U_1\setminus V_i}N(u)\right) \quad\text{for } i \in [t]\cup \{t+1,t+4\}. \notag
\end{align}
Notice that $|V_i'| \ge \frac{n}{2(t+2)} - (2\binom{t+4}{2}+1)12t\epsilon > 4(t+4)^3\epsilon^{1/3}n$.
Applying Lemma~\ref{LEMMA:greedily-embedding-Gi} with $\eta = 4t\epsilon$ and $T = [t]\cup \{t+1,t+4\}$
we obtain $u_i \in V_i'$ for $i\in T$ such that the set $U_2 = \{u_i\colon i\in T\}$ satisfies
\begin{enumerate}[label=(\alph*)]
  \item $\mathcal{H}[U_2] =\mathcal{G}[U_2] \cong K_{t+2}^3$, and
  \item $L(u)[U_2\setminus\{u_i\}] =L_{\mathcal{G}}(u)[U_2\setminus\{u_i\}]\cong K_{t+1}$
  for all $u\in U_1\cap V_i$ and $i\in S$.
\end{enumerate}
Let $F = \mathcal{H}[U_1\cup U_2 \cup \{v\}]$.
Suppose that there exists a homomorphism $\phi\colon V(F) \to V(\Gamma_{t+2})$ from $F$ to $\Gamma_{t+2}$.
Since the set $F[U_2] \cong K_{t+2}^3$ we have either $\phi(U_2) = [t]\cup\{t+1,t+4\}$
or $\phi(U_2) = [t]\cup\{t+2,t+3\}$.
By symmetry we may assume that the former case holds and $\phi(u_i) = i$ for $i\in [t]\cup\{t+1,t+4\}$.
Since $L_{\Gamma_{t+2}}(t+2)$ and $L_{\Gamma_{t+2}}(t+3)$ do not contain $K_{t+1}$ as a subgraph,
it follows from $(b)$ that for every $u\in V_i\cap U_1$ and $i\in [t]\cup\{t+1,t+4\}$
we have $\phi(u) \not\in\{t+2,t+3\}$.
Moreover, since $u$ is adjacent to all vertices in $U_2\setminus \{u_i\}$ in $\partial F$,
we have $\phi(u)\not\in \phi(U_2\setminus \{u_i\})$ as well.
Therefore, $\phi(u) = \phi(u_i) = i$.
Finally, since $U_2\cup \{v\}$ is $2$-covered in $F$, we have $\phi(v) \in \{t+2,t+3\}$.
If $\phi(v) = t+2$, then we have $L(v)\subset L_{\mathcal{G}}(t+2)$.
If $\phi(v) = t+3$, then we have $L(v)\subset L_{\mathcal{G}}(t+3)$.
\end{proof}%CLAIM
%
%%It follows from Claim~\ref{CLAIM:case2-structure-of-R} that $L(v)$ is $B_2'$-colorable.
%Without loss of generality we may assume that the map $\phi\colon V \to V(B_2')$ defined by
%$\phi(V_i) = i$ for $i\in S$ induced a homomorphism from $L$ to $B_2'$.
%By the definition of $B_2'$, this means that $L(v) \subset L_{\mathcal{G}}(t+2)$.
%Define
%\begin{align}
%\widehat{V}_i =
%\begin{cases}
%V_i\cup \{v\} & \text{if } i  = t+2, \\
%V_i & \text{if } i \in [t+4]\setminus\{t+2\}.
%\end{cases} \notag
%\end{align}
%Then we have $\mathcal{H} \subset \Gamma_{t+2}[\widehat{V}_{1},\ldots, \widehat{V}_{t+4}]$.
This completes the proof of Theorem~\ref{THM:main-sec1.1}~(d).
\end{proof}%THM
%%%%%%%%%%%%%%%%%%%%%%%%%%%%%%%%%%%%%%%%%%%%%%%%%
%%%%%%%%%%%%%%%%%%%%%%%%%%%%%%%%%%%%%%%%%%%%%%%%%
\section{Feasible region}\label{SEC:proof-feasible-region}
We prove Theorem~\ref{THM:feasibe-region} in this section.
First we need the following simple corollary of Theorem~\ref{THM:main-sec1.1}~(d).
For convenience, we will keep using $t+2$ instead of $t$ in this section.

Recall that for an $n$-vertex $r$-graph $\mathcal{H}$ the edge density of $\mathcal{H}$ is $\rho(mathcal{H}) = |\mathcal{H}|/\binom{n}{r}$, and the shadow density of $\mathcal{H}$ is $|\partial\mathcal{H}|/\binom{n}{r-1}$.

\begin{corollary}\label{CORO:stability}
For every integer $t\ge 2$
there exist constants $\epsilon_0>0$ and $N_0$ such that the following statement holds for all $\epsilon\le \epsilon_0$
and $n\ge N_0$.
Suppose that $\mathcal{H}$ is an $n$-vertex $\mathcal{F}_{t+2}$-free $3$-graph with
$\rho(\mathcal{H})\ge  \frac{t(t+1)}{(t+2)^2} - \epsilon$.
Then there exists a set $Z_{\epsilon}\subset V(\mathcal{H})$ of size at most $\epsilon^{1/2}n$
such that $\mathcal{H}\setminus Z_{\epsilon}$ is $\Gamma_{t+2}$-colorable and $\delta(\mathcal{H}) \ge \left( \frac{t(t+1)}{2(t+2)^2} - 3\epsilon^{1/2}\right)n^2$.
\end{corollary}

The idea for proving Corollary~\ref{CORO:stability} is to show that after removing a few vertices with small degree
from $\mathcal{H}$ the remaining $3$-graph has large minimum degree, and hence, by Theorem~\ref{THM:main-sec1.1}~(d),
it is $\Gamma_{t+2}$-colorable.
We refer the reader to the proof of Theorem~4.1 in~\cite{LMR1} for detailed calculations.

\begin{lemma}\label{LEMMA:feasible-region-shadow}
Suppose that $\mathcal{H}$ is an $n$-vertex $\Gamma_{t+2}$-colorable $3$-graph with
$|\mathcal{H}| \ge \left(\frac{t(t+1)}{6(t+2)^2}-\epsilon\right)n^3$.
Then $\left(\frac{t+1}{2(t+2)}-100t^4\epsilon^{1/2}\right)n^2 \le
|\partial\mathcal{H}| \le \left(\frac{t^2+3 t+3}{2 (t+2)^2}+5000t^4\epsilon^{1/2}\right)n^2$.
\end{lemma}
\begin{proof}
Let $V = V_1\cup\cdots\cup V_{t+4}$ be a partition such that $\mathcal{H}\subset \Gamma_{t+2}[V_1,\ldots,V_{t+4}]$.
Let $\mathcal{G} = \Gamma_{t+2}[V_1,\ldots,V_{t+4}]$.
Let $x_i = |V_i|/n$ for $i\in[t+4]$.
By Lemma~\ref{LEMMA:Lagrangian-stability}, there exists some $\alpha\in[0,1]$ such that
\begin{align}\label{equ:weight-stable-2}
x_i =
\begin{cases}
\frac{1}{t+2} \pm 30t\epsilon^{1/2} & \text{if } i\in [t], \\
\frac{\alpha}{t+2} \pm 30t\epsilon^{1/2} & \text{if } i\in \{t+1, t+4\},\\
\frac{1-\alpha}{t+2} \pm 30t\epsilon^{1/2} & \text{if } i\in\{t+2, t+3\}.
\end{cases} \notag
\end{align}
Notice that for every pair $\{u,v\}\in \partial\mathcal{G}$
the number of edges in $\mathcal{G}$ containing $\{u,v\}$ is at least
$\left(\frac{1}{t+2} - 30t\epsilon^{1/2}\right)n \ge \frac{n}{2(t+2)}$.
Therefore, it follows from a simple double counting that
\begin{align}
|\partial\mathcal{G}\setminus \partial\mathcal{H}|
\le \frac{3|\mathcal{G}\setminus\mathcal{H}|}{n/\left(2(t+2)\right)}
\le \frac{3\epsilon n^3}{n/\left(2(t+2)\right)}
\le 6\epsilon (t+2)n^2. \notag
\end{align}
Therefore,
\begin{align}
|\partial\mathcal{H}|/n^2
& \ge |\partial\mathcal{G}|/n^2 -6\epsilon (t+2)
 = p_{\partial\mathcal{G}}(x_1,\ldots,x_{t+4})-6(t+2)\epsilon \notag\\
& \ge \binom{t}{2}\left(\frac{1}{t+2} - 30t\epsilon^{1/2}\right)^2
        + t\left(\frac{1}{t+2} - 30t\epsilon^{1/2}\right)\left(\frac{2}{t+2} - 4\times 30t\epsilon^{1/2}\right) \notag\\
& \quad + \left(\frac{\alpha}{t+2} - 30t\epsilon^{1/2}\right)^2
        +\left(\frac{1-\alpha}{t+2} - 30t\epsilon^{1/2}\right)^2 \notag\\
& \quad   + 4\left(\frac{\alpha}{t+2} - 30t\epsilon^{1/2}\right)\left(\frac{1-\alpha}{t+2} - 30t\epsilon^{1/2}\right)
        -6(t+2)\epsilon \notag\\
& \ge \frac{t+1}{2(t+2)}+\frac{2\alpha(1-\alpha)}{(t+2)^2} -30t(t+3)\epsilon^{1/2}+450 t^2 \left(t^2+7 t+12\right)\epsilon-6(t+2)\epsilon\notag\\
& \ge \frac{t+1}{2(t+2)} - 100t^4\epsilon^{1/2}. \notag
\end{align}
On the other hand, we have
\begin{align}
|\partial\mathcal{H}|/n^2
& \le |\partial\mathcal{G}|/n^2
 = p_{\partial\mathcal{G}}(x_1,\ldots,x_{t+4}) \notag\\
& \le \binom{t}{2}\left(\frac{1}{t+2} +30t\epsilon^{1/2}\right)^2
        + t\left(\frac{1}{t+2} + 30t\epsilon^{1/2}\right)\left(\frac{2}{t+2} + 4\times 30t\epsilon^{1/2}\right) \notag\\
& \quad + \left(\frac{\alpha}{t+2} + 30t\epsilon^{1/2}\right)^2
        +\left(\frac{1-\alpha}{t+2} + 30t\epsilon^{1/2}\right)^2 \notag\\
& \quad   + 4\left(\frac{\alpha}{t+2} + 30t\epsilon^{1/2}\right)\left(\frac{1-\alpha}{t+2} + 30t\epsilon^{1/2}\right) \notag\\
& \le \frac{t+1}{2(t+2)}+\frac{2\alpha(1-\alpha)}{(t+2)^2}+30 t (t+3) \epsilon^{1/2}+450 t^2 \left(t^2+7 t+12\right)\epsilon \notag\\
& \le \frac{t+1}{2(t+2)}+\frac{1}{2(t+2)^2} + 5000t^4\epsilon^{1/2}
 = \frac{t^2+3 t+3}{2 (t+2)^2}+ 5000t^4\epsilon^{1/2}. \notag
\end{align}
\end{proof}

Now we are ready to prove Theorem~\ref{THM:feasibe-region}.

\begin{proof}[Proof of Theorem~\ref{THM:feasibe-region}]
First, we prove that $\mathrm{proj}\Omega(\mathcal{F}_{t+2}) = \left[0,\frac{t+3}{t+4}\right]$.
By a result in~\cite{LM1}, this is equivalent to show that $\frac{t+3}{t+4} \in \mathrm{proj}\Omega(\mathcal{F}_{t+2})$ and $\mathrm{proj}\Omega(\mathcal{F}_{t+2}) \subset \left[0,\frac{t+3}{t+4}\right]$.
Let $\mathcal{H}$ be an $n$-vertex $\mathcal{F}_{t+2}$-free $3$-graph.
By Lemma~\ref{LEMMA:F-free-shadow-Km-free}, the graph $\partial\mathcal{H}$ is $K_{t+5}$-free.
Thus, by Tur\'{a}n's theorem, we have $\rho(\partial\mathcal{H}) \le \frac{t+3}{t+4}$.
Therefore, we have $\mathrm{proj}\Omega(\mathcal{F}_{t+2})\subset \left[0,\frac{t+3}{t+4}\right]$.
On the other hand, since the balanced blowup of $\Gamma_{t+2}$ on $n$ vertices has edge density $\frac{t+3}{t+4}$ as $n\to \infty$,
we have $\frac{t+3}{t+4} \in \mathrm{proj}\Omega(\mathcal{F}_{t+2})$.
Therefore, $\mathrm{proj}\Omega(\mathcal{F}_{t+2}) = \left[0,\frac{t+3}{t+4}\right]$.

Let $I_t = \left[\frac{t+1}{t+4},\frac{t^2+3 t+3}{(t+2)^2}\right]$.
Next, we show that $I_t \times \left\{\frac{t (t+1)}{(t+2)^2}\right\}\subset \Omega(\mathcal{F}_{t+2})$.
This is done by constructing for every $x\in I_t$ a sequence of $\mathcal{F}_{t+2}$-free $3$-graphs whose shadow densities approach $x$, and whose edge densities approach $\frac{t (t+1)}{(t+2)^2}$.
Let $\alpha \in [0,1/2]$ be a real number.
Let $x_1 = \cdots= x_t = \frac{1}{t+2}$, $x_{t+1} = x_{t+4} = \frac{\alpha}{t+2}$, and
$x_{t+2} = x_{t+3} = \frac{1-\alpha}{t+2}$.
For every $n\in \mathbb{N}$ let
$\mathcal{H}_n(\alpha) = \Gamma_{t+2}[V_1,\ldots,V_{t+4}]$ be a blowup of $\Gamma_{t+2}$
such that $|V_i| = \lfloor x_i n \rfloor$ for $i\in [t+4]$.
It is easy to see that
\begin{align}
\lim_{n\to\infty}\rho\left(\partial\mathcal{H}_n(\alpha)\right)
= 2 \cdot p_{K_{t+4}}(x_1,\ldots,x_{t+4})
= \frac{t^2+3 t+2+4\alpha(1-\alpha)}{(t+2)^2}, \notag
\end{align}
and
\begin{align}
\lim_{n\to\infty}\rho\left(\mathcal{H}_n(\alpha)\right)
= 6\cdot p_{\Gamma_{t+2}}(x_1,\ldots,x_{t+4})
= \frac{t (t+1)}{(t+2)^2}. \notag
\end{align}
Letting $\alpha$ vary from $0$ to $1/2$,
the function $\frac{t^2+3 t+2+4\alpha(1-\alpha)}{2 (t+2)^2}$ grows from $\frac{t+1}{t+4}$ to $\frac{t^2+3 t+3}{(t+2)^2}$. Therefore,
$I_t \times \left\{\frac{t (t+1)}{6 (t+2)^2}\right\}\subset \Omega(\mathcal{F}_{t+2})$.

Finally, we show that the feasible region function $g(\mathcal{F}_{t+2})$ attains its maximum only on the interval $I_t$.
Suppose that $\left(\mathcal{H}_n\right)_{n=1}^{\infty}$ is a sequence of $\mathcal{F}_{t+2}$-free $3$-graphs
with $\lim_{n\to \infty}\rho(\partial\mathcal{H}_n) = x$ and $\lim_{n\to \infty}\rho(\mathcal{H}_n) = \frac{t(t+1)}{(t+2)^2}$.
To keep our notations simple, let us assume that $v(\mathcal{H}_n) = n$ for all $n\ge 1$.
Let $\epsilon>0$ be a sufficiently small constant.
Then there exists $N_0$ such that $\rho(\partial\mathcal{H}_n) = x\pm \epsilon$ and
$\rho(\mathcal{H}_n) \ge \frac{t(t+1)}{(t+2)^2} -\epsilon$ for all $n\ge N_0$.
It follows from Corollary~\ref{CORO:stability} that there exists a set $Z_{n}\subset V(\mathcal{H}_n)$
of size $\epsilon^{1/2} n$
such that the $3$-graph $\mathcal{H}'_{n}$ obtained from $\mathcal{H}_n$ by removing all edges that have
nonempty intersection with $Z_n$ is $\Gamma_{t+2}$-colorable with minimum degree at least $\delta(\mathcal{H}) \ge \left( \frac{t(t+1)}{2(t+2)^2} - 3\epsilon^{1/2}\right)n^2$.
Fixing a $\Gamma_{t+2}$-coloring $V(\mathcal{H}_n') = V_{n,1} \cup \cdots \cup V_{n,t+4}$ of $\mathcal{H}_n'$
and let $\mathcal{G}_n = \Gamma_{t+2}[V_{n,1} \cup \cdots \cup V_{n,t+4}]$ be the blowup of $\Gamma_{t+2}$.

Since $\rho(\mathcal{H}_n') \ge \rho(\mathcal{H}_n)-|Z_{\epsilon}|n^2/\binom{n}{3}\ge  \rho(\mathcal{H}_n)-6\epsilon^{1/2} \ge \frac{t(t+1)}{(t+2)^2}-7\epsilon^{1/2}$,
it follows from Lemma~\ref{LEMMA:feasible-region-shadow} that
$$\frac{t+1}{t+2}-2800t^4\epsilon^{1/4} \le
\rho(\partial\mathcal{H}_n') \le \rho(\partial\mathcal{G}_n) \le \frac{t^2+3 t+3}{(t+2)^2}+140000t^4\epsilon^{1/4}.$$
By Claim~\ref{CLAIM:link-empty-inside},
for every $v\in Z_n$ we have $L_{\mathcal{H}_n}(v)\cap V_{n,i} = \emptyset$ \footnote{Note that in the proof of Claim~\ref{CLAIM:link-empty-inside}
we do not require $d_{\mathcal{H}_n}(v)$ to be large.} for all $i\in [t+4]$.
Therefore, $\rho(\partial\mathcal{H}_n) \le \rho(\partial\mathcal{G}_n) + |Z_n|n/\binom{n}{2} \le \rho(\partial\mathcal{G}_n) +3\epsilon^{1/2}$.
So,
\begin{align}
\frac{t+1}{t+2}-2800t^4\epsilon^{1/4}
\le \rho(\partial\mathcal{H}'_n)
\le \rho(\partial\mathcal{H}_n)
\le \rho(\partial\mathcal{G}_n) +3\epsilon^{1/2}
\le \frac{t^2+3 t+3}{(t+2)^2}+140001t^4\epsilon^{1/4}.\notag
\end{align}
Consequently,
\begin{align}
\frac{t+1}{t+2}-2800t^4\epsilon^{1/4}-\epsilon
\le x
\le \frac{t^2+3 t+3}{(t+2)^2}+140001t^4\epsilon^{1/4} +\epsilon.\notag
\end{align}
Letting $\epsilon \to 0$, we have
\begin{align}
\frac{t+1}{t+2}
\le x
\le \frac{t^2+3 t+3}{(t+2)^2}.\notag
\end{align}
This completes the proof of Theorem~\ref{THM:feasibe-region}.
\end{proof}
%%%%%%%%%%%%%%%%%%%%%%%%%%%%%%%%%%%%%%%%%%%%%%%%%
%%%%%%%%%%%%%%%%%%%%%%%%%%%%%%%%%%%%%%%%%%%%%%%%%
%%%%%%%%%%%%%%%%%%%%%%%%%%%%%%%%%%%%%%%%%%%%%%%%%
%%%%%%%%%%%%%%%%%%%%%%%%%%%%%%%%%%%%%%%%%%%%%%%%%
\section{Concluding remarks}\label{SEC:Remarks}
\noindent$\bullet$  We defined the crossed blowup and used it to construct a finite family of triple systems with infinitely many extremal constructions.
One could extended this operation in the following way.

Instead of replacing a pair of vertices (i.e. $\{v_1,v_2\}$) by four vertices (i.e. $\{v_1,v_1',v_2,v_2'\}$),
one could replace it by $2^k$ vertices for $k\ge 2$.

Let $Q_{k}$ denote the vertex set of the $k$-dimensional hypercube (each vertex is represented by a length-$k$ binary string).
For $i\in [k]$ let $Q_{k}(i,0)\subset Q_{k}$ and $Q_{k}(i,1)\subset Q_{k}$ denote the collection of vertices whose $i$-th coordinate is
$0$ and $1$, respectively.

\begin{definition}[$k$-crossed blowup]\label{DFNp:k-crosse-blowup}
Let $k\ge 2$ be an integer.
Let $\mathcal{G}$ be a $3$-graph  and $\{v_1,v_2\}\subset V(\mathcal{H})$
be a pair with $d_{\mathcal{G}}(v_1,v_2) = d \ge k$.
Assume that $N_{\mathcal{G}}(v_1,v_2)=\{u_1,  \ldots, u_d\}$.
The $k$-crossed blowup $\mathcal{G}\boxplus^{k}\{v_1,v_2\}$ of $\mathcal{G}$ on $\{v_1,v_2\}$ is obtained in the following way.
\begin{enumerate}[label=(\alph*)]
\item Remove all edges in $\mathcal{G}$ that contain $\{v_1,v_2\}$,
\item replace $\{v_1,v_2\}$ by $2^k$ new vertices $Q_{k}$ so that every vertex in $Q(1,0)$ is a clone of $v_1$ and every vertex in
        $Q(1,1)$ is a colon of $v_2$,
\item for $i\in [k-1]$ we add a set $\mathcal{E}_i$ of triples containing $u_i$ so that $L_{\mathcal{E}_i}(u_i)$ is the
        complete bipartite graph $B[Q(i,0),Q(i,1)]$ with two parts $Q(i,0),Q(i,1)$,
\item for $i \in [k,d]$ we add  a set $\mathcal{E}_i$ of triples containing $u_i$ so that $L_{\mathcal{E}_i}(u_i)$ is the
        complete bipartite graph $B[Q(k,0),Q(k,1)]$ with two parts $Q(k,0),Q(k,1)$.
\end{enumerate}
\end{definition}

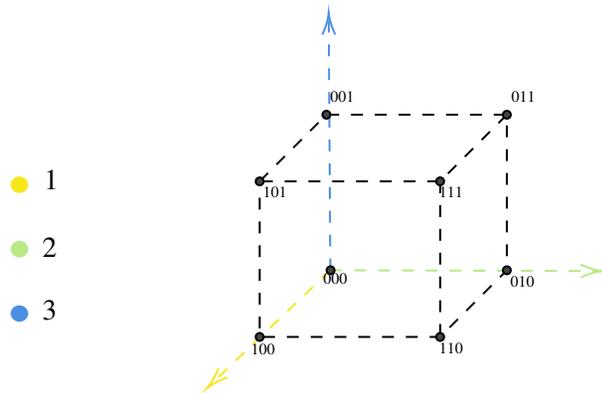
\begin{figure}[htbp]
\centering
\tikzset{every picture/.style={line width=0.75pt}} %set default line width to 0.75pt
\begin{tikzpicture}[x=0.75pt,y=0.75pt,yscale=-1,xscale=1]
%uncomment if require: \path (0,487); %set diagram left start at 0, and has height of 487
%Shape: Cube [id:dp38214847999089274]
\draw  [fill={rgb, 255:red, 255; green, 255; blue, 255 }  ,fill opacity=1 ][dash pattern={on 4.5pt off 4.5pt}] (431.08,189.69) -- (464.73,156.05) -- (555.75,156.05) -- (555.75,234.56) -- (522.1,268.21) -- (431.08,268.21) -- cycle ; \draw  [dash pattern={on 4.5pt off 4.5pt}] (555.75,156.05) -- (522.1,189.69) -- (431.08,189.69) ; \draw  [dash pattern={on 4.5pt off 4.5pt}] (522.1,189.69) -- (522.1,268.21) ;
\draw [color={rgb, 255:red, 184; green, 233; blue, 134 }  ,draw opacity=1 ] [dash pattern={on 4.5pt off 4.5pt}]  (466.83,234.56) -- (601.92,234.96) ;
\draw [shift={(603.92,234.97)}, rotate = 180.17] [color={rgb, 255:red, 184; green, 233; blue, 134 }  ,draw opacity=1 ][line width=0.75]    (10.93,-3.29) .. controls (6.95,-1.4) and (3.31,-0.3) .. (0,0) .. controls (3.31,0.3) and (6.95,1.4) .. (10.93,3.29)   ;
%Straight Lines [id:da34206005732713995]
\draw [color={rgb, 255:red, 248; green, 231; blue, 28 }  ,draw opacity=1 ] [dash pattern={on 4.5pt off 4.5pt}]  (466.72,234.56) -- (406.12,292.51) ;
\draw [shift={(404.68,293.89)}, rotate = 316.28] [color={rgb, 255:red, 248; green, 231; blue, 28 }  ,draw opacity=1 ][line width=0.75]    (10.93,-3.29) .. controls (6.95,-1.4) and (3.31,-0.3) .. (0,0) .. controls (3.31,0.3) and (6.95,1.4) .. (10.93,3.29)   ;
%Straight Lines [id:da41755992151015686]
\draw [color={rgb, 255:red, 74; green, 144; blue, 226 }  ,draw opacity=1 ] [dash pattern={on 4.5pt off 4.5pt}]  (466.72,234.56) -- (465.9,106.67) ;
\draw [shift={(465.89,104.67)}, rotate = 89.63] [color={rgb, 255:red, 74; green, 144; blue, 226 }  ,draw opacity=1 ][line width=0.75]    (10.93,-3.29) .. controls (6.95,-1.4) and (3.31,-0.3) .. (0,0) .. controls (3.31,0.3) and (6.95,1.4) .. (10.93,3.29)   ;
%Shape: Ellipse [id:dp02642189486263047]
\draw  [color={rgb, 255:red, 248; green, 231; blue, 28 }  ,draw opacity=1 ][fill={rgb, 255:red, 248; green, 231; blue, 28 }  ,fill opacity=1 ] (305.99,191.43) .. controls (305.99,189.41) and (307.75,187.77) .. (309.91,187.77) .. controls (312.07,187.77) and (313.82,189.41) .. (313.82,191.43) .. controls (313.82,193.45) and (312.07,195.08) .. (309.91,195.08) .. controls (307.75,195.08) and (305.99,193.45) .. (305.99,191.43) -- cycle ;
%Shape: Ellipse [id:dp7832698641967366]
\draw  [color={rgb, 255:red, 184; green, 233; blue, 134 }  ,draw opacity=1 ][fill={rgb, 255:red, 184; green, 233; blue, 134 }  ,fill opacity=1 ] (306,223.91) .. controls (306,221.89) and (307.75,220.25) .. (309.91,220.25) .. controls (312.07,220.25) and (313.82,221.89) .. (313.82,223.91) .. controls (313.82,225.93) and (312.07,227.57) .. (309.91,227.57) .. controls (307.75,227.57) and (306,225.93) .. (306,223.91) -- cycle ;
%Shape: Ellipse [id:dp7166358852053158]
\draw  [color={rgb, 255:red, 74; green, 144; blue, 226 }  ,draw opacity=1 ][fill={rgb, 255:red, 74; green, 144; blue, 226 }  ,fill opacity=1 ] (306,256.39) .. controls (306,254.37) and (307.75,252.74) .. (309.91,252.74) .. controls (312.07,252.74) and (313.82,254.37) .. (313.82,256.39) .. controls (313.82,258.41) and (312.07,260.05) .. (309.91,260.05) .. controls (307.75,260.05) and (306,258.41) .. (306,256.39) -- cycle ;
%
%\draw  [fill=sqsqsq,fill opacity=0.25]
% (306,256.39) to [out = 200, in = 160] (306,223.91) to [out = 200, in = 160] (305.99,191.43)
%  to [out = 170, in = 190] (306,256.39);
% Text Node
\draw (320.92,182.12) node [anchor=north west][inner sep=0.75pt]   [align=left] {1};
% Text Node
\draw (320.05,216.23) node [anchor=north west][inner sep=0.75pt]   [align=left] {2};
% Text Node
\draw (320.05,247.9) node [anchor=north west][inner sep=0.75pt]   [align=left] {3};
% Text Node
%\draw  [fill={rgb, 255:red, 255; green, 255; blue, 255 }  ,fill opacity=1 ][dash pattern={on 4.5pt off 4.5pt}] (431.08,189.69) -- (464.73,156.05) -- (555.75,156.05) -- (555.75,234.56) -- (522.1,268.21) -- (431.08,268.21) -- cycle ; \draw  [dash pattern={on 4.5pt off 4.5pt}] (555.75,156.05) -- (522.1,189.69) -- (431.08,189.69) ; \draw  [dash pattern={on 4.5pt off 4.5pt}] (522.1,189.69) -- (522.1,268.21) ;
\draw [fill=uuuuuu] (431.08,189.69) circle (1.5pt);
\draw [fill=uuuuuu] (464.73,156.05) circle (1.5pt);
\draw [fill=uuuuuu] (555.75,156.05) circle (1.5pt);
\draw [fill=uuuuuu] (555.75,234.56) circle (1.5pt);
\draw [fill=uuuuuu] (522.1,268.21) circle (1.5pt);
\draw [fill=uuuuuu] (431.08,268.21) circle (1.5pt);
\draw [fill=uuuuuu] (466.83,234.56)  circle (1.5pt);
\draw [fill=uuuuuu] (522.1,189.69) circle (1.5pt);
\draw (461.8,234.47) node [anchor=north west][inner sep=0.75pt]  [font=\tiny] [align=left] {000};
% Text Node
\draw (556.04,235.77) node [anchor=north west][inner sep=0.75pt]  [font=\tiny] [align=left] {010};
% Text Node
\draw (425.32,270.2) node [anchor=north west][inner sep=0.75pt]  [font=\tiny] [align=left] {100};
% Text Node
\draw (520.11,269.42) node [anchor=north west][inner sep=0.75pt]  [font=\tiny] [align=left] {110};
% Text Node
\draw (465.28,142.7) node [anchor=north west][inner sep=0.75pt]  [font=\tiny] [align=left] {001};
% Text Node
\draw (556.57,142.7) node [anchor=north west][inner sep=0.75pt]  [font=\tiny] [align=left] {011};
% Text Node
\draw (431.37,190.91) node [anchor=north west][inner sep=0.75pt]  [font=\tiny] [align=left] {101};
% Text Node
\draw (520.18,190.91) node [anchor=north west][inner sep=0.75pt]  [font=\tiny] [align=left] {111};
\end{tikzpicture}
\caption{$\{145,245,345\}\boxplus^3\{4,5\}$, in which $L(1)$ is the complete bipartite graph with
         parts $\{000,001,010,011\}$ and $\{100,101,110,111\}$ (these two squares are perpendicular to the yellow axis).
         Similarly to $L(2)$ and $L(3)$.}
\end{figure}

Notice that the crossed blowup we defined in Section~\ref{SEC:Introduction} is a $2$-crossed blowup.
Using a similar argument, one could extend
Proposition~\ref{PROP:nonminimal-corss-blowup} to $k$-crossed blowups for all $k\ge 2$.

There is also a natural extension to $r$-graphs for every $r\ge 4$. We omit the definition here.

\medskip

\noindent$\bullet$ There are two ways to extend Theorem~\ref{THM:main-sec1.1} to $r$-graphs for every $r\ge 4$.
One is to use $\Gamma_{t+2}$ to construct an $r$-graph in the following way:
take $r-3$ new vertices $\{u_1,\ldots, u_{r-3}\}$ and let
\begin{align}
\Gamma_{t+2}^r = \left\{\{u_1,\ldots, u_{r-3}\}\cup E\colon E\in \Gamma_{t+2}\right\}. \notag
\end{align}
Following the argument as in~\cite{LMR3} one could extend Theorem~\ref{THM:main-sec1.1} to $r$-graphs.

Another way is to consider the crossed blowup of the $r$-uniform complete graph $K_{t+2}^{r}$.
Following the argument as in the present paper one could extend Theorem~\ref{THM:main-sec1.1} to $r$-graphs as well.

\medskip

\noindent$\bullet$ For $t=1$ one could use the fact that the set $Z(\Gamma_1)$ is a two-dimensional simplex
\begin{align}
\left\{\alpha \vec{x}_1 +\beta \vec{x}_2 +(1-\alpha-\beta) \vec{x}_3
        \colon \alpha\ge 0, \beta\ge 0, \alpha + \beta\le 1\right\}, \notag
\end{align}
where $\vec{x}_1 = \left(\frac{1}{3},0,\frac{1}{3},0,0,\frac{1}{3}\right)$,
$\vec{x}_2 = \left(0,\frac{1}{3},\frac{1}{3},0,0,\frac{1}{3}\right)$, and
$\vec{x}_3 = \left(0,\frac{1}{3},0,\frac{1}{3},\frac{1}{3},0\right)$,
to improve Theorem~\ref{THM:main-sec1.1}~(b) to show that
the number of maximum $\mathcal{F}_1$-free $3$-graphs on $n$ vertices is $\Theta(n^2)$.
%In general, for every integer $r\ge 3$ and $k\ge 2$
%by consider the following $r$-graph
%\begin{align}
%\left\{\{1,\ldots, r-1,r\},\{1,\ldots, r-1,r+1\}, \ldots, \{1,\cdots, r-1,r+k-1\}\right\} \boxplus \{1,2\}, \notag
%\end{align}
%one could prove that there exists a
%finite family $\mathcal{F}_{r,k}$ such that for every integer $n\in \mathbb{N}^{+}$ that satisfies $r\mid n$
%the number of maximum $\mathcal{F}_{r,k}$-free $r$-graphs on $n$ vertices is $\Theta(n^k)$.

\medskip

\noindent$\bullet$ After the submission of this paper, Pikhurko and the third author~\cite{LP22} showed a different method to construct finite families with an infinite stability number. Moreover, they showed that  there exists a finite family $\mathcal{F}$ of $3$-graphs whose feasible region function attains its maximum on a Cantor-type set.
Balogh, Clemen, and Luo~\cite{BCL22} constructed a $1$-stable $3$-graph with exponentially (in terms of the number of vertices) many nonisomorphic extremal constructions and showed that the feasible region function attains its maximum on a nontrivial interval.
%%%%%%%%%%%%%%%%%%%%%%%%%%%%%%%%%%%%%%%%%%%%%%%%%
%%% AUTHOR: optional appendix here
%\appendix %% you may comment this out if no Appendix
% \section*{Appendix}
% \section{Improving the constants}
% Material is placed here as needed.

%%% AUTHOR: optional acknowledgments here
\section*{Acknowledgments} %%  you may comment this out if no Ackno
We would like to thank a referee for the careful reading of the manuscript and for providing many very helpful comments.

%%% AUTHOR:
%%% Bibliography goes here. Note that the arXiv cannot process bibtex
%%% or biber bibliographies.  Example of acceptable bibliograpy format:
\bibliographystyle{amsplain}

\begin{thebibliography}{99}
\bibitem{AES74}
B.~Andr\'{a}sfai, P.~Erd\H{o}s, and V.~T. S\'{o}s.
\newblock On the connection between chromatic number, maximal clique and
  minimal degree of a graph.
\newblock {\em Discrete Math.}, 8:205--218, 1974.

\bibitem{BT11}
R.~Baber and J.~Talbot.
\newblock Hypergraphs do jump.
\newblock {\em Combin. Probab. Comput.}, 20(2):161--171, 2011.

\bibitem{BCL22}
J.~Balogh, F.~C. Clemen, and H.~Luo.
\newblock Non-degenerate hypergraphs with exponentially many extremal
  constructions.
\newblock {\em arXiv preprint arXiv:2208.00652}, 2022.

\bibitem{BR83}
W.~G. Brown.
\newblock On an open problem of {P}aul {T}ur{\'a}n concerning 3-graphs.
\newblock In {\em Studies in pure mathematics}, pages 91--93. Springer, 1983.

\bibitem{CL99}
F.~{Chung} and L.~{Lu}.
\newblock {An upper bound for the {T}ur\'an number \(t_3(n,4)\)}.
\newblock {\em {J. Comb. Theory, Ser. A}}, 87(2):381--389, 1999.

\bibitem{DC88}
D.~de~Caen.
\newblock On upper bounds for {$3$}-graphs without tetrahedra.
\newblock volume~62, pages 193--202. 1988.
\newblock Seventeenth Manitoba Conference on Numerical Mathematics and
  Computing (Winnipeg, MB, 1987).

\bibitem{ES66}
P.~Erd\H{o}s and M.~Simonovits.
\newblock A limit theorem in graph theory.
\newblock {\em Studia Sci. Math. Hungar.}, 1:51--57, 1966.

\bibitem{ES46}
P.~Erd\H{o}s and A.~H. Stone.
\newblock On the structure of linear graphs.
\newblock {\em Bull. Amer. Math. Soc.}, 52:1087--1091, 1946.

\bibitem{FO88}
D.~G. {Fon-Der-Flaass}.
\newblock {On a method of construction of \((3,4)\)-graphs}.
\newblock {\em {Mat. Zametki}}, 44(4):546--550, 1988.

\bibitem{FF84}
P.~{Frankl} and Z.~{F\"uredi}.
\newblock {An exact result for 3-graphs}.
\newblock {\em {Discrete Math.}}, 50:323--328, 1984.

\bibitem{FF89}
P.~Frankl and Z.~F\"{u}redi.
\newblock Extremal problems whose solutions are the blowups of the small
  {W}itt-designs.
\newblock {\em J. Combin. Theory Ser. A}, 52(1):129--147, 1989.

\bibitem{FPRT07}
P.~Frankl, Y.~Peng, V.~R\"{o}dl, and J.~Talbot.
\newblock A note on the jumping constant conjecture of {E}rd{\H o}s.
\newblock {\em J. Combin. Theory Ser. B}, 97(2):204--216, 2007.

\bibitem{FR84}
P.~Frankl and V.~R\"{o}dl.
\newblock Hypergraphs do not jump.
\newblock {\em Combinatorica}, 4(2-3):149--159, 1984.

\bibitem{KA68}
G.~Katona.
\newblock A theorem of finite sets.
\newblock In {\em Theory of graphs ({P}roc. {C}olloq., {T}ihany, 1966)}, pages
  187--207, 1968.

\bibitem{KE11}
P.~Keevash.
\newblock Hypergraph {T}ur\'{a}n problems.
\newblock In {\em Surveys in combinatorics 2011}, volume 392 of {\em London
  Math. Soc. Lecture Note Ser.}, pages 83--139. Cambridge Univ. Press,
  Cambridge, 2011.

\bibitem{Ke18}
P.~Keevash.
\newblock Counting {S}teiner triple systems.
\newblock In {\em European {C}ongress of {M}athematics}, pages 459--481. Eur.
  Math. Soc., Z\"{u}rich, 2018.

\bibitem{KO82}
A.~V. {Kostochka}.
\newblock {A class of constructions for {T}ur{\'a}n's (3,4)-problem.}
\newblock {\em {Combinatorica}}, 2:187--192, 1982.

\bibitem{KR63}
J.~B. Kruskal.
\newblock The number of simplices in a complex.
\newblock In {\em Mathematical optimization techniques}, pages 251--278. Univ.
  of California Press, Berkeley, Calif., 1963.

\bibitem{Liu20a}
X.~Liu.
\newblock Cancellative hypergraphs and {S}teiner triple systems.
\newblock {\em J. Comb. Theory, Ser. B}.
\newblock accepted.

\bibitem{LIU19}
X.~Liu.
\newblock New short proofs to some stability theorems.
\newblock {\em European J. Combin.}, 96:Paper No. 103350, 8, 2021.

\bibitem{LM1}
X.~Liu and D.~Mubayi.
\newblock The feasible region of hypergraphs.
\newblock {\em J. Comb. Theory, Ser. B}, 148:23--59, 2021.

\bibitem{LM22}
X.~Liu and D.~Mubayi.
\newblock A hypergraph {T}ur{\'a}n problem with no stability.
\newblock {\em Combinatorica}, pages 1--30, 2022.

\bibitem{LMR1}
X.~Liu, D.~Mubayi, and C.~Reiher.
\newblock Hypergraphs with many extremal configurations.
\newblock {\em Israel. J. Math.}
\newblock to appear.

\bibitem{LMR3}
X.~Liu, D.~Mubayi, and C.~Reiher.
\newblock Hypergraphs with many extremal configurations {II}.
\newblock In preparation.

\bibitem{LMR2}
X.~Liu, D.~Mubayi, and C.~Reiher.
\newblock A unified approach to hypergraph stability.
\newblock {\em J. Combin. Theory Ser. B}, 158:36--62, 2023.

\bibitem{LP22}
X.~Liu and O.~Pikhurko.
\newblock Finite hypergraph families with rich extremal {T}ur{\' a}n
  constructions via mixing patterns.
\newblock {\em arXiv preprint arXiv:2212.08636}, 2022.

\bibitem{MU07}
D.~Mubayi.
\newblock Structure and stability of triangle-free set systems.
\newblock {\em Trans. Amer. Math. Soc.}, 359(1):275--291, 2007.

\bibitem{PI14}
O.~Pikhurko.
\newblock On possible {T}ur\'{a}n densities.
\newblock {\em Israel J. Math.}, 201(1):415--454, 2014.

\bibitem{RA10}
A.~{Razborov}.
\newblock {On 3-hypergraphs with forbidden 4-vertex configurations}.
\newblock {\em {SIAM J. Discrete Math.}}, 24(3):946--963, 2010.

\bibitem{AS95}
A.~Sidorenko.
\newblock What we know and what we do not know about {T}ur\'{a}n numbers.
\newblock {\em Graphs Combin.}, 11(2):179--199, 1995.

\bibitem{SI68}
M.~Simonovits.
\newblock A method for solving extremal problems in graph theory, stability
  problems.
\newblock In {\em Theory of {G}raphs ({P}roc. {C}olloq., {T}ihany, 1966)},
  pages 279--319. Academic Press, New York, 1968.

\bibitem{TU41}
P.~Tur{\'a}n.
\newblock On an extermal problem in graph theory.
\newblock {\em Mat. Fiz. Lapok}, 48:436--452, 1941.

\bibitem{Zy}
A.~A. Zykov.
\newblock On some properties of linear complexes.
\newblock {\em Mat. Sbornik N.S.}, 24(66):163--188, 1949.

\end{thebibliography}

%% AUTHOR: You can generate such a bibliography from a .bib file by
%% running pdflatex/bibtex/pdflatex/pdflatex and then pasting the .bbl file
%% between \begin{thebibliography} and \end{bibliography}

%%% AUTHOR: Include a short description of each author following the
%%% structure below. Use the same short tags used previously.
%%% Use \imageat{} and \imagedot{} instead of "@" and "." in
%%% email addresses-this replaces the symbols with graphics to avoid
%%% e-mail address harvesting from the .pdf file
\begin{dajauthors}
\begin{authorinfo}[JH]
  Jianfeng Hou\\
  Center for Discrete Mathematics\\
  Fuzhou University\\
  Fujian, 350003, China\\
  jfhou\imageat{}fzu\imagedot{}edu\imagedot{}cn \\
  %\url{https://www.cs.elte.hu/erdos}
\end{authorinfo}
\begin{authorinfo}[HL]
    Heng Li\\
    School of Mathematics\\
    Shandong University\\
    Shandong, 250100, China\\
    heng\imagedot{}li\imageat{}sdu\imagedot{}edu\imagedot{}cn
  %\url{http://www.csc.kth.se/~johanh}
\end{authorinfo}
\begin{authorinfo}[XL]
    Xizhi Liu\\
    Mathematics Institute and DIMAP\\
    University of Warwick\\
    Coventry, CV4 7AL, UK\\
    xizhi\imagedot{}liu\imagedot{}ac\imageat{}gmail\imagedot{}com
  %\url{http://www.cs.elte.hu/~lovasz}
\end{authorinfo}
\begin{authorinfo}[DM]
    Dhruv Mubayi\\
    Department of Mathematics,
    Statistics, and Computer Science\\
    University of Illinois\\
    Chicago, IL, 60607 USA\\
    mubayi\imageat{}uic\imagedot{}edu\\
    \url{http://homepages.math.uic.edu/~mubayi/}
\end{authorinfo}
\begin{authorinfo}[YZ]
    Yixiao Zhang\\
    Center for Discrete Mathematics\\
    Fuzhou University\\
    Fujian, 350003, China\\
    fzuzyx\imageat{}gmail\imagedot{}com\\
\end{authorinfo}
\end{dajauthors}

\end{document}